\documentclass[a4paper,reqno]{amsart}
\usepackage{amssymb, amsmath, amscd}
\newtheorem{theorem}{Theorem}[section]
\newtheorem{lemma}[theorem]{Lemma}
\newtheorem{remark}[theorem]{Remark}
\newtheorem{definition}[theorem]{Definition}
\newtheorem{corollary}[theorem]{Corollary}

\makeatletter
 \@addtoreset{equation}{section}
\makeatother
\def\ad{\operatorname{ad}}
\begin{document}
\title[Homogeneous affine gradient Ricci solitons]{Homogeneous affine surfaces: affine
Killing vector fields and Gradient Ricci solitons}
\author{M. Brozos-V\'{a}zquez \, E. Garc\'{i}a-R\'{i}o, \, and P. Gilkey}
\address{MBV: Departmento de Matem\'aticas, Escola Polit\'ecnica Superior, Universidade da Coru\~na, Spain}
\email{miguel.brozos.vazquez@udc.gal}
\address{EGR: Faculty of Mathematics,
University of Santiago de Compostela,
15782 Santiago de Compostela, Spain}
\email{eduardo.garcia.rio@usc.es}
\address{PG: Mathematics Department, \; University of Oregon, \;\;
 Eugene \; OR 97403, \; USA}
\email{gilkey@uoregon.edu}
\thanks{ Supported by projects EM2014/009, GRC2013-045 and MTM2013-41335-P with FEDER funds (Spain).}
\keywords{ Homogeneous affine surface, affine Killing vector field, affine gradient Ricci soliton, affine gradient Yamabe soliton, Riemannian extension.}
\subjclass[2010]{53C21}
\begin{abstract}
The homogeneous affine surfaces have been classified
by Opozda. They may be grouped into 3 families, which are not disjoint. The connections which arise as the
 Levi-Civita connection of a surface with a metric of constant Gauss curvature form one family; there are, however,
 two other families. For a surface in one of these other two families, we examine the 
 Lie algebra of affine Killing vector fields and we give
 a complete classification of the homogeneous affine gradient Ricci solitons.
 The rank of the Ricci tensor plays a central role in our analysis.
 \end{abstract}
\maketitle
\section{Introduction}

\subsection{Homogeneity}
The notion of homogeneity is central in geometry. In order to make precise the level of homogeneity 
one usually refers to the underlying structure. In pseudo-Riemannian geometry, local homogeneity 
means that for any two points there is a local isometry sending one point to the other. 
If an additional structure (K\"ahler, contact, etc.) is considered on the manifold, 
then one further assumes that this structure is preserved by the local isometries. 
In the affine setting, homogeneity means that for any two points
there is an affine transformation sending one point into the other.
There is an intermediate level of homogeneity which was explored in \cite{G-RGN,KVOp}.
A pseudo-Riemannian manifold may be locally affine homogeneous but not locally homogeneous, i.e., for any two points there exists a (not necessarily isometric) transformation sending one point to the other which preserves the Levi-Civita connection.

Homogeneous affine surfaces were studied from a local point of view by several authors. 
A complete description was first given in \cite{KVOp2} for the special case when the Ricci tensor is skew-symmetric. 
The general situation was later addressed in \cite{Op04},
where Opozda obtained the local form of the connection of any locally homogeneous affine surface. 
More recently, Opozda's result was generalized in \cite{AMK08} to the more general case of connections with torsion.
The above classification results have been extensively used both in the affine and the pseudo-Riemannian setting, 
where one uses the Riemannian extension to relate affine and pseudo-Riemannian geometry.

\subsection{Notational conventions}
An affine manifold is a pair $\mathcal{M}=(M,\nabla)$ 
where $\nabla$ is a torsion free connection on the tangent bundle
of a smooth manifold $M$ of dimension $m$. Let $\vec x=(x^1,\dots,x^m)$ be a system
of local coordinates on $M$. We adopt the {\it Einstein convention}
and sum over repeated indices to expand:
$$\nabla_{\partial_{x^i}}\partial_{x^j}=\Gamma_{ij}{}^k\partial_{x^k}$$
in terms of the {\it Christoffel symbols} \,$\Gamma=\Gamma^\nabla:=(\Gamma_{ij}{}^k)$; the condition that
$\nabla$ is torsion free is then equivalent to the symmetry $\Gamma_{ij}{}^k=\Gamma_{ji}{}^k$.
The curvature operator $R$,
 the Ricci tensor $\rho$, and the symmetric Ricci tensor $\rho^s$ are given, respectively, by setting 
\begin{eqnarray*}
&&R(\xi_1,\xi_2):=\nabla_{\xi_1}\nabla_{\xi_2}-\nabla_{\xi_2}\nabla_{\xi_1}-\nabla_{[\xi_1,\xi_2]},\\
&&\rho(\xi_1,\xi_2):=\operatorname{Tr}\{\xi_3\rightarrow R(\xi_3,\xi_1)\xi_2\},\text{ and }
 \rho^s(\xi_1,\xi_2):=\textstyle\frac12(\rho(\xi_1,\xi_2)+\rho(\xi_2,\xi_1))\,.
\end{eqnarray*}

\subsection{Locally homogeneous affine surfaces}
Let $\mathcal{M}=(M,\nabla)$ be an affine surface. We say that $\mathcal{M}$ is {\it locally homogeneous} if
given any two points of $M$, there is the germ of a diffeomorphism $\Phi$ taking one point to another with $\Phi^*\nabla=\nabla$.
One has
the following classification result due of Opozda \cite{Op04}:
\begin{theorem}\label{T1.1}
Let $\mathcal{M}=(M,\nabla)$ be a locally homogeneous affine surface. Then at least one of the following
three possibilities holds which describe the local geometry:
\begin{enumerate}
\item There exist local coordinates $(x^1,x^2)$ so that
$\Gamma_{ij}{}^k=\Gamma_{ji}{}^k$ is constant.
\item There exist local coordinates $(x^1,x^2)$ so that
$\Gamma_{ij}{}^k=(x^1)^{-1}C_{ij}{}^k$ where $C_{ij}{}^k=C_{ji}{}^k$ is constant.
\item $\nabla$ is the Levi-Civita connection of a metric of constant sectional curvature.
\end{enumerate}\end{theorem}

\begin{definition}\rm
An affine surface $\mathcal{M}$ is said to be {\it Type~$\mathcal{A}$} (resp. {\it Type~$\mathcal{B}$
or Type~$\mathcal{C}$})
if $\mathcal{M}$ is locally homogeneous, if $\mathcal{M}$
is not flat, and if Assertion~1 (resp. Assertion~2 or Assertion~3) of Theorem~\ref{T1.1} holds.
Let
 \begin{eqnarray*}
&&\mathcal{F}^{\mathcal{A}}:=\{\mathcal{M}=(\mathbb{R}^2,\nabla):\Gamma^\nabla
\text{ constant and }\nabla\text{ not flat}\},\\
&&\mathcal{F}^{\mathcal{B}}:=\{\mathcal{M}=(\mathbb{R}^+\times\mathbb{R},\nabla):
\Gamma^\nabla=(x^1)^{-1}C\text{ for }C\text{ constant and }\nabla\text{ not flat}\}\,.
\end{eqnarray*}
Let $\mathcal{M}\in\mathcal{F}^{\mathcal{A}}$. We will show in Lemma~\ref{L2.2} that
$\rho$ is symmetric. Since $\mathcal{M}$ is not flat, 
$\operatorname{Rank}\{\rho\}\ne0$. We therefore may decompose
$\mathcal{F}^{\mathcal{A}}=\mathcal{F}^{\mathcal{A}}_1\cup\mathcal{F}^{\mathcal{A}}_2$ where
$$\mathcal{F}^{\mathcal{A}}_\nu:=\{\mathcal{M}\in\mathcal{F}^{\mathcal{A}}:\operatorname{Rank}\{\rho\}=\nu\}\,.$$
The affine surfaces in the family $\mathcal{F}^{\mathcal{A}}$ (resp. $\mathcal{F}^{\mathcal{B}}$) form natural
models for the Type~$\mathcal{A}$ (resp. Type~$\mathcal{B}$) surfaces and we will often work in this context.
\end{definition}

Surfaces of Type~$\mathcal{A}$ and Type~$\mathcal{B}$ can have quite different geometric properties. 
The Ricci tensor of any Type~$\mathcal{A}$ surface is symmetric; this can fail for a Type~$\mathcal{B}$ surface.
Thus the geometry of a Type~$\mathcal{B}$ surface is not as rigid as that of a Type~$\mathcal{A}$ surface;
this is closely related to the existence of non-flat affine Osserman structures \cite{D,GKVV99}.
Any Type~$\mathcal{A}$ surface is projectively flat; this can fail for a Type~$\mathcal{B}$ surface.
The local geometry of any Type~$\mathcal{A}$ surface can be realized on a compact torus 
\cite{G-SG,Opozda}; this can also fail for a Type~$\mathcal{B}$ geometry.

\begin{remark}\label{R1.3}\rm
If $M=\mathbb{R}^2$ and if the Christoffel symbols $\Gamma$ of $\nabla$ are constant, then
$\mathbb{R}^2$ acts transitively on $M$ by translations and this group action preserves $\nabla$. Thus every element
of $\mathcal{F}^{\mathcal{A}}$ is affine homogeneous. If $M=\mathbb{R}^+\times\mathbb{R}$
and if the Christoffel symbols of $\nabla$ have the form $\Gamma=(x^1)^{-1}C$ for $C$ constant, then the
$ax+b$ group acts transitively on $M$ by $(a,b):(x^1,x^2)\rightarrow(ax^1,ax^2+b)$ for $ a>0$ 
and this
group action preserves $\nabla$.
Thus every element of $\mathcal{F}^{\mathcal{B}}$ is affine homogeneous. These two structure groups (which
up to isomorphism are the only two simply connected 2-dimensional Lie groups) will play an important
role in our analysis.
\end{remark}

\begin{remark}\rm The three possibilities of Theorem~\ref{T1.1} are not exclusive
as we shall see presently. In Theorem~\ref{T3.11}, we will identify the local 
geometries which are both Type~$\mathcal{A}$
and Type~$\mathcal{B}$ and also the local geometries
which are both type Type~$\mathcal{B}$ and
Type~$\mathcal{C}$. There are no surfaces which are both Type~$\mathcal{A}$
and Type~$\mathcal{C}$.
\end{remark}

\subsection{Outline of the paper} 
In Section~\ref{S2}, we use the action of the natural structure
groups on the families $\mathcal{F}^{\mathcal{A}}$ and $\mathcal{F}^{\mathcal{B}}$ to
partially normalize the Christoffel symbols. Let $\mathcal{M}$ be a Type-$\mathcal{A}$
surface with $\operatorname{Rank}(\rho)=1$.
In Lemma~\ref{L2.5}, 
we will define $\alpha(\mathcal{M})$ and show it is an affine invariant in this setting.
Subsequently, in Theorem~\ref{T3.8}, we will show that $\alpha$
identifies the moduli space of such surfaces with $\rho\ge0$ with $[0,\infty)$ and
with $\rho\le0$ with $(-\infty,0]$.
 
Similarly, we may
partially normalize the Christoffel symbols for Type~$\mathcal{B}$ geometries in Lemma~\ref{L2.8}.
Lemma~\ref{L2.10} provides a complete
characterization of the elements of $\mathcal{F}^{\mathcal{B}}$
where $\rho$ is symmetric, recurrent, and of rank 1, and where $\nabla\rho$ is symmetric.
This will play a central role in our identification of the affine surfaces which are both
Type~$\mathcal{A}$ and Type~$\mathcal{B}$. 

Section~3 is devoted to the study of the Lie algebra $\mathfrak{K}(\mathcal{M})$ of affine Killing vector fields.
Let $\mathcal{M}$ be an affine surface. In Lemma~\ref{L3.1}, we will show if $\mathcal{M}$ is
homogeneous, then
$2\leq\dim\{\mathfrak{K}(\mathcal{M})\}\leq 6$;
the extremal case where $\dim\{\mathfrak{K}(\mathcal{M})\}=6$ occurs only if $\mathcal{M}$ is flat.
We shall exclude the flat setting from consideration henceforth.

Let $\mathcal{M}\in\mathcal{F}^{\mathcal{A}}$. 
To simplify the notation, we set
$\partial_1:=\partial_{x^1}$ and $\partial_2:=\partial_{x^2}\,.$
Let $\mathfrak{K}_0^{\mathcal{A}}:=\operatorname{Span}\{\partial_1,\partial_2\}$ be
the Lie algebra of the translation group $\mathbb{R}^2$. By Remark~\ref{R1.3},
$\mathfrak{K}_0^{\mathcal{A}}\subset\mathfrak{K}(\mathcal{M})$.
In Theorem~\ref{T3.4}, we show $\dim\{\mathfrak{K}(M)\}>2$
if and only if $\rho$ has rank 1 and that $\dim\{\mathfrak{K}(M)\}=4$ in this setting.
In Theorem~\ref{T3.8}, we exhibit invariants which completely detect
the local isomorphism class of a Type~$\mathcal{A}$ affine surface with $\operatorname{Rank}\{\rho\}=1$,
we also determine which Type~$\mathcal{A}$ surfaces are also of Type~$\mathcal{B}$, and we
give the abstract structure of the (local) Lie algebras involved using the classification of 
Patera et. al \cite{PSWZ76}; 
representatives of these classes are given in Lemma~\ref{L3.6}.

Let $\mathcal{M}\in\mathcal{F}^{\mathcal{B}}$.
We will show that $\dim\{\mathfrak{K}(\mathcal{M})\}\in\{2,3,4\}$ in Section~\ref{S3.2};
$\mathcal{M}$ is also of Type~$\mathcal{A}$ if and only if $\dim\{\mathfrak{K}(\mathcal{M})\}=4$. This
characterizes the local geometries which are the intersection of Type~$\mathcal{A}$ and
Type~$\mathcal{B}$. The geometries which are of both Type~$\mathcal{B}$ and of
Type~$\mathcal{C}$ form a proper subset of those surfaces where
$\dim\{\mathfrak{K}(\mathcal{M})\}=3$.	

The {\it Hessian} $H_f^\nabla$ of $f\in C^\infty(M)$ is the symmetric $2$-tensor
$$
H_f^\nabla:=\nabla(df)=f_{;ij}dx^i\circ dx^j\,.
$$
If $g$ is a pseudo-Riemannian metric on
$M$, let $H_f^g:=H_f^{\nabla^g}$ be the Hessian which is defined by the Levi-Civita connection $\nabla^g$
and let $\rho^g$ be the associated Ricci tensor.
\begin{definition}\label{D1.5}\rm
Let $M$ be a smooth manifold, let $\nabla$ be a torsion free connection on $M$, let $g$
be a pseudo-Riemannian metric on $M$,
let $\tau$ be the scalar curvature of $g$, and let
$f\in C^\infty(M)$ be a smooth function on $M$. We say that
 \begin{enumerate}
 \item $(M,\nabla,f)$ is an {\it affine gradient Yamabe soliton} if
 $H_f^\nabla =0$. 
 Let $\mathcal{Y}(\mathcal{M})$ be the space of functions on $M$ so that $(M,\nabla, f)$ is an affine gradient Yamabe soliton;
 $\mathcal{Y}(\mathcal{M})=\ker(H^\nabla)$.
 \item $(M,\nabla,f)$ is an {\it affine gradient Ricci soliton} if
 $H_f^\nabla+\rho_s=0$. Let $\mathfrak{A}(\mathcal{M})$ be the space of functions on $M$ so that $(M,\nabla, f)$ is an affine gradient Ricci soliton.
If
$\mathfrak{A}(\mathcal{M})$ is non-empty, then 
$\mathfrak{A}(\mathcal{M})=f_0+\mathcal{Y}(\mathcal{M})$ for any $f_0\in\mathfrak{A}(\mathcal{M})$.
 \item $(M,g,f)$ is a \emph{gradient Yamabe soliton} if there exists $\lambda\in\mathbb{R}$ so
$H_f^{g} =(\tau-\lambda)g$.
\item $(M,g,f)$ is a {\it gradient Ricci soliton} if there exists $\lambda\in\mathbb{R}$
so $H_f^{g}+\rho^g=\lambda g$.
If $\lambda=0$, then the soliton is said to be {\it steady}.
\item A soliton is said to be {\it trivial} if the potential function $f$ is constant.
 \end{enumerate} \end{definition}

There is a close connection between affine geometry and neutral signature geometry. 
Let $\mathcal{M}=(M,\nabla)$ be an affine manifold and let
$(x^1,\dots,x^m)$ be local coordinates on $M$. Express $\omega=y_idx^i$ to introduce the dual
fiber coordinates $(y_1,\dots,y_m)$ on the cotangent bundle $T^*M$. Let $\phi=\phi_{ij}$ be a symmetric
$2$-tensor on $M$. The {\it deformed Riemannian extension} $g_{\nabla,\phi}$ is the metric of
neutral signature $(m,m)$ on $T^*M$ given by
$$g_{\nabla,\phi}=dx^i\otimes dy_i+dy_i\otimes dx^i+(\phi_{ij}-2y_k\Gamma_{ij}{}^k)dx^i\otimes dx^j\,.$$
It is invariantly defined, i.e. it is independent of the particular coordinate system chosen.
The following result \cite{BCGV16,BG14} provided our initial motivation for examining affine
gradient Ricci solitons in the 2-dimensional setting; we state the results for gradient Ricci solitons and
Yamabe solitons in parallel to simplify the exposition:

\begin{theorem}
Let $(N,g,F)$ be a non-trivial self-dual gradient Ricci (resp. Yamabe) soliton of neutral signature
$(2,2)$.
\begin{enumerate}
\item If $\|dF\|\ne0$ at a point $P\in N$, then $(N,g)$ is
 locally isometric to a warped product
$I\times_\psi N_1$ where $N_1$ is a 3-dimensional pseudo-Riemannian manifold of constant sectional curvature
(resp. scalar curvature).
\item If $\|dF\|=0$ on $N$, then $(N,g)$ is locally isometric to the cotangent bundle $T^*M$
of an affine surface $(M,\nabla)$ equipped with the deformed Riemannian extension $g_{\nabla,\phi}$. Furthermore, the potential function of the soliton is of the form $F=f\circ\pi$, for some function $f$ on $M$ so that $(M,\nabla,\frac{1}{2}f)$ is an affine gradient Ricci (resp. Yamabe) soliton.
\end{enumerate}
\end{theorem}

In Section~\ref{S4}, we examine affine gradient Ricci solitons if $\mathcal{M}$ is Type~$\mathcal{A}$ and/or
Type~$\mathcal{B}$.
Let $f$ be the potential function of an affine gradient Ricci soliton and let $X\in\mathfrak{K}(M)$.
In Lemma \ref{L4.1}, we show that $X(f)$ is the potential function of an affine gradient Yamabe soliton. 
Thus affine gradient Ricci solitons and Yamabe solitons are closely linked concepts.
Using this fact, we analyze the existence of affine gradient Ricci and Yamabe solitons on homogeneous affine surfaces. 
Not unexpectedly, Type~$\mathcal{A}$ and Type~$\mathcal{B}$ affine connections behave differently.

Let $\mathcal{M}\in\mathcal{F}^{\mathcal{A}}$. In Theorem~\ref{T4.3}, we show
$\mathcal{M}$ is a gradient Ricci soliton if and only if $\operatorname{Rank}\{\rho\}=1$ or, equivalently
in view of the results of Section~\ref{S3}, $\dim\{\mathfrak{K}(\mathcal{M})\}>2$.
There are elements of $\mathcal{F}^{\mathcal{B}}$ which have skew-symmetric Ricci tensor or, equivalently,
so that $(T^*M,g_{\nabla,\phi})$ is Ricci flat and hence are trivial Ricci solitons. 
In Theorem~\ref{T4.9} and Theorem~\ref{T4.10}, we give elements of $\mathcal{F}^{\mathcal{B}}$
which are  non-trivial affine gradient Ricci solitons and which are not of Type~$\mathcal{A}$. 
Finally, Theorem~\ref{T4.12} gives a complete classification, up to affine equivalence, of homogeneous affine gradient Ricci solitons. 
The associated deformed Riemannian extensions then form
a large family of non-conformally flat self-dual gradient Ricci and Yamabe solitons.

\subsection{ Local versus global geometry}
There is always a question of the local versus the global geometry of an object in differential geometry. 
Let $\mathcal{M}$ be a locally homogeneous affine surface. The
dimension of the space of germs of affine Killing vector fields (resp. affine gradient Ricci solitons) 
is constant on $\mathcal{M}$. Let $X_i$ (resp. $f_i$) be affine Killing vector fields (resp.
define affine gradient Ricci solitons) which are defined on a connected open subset $\mathcal{O}$ of $\mathcal{M}$.
If there is a non-empty subset $\mathcal{O}_1\subset\mathcal{M}$ with $X_1=X_2$ (resp. $f_1=f_2$) on
$\mathcal{O}_1$, then $X_1=X_2$ (resp. $f_1=f_2$) on $\mathcal{O}$. Thus questions of passing
from the local to the global for either affine Killing vector fields or affine gradient Ricci solitons involve the
holonomy action of the fundamental group; there is no obstruction if $\mathcal{M}$ is assumed
simply connected. We shall not belabor the point and ignore the question of passing from local to global
henceforth.

\subsection{Moduli spaces} The moduli space $\mathcal{Z}_{\mathcal{A}}$ of isomorphism classes of
germs of Type~$\mathcal{A}$ structures is 2-dimensional 
\cite{KV03}. The strata of $\mathcal{Z}_{\mathcal{A}}$ where $\operatorname{Rank}\{\rho\}=1$ is handled 
by Theorem~\ref{T3.8}; it contains two components isomorphic to $[0,\infty)$ and $(-\infty,0]$.
 In a
subsequent paper  \cite{BGG16}, we will discuss the strata of $\mathcal{Z}_{\mathcal{A}}$ where 
$\rho$ is non-degenerate of signature $(p,q)$; these may be identified with closed simply connected
subsets of $\mathbb{R}^2$.  Let $\mathcal{Z}_{\mathcal{B}}$ be the
moduli space of Type-$\mathcal{B}$ structures. The strata of $\mathcal{Z}_{\mathcal{B}}$
where $\dim\{\mathfrak{K}(\mathcal{M})\}=4$ is handled by Theorem~\ref{T3.8}
since all these surfaces are also of Type~$\mathcal{A}$. We will also show in \cite{BGG16} that the
strata of $\mathcal{Z}_{\mathcal{B}}$
where $2\le\dim\{\mathfrak{K}(\mathcal{M})\}\le3$ is a real analytic manifold
with non-trivial topology.

\section{Homogeneous affine surfaces}\label{S2}

In this section, we use the structure groups described above acting on the families $\mathcal{F}^{\mathcal{A}}$ and
$\mathcal{F}^{\mathcal{B}}$ to perform certain normalizations.
Recall that a $k$-tensor $T$ is said to be {\it symmetric} if
$T(v_1,\dots,v_k)=T(v_{\sigma(1)},\dots,v_{\sigma(k)})$ for every
 permutation $\sigma$, and that
$T$ is said to be {\it recurrent} if $\nabla T=\omega\otimes T$ for
some $1$-form $\omega$.

\subsection{Rank 2 symmetric Ricci tensor}
We will show presently that $\rho=\rho^s$ if $\mathcal{M}$ is Type~$\mathcal{A}$.
However, $\rho$ need not be symmetric if $\mathcal{M}$ is Type~$\mathcal{B}$.
\begin{lemma}
\ \begin{enumerate}
\item Let $\mathcal{M}\in\mathcal{F}^{\mathcal{A}}$ satisfy $\operatorname{Rank}\{\rho\}=2$.
Then $\rho$ determines a flat pseudo-Riemannian metric on $\mathcal{M}$.
\item Let $\mathcal{M}\in\mathcal{F}^{\mathcal{B}}$ satisfy $\operatorname{Rank}\{\rho^s\}=2$.
\begin{enumerate}
\item $\rho^s$ defines a 
pseudo-Riemannian metric of constant Gauss curvature~$\kappa$. 
\item $\kappa=0$ if and only if $\rho_{22}=0$. 
\item If $\nabla$ is projectively flat, then the metric defined by $\rho^s$ has $\kappa\neq 0$. 
\end{enumerate}\end{enumerate}
\end{lemma}

\begin{proof} If $\mathcal{M}$ is Type~$\mathcal{A}$, then $\rho$ is symmetric.
If $\operatorname{Rank}\{\rho^s\}=2$, then $\rho^s$ defines a pseudo-Riemannian metric.
If $\mathcal{M}\in\mathcal{F}^{\mathcal{A}}$, then $\rho^s$ is invariant under the translation group
$(a,b):(x^1,x^2)\rightarrow(x^1+a,x^2+b)$. This group acts transitively on $\mathbb{R}^2$ and hence the
components of $\rho$ are constant. This implies $\rho$ is flat. If $M\in\mathcal{F}^{\mathcal{B}}$,
then $\rho^s$ is invariant under the $ax+b$ group $(a,b):(x^1,x^2)\rightarrow(ax^1,ax^2+b)$. This
non-Abelian 2-dimensional Lie group acts transitively on $\mathbb{R}^+\times\mathbb{R}$ and
hence $\rho^s$ has constant Gauss curvature $\kappa$. This proves Assertion~1 and Assertion~2a. The proof of the remaining assertions follows as in \cite{CGV10}.
\end{proof}

\subsection{Type~$\mathcal{A}$ homogeneous affine surfaces}
We omit the proof of the following result as it is a direct computation (see also \cite{CGV10}):
\begin{lemma}\label{L2.2}
Let $\mathcal{M}\in\mathcal{F}^{\mathcal{A}}$. Then
\begin{enumerate}
\item The Ricci tensor of $\mathcal{M}$ is symmetric ($\rho_{12}=\rho_{21}$) and one has:
\smallbreak $\rho_{11}=(\Gamma_{11}{}^1-\Gamma_{12}{}^2) \Gamma_{12}{}^2
+\Gamma_{11}{}^2 (\Gamma_{22}{}^2-\Gamma_{12}{}^1)$,
\smallbreak $\rho_{12}=\Gamma_{12}{}^1 \Gamma_{12}{}^2-\Gamma_{11}{}^2 \Gamma_{22}{}^1$,
\smallbreak $\rho_{22}=-(\Gamma_{12}{}^1)^2+\Gamma_{22}{}^2 \Gamma_{12}{}^1
+(\Gamma_{11}{}^1-\Gamma_{12}{}^2) \Gamma_{22}{}^1$.
\medbreak\item $\nabla\rho$ is symmetric ($\rho_{12;1}=\rho_{21;1}=\rho_{11;2}$, $\rho_{12;2}=\rho_{21;2}=\rho_{22;1}$) and one has:
\smallbreak
$\rho_{11;1}=2\{-(\Gamma_{11}{}^1)^2 \Gamma_{12}{}^2+\Gamma_{11}{}^1(\Gamma_{11}{}^2
 (\Gamma_{12}{}^1-\Gamma_{22}{}^2)+(\Gamma_{12}{}^2)^2)$
 \smallbreak\qquad\qquad
 $+\Gamma_{11}{}^2(\Gamma_{11}{}^2\Gamma_{22}{}^1-\Gamma_{12}{}^1\Gamma_{12}{}^2)\}$,
 \smallbreak
 $\rho_{12;1}=2 \left(\Gamma_{11}{}^2 \left((\Gamma_{12}{}^1)^2-\Gamma_{12}{}^1 \Gamma_{22}{}^2+\Gamma_{12}{}^2\Gamma_{22}{}^1\right)-\Gamma_{11}{}^1\Gamma_{12}{}^1\Gamma_{12}{}^2\right)$,
 \smallbreak
 $\rho_{12;2}=2 \left(\Gamma_{12}{}^2 (-\Gamma_{11}{}^1\Gamma_{22}{}^1-\Gamma_{12}{}^1\Gamma_{22}{}^2+\Gamma_{12}{}^2\Gamma_{22}{}^1)+\Gamma_{11}{}^2\Gamma_{12}{}^1\Gamma_{22}{}^1\right)$,
 \smallbreak
 $\rho_{22;2}= 2 \{\Gamma_{22}{}^1(\Gamma_{22}{}^2 (\Gamma_{12}{}^2-\Gamma_{11}{}^1)+\Gamma_{11}{}^2\Gamma_{22}{}^1)+(\Gamma_{12}{}^1)^2 \Gamma_{22}{}^2$,
 \smallbreak\qquad\qquad
 $ -\Gamma_{12}{}^1(\Gamma_{12}{}^2\Gamma_{22}{}^1+(\Gamma_{22}{}^2)^2)\}$.
\end{enumerate}\end{lemma}

If $\mathcal{M}\in\mathcal{F}^{\mathcal{A}}_1$, then
we can always make a linear change of coordinates to replace $\mathcal{M}$ by
an isomorphic surface where $\rho=\rho_{22}dx^2\otimes dx^2$, i.e.
$\rho_{11}=\rho_{12}=\rho_{21}=0$. The following is a useful technical result:

\begin{lemma}\label{L2.3}
Let $\mathcal{M}\in\mathcal{F}^{\mathcal{A}}$.
The following conditions are equivalent:
\begin{enumerate}
\item $\rho(\mathcal{M})=\rho_{22}dx^2\otimes dx^2$.
\item $\Gamma_{11}{}^2=0$ and $\Gamma_{12}{}^2=0$.
\item $\rho=\{
\Gamma_{12}{}^1(\Gamma_{22}{}^2-\Gamma_{12}{}^1)+\Gamma_{11}{}^1\Gamma_{22}{}^1\}
dx^2\otimes dx^2$.
\end{enumerate}\end{lemma}

\begin{proof} We assume Assertion~1 holds and apply Lemma~\ref{L2.2}.
\begin{enumerate}
\item Suppose first $\Gamma_{22}{}^1$ is non-zero. 
By rescaling, we may suppose $\Gamma_{22}{}^1=1$.
To ensure $\rho_{12}=0$, we set 
 $\Gamma_{11}{}^2=\Gamma_{12}{}^1\Gamma_{12}{}^2$ and obtain
$\rho_{11}=\Gamma_{12}{}^2\rho_{22}$. 
Since $\rho_{22}\ne0$, $\Gamma_{12}{}^2=0$ and
hence $\Gamma_{11}{}^2=0$ as well.
\item Suppose next that $\Gamma_{22}{}^1=0$.
Setting $\rho_{12}=0$ yields $\Gamma_{12}{}^1\Gamma_{12}{}^2=0$. 
Since $\rho_{22}=\Gamma_{12}{}^1(\Gamma_{22}{}^2-\Gamma_{12}{}^1)$, $\Gamma_{12}{}^1\ne0$. 
Thus $\Gamma_{12}{}^2=0$. We now compute that 
$\rho_{11}=\Gamma_{11}{}^2(\Gamma_{22}{}^2-\Gamma_{12}{}^1)$
and $\rho_{22}=\Gamma_{12}{}^1(\Gamma_{22}{}^2-\Gamma_{12}{}^1)$. Consequently,
 $\Gamma_{11}{}^2=0$.
 \end{enumerate}
Thus in either eventuality
we obtain $\Gamma_{11}{}^2=0$ and $\Gamma_{12}{}^2=0$ so
Assertion~1 implies Assertion~2. The proof that Assertion~2 implies
Assertion~3 is a direct computation. The proof that Assertion~3 implies Assertion~1
is immediate. 
\end{proof}

\begin{definition}\label{D2.4}
\rm
Let $\mathcal{M}\in\mathcal{F}^{\mathcal{A}}_1$.
Choose $X\in T_PM$ so $\rho(X,X)\ne0$ and set 
$$
\alpha_X(\mathcal{M}):=\nabla\rho(X,X;X)^2\cdot\rho(X,X)^{-3}
\text{ and }\epsilon_X(\mathcal{M}):=\operatorname{Sign}\{\rho(X,X)\}=\pm1\,.
$$
\end{definition}

\begin{lemma}\label{L2.5}
Let $\mathcal{M}\in\mathcal{F}^{\mathcal{A}}_1$.
\begin{enumerate}\item There exists a $1$-form $\omega$ so $\nabla^k\rho=(k+1)!\omega^k\otimes\rho$ for any $k$. 
\item $\rho$ is recurrent.
\item $\operatorname{Ker}\{\rho\}$ is a parallel distribution.
\item $\alpha_X(\mathcal{M})$ and $\epsilon_X(\mathcal{M})$ are independent of the choice of $X$
and determine invariants we will denote by $\alpha(\mathcal{M})$ and $\epsilon(\mathcal{M})$.\end{enumerate}
\end{lemma}

\begin{proof} Choose coordinates on $\mathbb{R}^2$ so that $\rho=\rho_{22}dx^2\otimes dx^2$.
Assertion~1 then follows from Lemma~\ref{L2.2} and Lemma~\ref{L2.3}, and Assertion~2 then follows
from Assertion~1. We have $\ker(\rho)=\operatorname{Span}\{\partial_1\}$.
Lemma~\ref{L2.3} then shows $\operatorname{Span}\{\partial_1\}$ is a parallel distribution
as desired. Use Assertion~1 to express $\rho=c_0\omega\otimes\omega$ and
$\nabla\rho=c_1\omega\otimes\omega\otimes\omega$. One verifies 
$\alpha_X(\mathcal{M})=(\omega(X)^3c_1)^2(\omega(X)^2c_0)^{-3}$ and
$\epsilon_X(\mathcal{M})=\operatorname{Sign}\{\omega(X)^2c_0\}$
are independent of $X$.
\end{proof}

\begin{remark}\rm
Clearly $\alpha(\mathcal{M})=0$ if and only if $\mathcal{M}$ is symmetric. 
Furthermore, if $\alpha(\mathcal{M})\ne0$, then 
$\epsilon(\mathcal{M})=\operatorname{Sign}(\alpha(\mathcal{M}))$ so
$\epsilon$ is determined by $\alpha$ except in the symmetric setting. We will show
subsequently in Theorem~\ref{T3.8} that $\alpha$ and $\epsilon$ determine
the local isomorphism class of a Type~$\mathcal{A}$ surface with 
$\operatorname{Rank}\{\rho\}=1$.
\end{remark}

\subsection{Type~$\mathcal{B}$ homogeneous affine surfaces}
We begin by extending Lemma~\ref{L2.2} to this setting. We omit
the proof of the following result as it is a direct computation~(see \cite{CGV10}).

\begin{lemma}\label{L2.7}
Let $\mathcal{M}\in\mathcal{F}^{\mathcal{B}}$ so $\Gamma=(x^1)^{-1}C$.
\begin{enumerate}
\item $\rho_{11}= 
(x^1)^{-2}\{{C_{12}{}^2 (C_{11}{}^1}-C_{12}{}^2+1)+C_{11}{}^2 (C_{22}{}^2-C_{12}{}^1)\}$.
\item $\rho_{12}=(x^1)^{-2}\{-C_{11}{}^2 C_{22}{}^1+C_{12}{}^1
 C_{12}{}^2+C_{22}{}^2\}$.
\item $\rho_{21}= 
(x^1)^{-2}\{-C_{11}{}^2 C_{22}{}^1+C_{12}{}^1
 C_{12}{}^2-C_{12}{}^1\}$.
\item $\rho_{22}=(x^1)^{-2}\{C_{11}{}^1 C_{22}{}^1-(C_{12}{}^1)^2+C_{12}{}^1 C_{22}{}^2-C_{12}{}^2
 C_{22}{}^1-C_{22}{}^1\}$.
\end{enumerate}
\end{lemma}

We use the coordinate transformation $(x^1,x^2)\rightarrow (x^1,\varepsilon x^1+x^2)$ to partially
normalize the Christoffel symbols. The following result will be used in the 
proof of Lemma~\ref{L3.15} subsequently.
\begin{lemma}\label{L2.8}
Let $\mathcal{M}\in\mathcal{F}^{\mathcal{B}}$.
\begin{enumerate}
\item If $C_{22}{}^1\ne0$, then by replacing $x^2$ by $x^2-\varepsilon x^1$,
we may assume that $C_{12}{}^1=0$.
\item If $C_{12}{}^1=0$, $C_{22}{}^1=0$, $C_{22}{}^2=0$,
and $C_{11}{}^1-2C_{12}{}^2\ne0$,
then by replacing $x^2$ by $x^2-\varepsilon x^1$,
we may assume that $C_{11}{}^2=0$ without changing the other Christoffel symbols.
\end{enumerate}
\end{lemma}

\begin{proof}
Let $(u^1,u^2):=(x^1,\varepsilon x^1+x^2)$. We then have:
\begin{eqnarray*}
&&du^1=dx^1,\quad du^2=\varepsilon dx^1+dx^2,\quad 
 \partial_1^u=\partial_1^x-\varepsilon\partial_2^x,\quad
 \partial_2^u=\partial_2^x,\\
&&\nabla_{\partial_1^u}\partial_2^u
=\nabla_{\partial_1^x-\varepsilon\partial_2^x}\partial_2^x
=({}^x\Gamma_{12}{}^1-{}^x\Gamma_{22}{}^1\varepsilon)\partial_1^x+\star\partial_2^x\\
&&\quad=({}^x\Gamma_{12}{}^1-{}^x\Gamma_{22}{}^1\varepsilon)\partial_1^u
+\star\partial_2^u,\\
&&{}^uC_{12}{}^1={}^xC_{12}{}^1-\varepsilon\cdot{}^xC_{22}{}^1\,.
\end{eqnarray*}
We prove Assertion~1 by taking $\varepsilon={}^xC_{12}{}^1({}^xC_{22}{}^1)^{-1}$.

Assume $C_{12}{}^1=0$, $C_{22}{}^1=0$, $C_{22}{}^2=0$,
and $C_{11}{}^1-2C_{12}{}^2\ne0$. We compute
\medbreak\qquad
$\nabla_{\partial_1^u}\partial_1^u=
{}^x\Gamma_{11}{}^1\partial_1^x+{}^x\Gamma_{11}{}^2\partial_2
-2\epsilon\cdot{}^x\Gamma_{12}{}^1\partial_1^x-2\epsilon\cdot{}^x\Gamma_{12}{}^2\partial_2^x$
\smallbreak\qquad\qquad\qquad\quad
$+\epsilon^2\cdot{}^x\Gamma_{22}{}^1\partial_1^x
+\epsilon^2\cdot{}^x\Gamma_{22}{}^2\partial_2^x$
\smallbreak\qquad\qquad\qquad
$={}^x\Gamma_{11}{}^1\partial_1^x+{}^x\Gamma_{11}{}^2\partial_2^x
-2\epsilon\cdot{}^x\Gamma_{12}{}^2\partial_2^x$
\smallbreak\qquad\qquad\qquad
$={}^x\Gamma_{11}{}^1(\partial_1^x-\epsilon\partial_2^x)+
({}^x\Gamma_{11}{}^2+\epsilon\{{}^x\Gamma_{11}{}^1-2\cdot{}^x\Gamma_{12}{}^2\})\partial_2^x$,
\medbreak\qquad
$\nabla_{\partial_1^u}\partial_2^u=
{}^x\Gamma_{12}{}^1\partial_1^x+{}^x\Gamma_{12}{}^2\partial_2^x-\epsilon\cdot
{}^x\Gamma_{22}{}^1\partial_1^x-\epsilon\cdot{}^x\Gamma_{22}{}^2\partial_2^x$
\smallbreak\qquad\qquad\qquad
$={}^x\Gamma_{12}{}^2\partial_2^x$,
\medbreak\qquad
$\nabla_{\partial_2^u}\partial_2^u=
{}^x\Gamma_{22}{}^1\partial_1^x+{}^x\Gamma_{22}{}^2\partial_2^x=0$,
\medbreak\qquad
${}^u\Gamma_{11}{}^1={}^x\Gamma_{11}{}^1,\quad
{}^u\Gamma_{11}{}^2={}^x\Gamma_{11}{}^2+\epsilon({}^x\Gamma_{11}{}^1
-2\cdot{}^x\Gamma_{12}{}^2)$,\quad
${}^u\Gamma_{12}{}^1=0$,
\medbreak\qquad${}^u\Gamma_{12}{}^2={}^x\Gamma_{12}{}^2$,\quad
${}^u\Gamma_{22}{}^1=0,\qquad\qquad\qquad\qquad
\qquad\qquad\quad{}^u\Gamma_{22}{}^2=0$.
\medbreak\noindent
We set $\epsilon=-({}^x\Gamma_{11}{}^1-2{}^x\Gamma_{12}{}^2)^{-1}\cdot
{}^x\Gamma_{11}{}^2$ to establish Assertion~2.
\end{proof}

\begin{remark}\label{R2.9}\rm
We apply Lemma~\ref{L2.8} to simplify the expressions of the Ricci tensor:
\begin{enumerate}
\item If $C_{22}{}^1\ne0$ we may assume that $C_{12}{}^1=0$ and express
\smallbreak $\rho_{11}=
(x^1)^{-2}\{{C_{12}{}^2 (C_{11}{}^1}-C_{12}{}^2+1)+C_{11}{}^2 C_{22}{}^2\}$,
\smallbreak $\rho_{12}=(x^1)^{-2}\{-C_{11}{}^2 C_{22}{}^1+C_{22}{}^2\}$,\qquad
$\rho_{21}= 
(x^1)^{-2}\{-C_{11}{}^2 C_{22}{}^1\}$,
\smallbreak $ \rho_{22}=(x^1)^{-2}\{(C_{11}{}^1-C_{12}{}^2-1)C_{22}{}^1\}$.
\medbreak\item
If $C_{12}{}^1=0$, $C_{22}{}^1=0$, $C_{22}{}^2=0$,
and $C_{11}{}^1-2C_{12}{}^2\ne0$,
we may assume that $C_{11}{}^2=0$ and express:
\smallbreak $\rho_{11}=(x^1)^{-2}\{C_{12}{}^2 (C_{11}{}^1-C_{12}{}^2+1)\},\quad
\rho_{12}=0,\quad \rho_{21}= 0,\quad \rho_{22}=0$.
\end{enumerate}\end{remark}

The Ricci tensor of a Type~$\mathcal{B}$ surface is not symmetric in general. 
Indeed, it is symmetric if and only if $\Gamma_{22}{}^2=-\Gamma_{12}{}^1$. 
Consequently, this family of surfaces is not projectively flat in general, in contrast to 
Type~$\mathcal{A}$ surfaces. We decompose the Ricci tensor into its symmetric and its 
alternating parts in the form
 $\rho=\rho^s+\rho^a$.
If $\operatorname{Rank}\{\rho^s\}=0$, then $\rho=\rho^a$
(i.e., $\rho_{ij}=-\rho_{ji}$ for all $1\le i,j\le 2$).
That case will be examined in detail in Lemma~\ref{L4.6}; we postpone the analysis until Section~\ref{S4}
since it will be crucial to our discussion of affine gradient Ricci solitons on Type~$\mathcal{B}$
surfaces and it is appropriate to introduce the necessary notation then.
We now examine the case that $\rho$ is symmetric, $\nabla\rho$ is symmetric, and
 $\operatorname{Rank}\{\rho^s\}=1$. 

\begin{lemma}\label{L2.10}
Let $\mathcal{M}\in\mathcal{F}^{\mathcal{B}}$.
The following conditions are equivalent:
 \begin{enumerate}
 \item We have that $C_{12}{}^1=0$, $C_{22}{}^1=0$, and $C_{22}{}^2=0$.
 \item We have that
 \begin{enumerate}
\item $\rho=(x^1)^{-2}(1+C_{11}{}^1-C_{12}{}^2)C_{12}{}^2dx^1\otimes dx^1$.
\item $\nabla\rho=(x^1)^{-3}(-2(1+C_{11}{}^1)(1+C_{11}{}^1-C_{12}{}^2)C_{12}{}^2)dx^1\otimes dx^1\otimes dx^1$.
\item $\alpha(\mathcal{M}):=\rho_{11;1}^2/\rho_{11}^3
=4(1+C_{11}{}^1)^2/\{(1+C_{11}{}^1-C_{12}{}^2)C_{12}{}^2\}$.
 \end{enumerate}
 \item $\rho$ is symmetric, recurrent, and of rank $1$ and $\nabla\rho$ is symmetric.
 \end{enumerate}
\end{lemma}
\begin{proof}
A direct computation shows that Assertion~1 implies Assertion~2. It
is immediate that
Assertion~2 implies Assertion~3. Assume Assertion~3 holds
so $\rho$ is symmetric. This implies
$C_{12}{}^1+C_{22}{}^2=0$. Since $\rho$ is symmetric and has rank $1$,
we may express 
$$\rho=(x^1)^{-2}\varepsilon(a_1dx^1+a_2dx^2)\otimes(a_1dx^1+a_2dx^2)$$
where $\varepsilon=\pm1$.
There are two possibilities.
\smallbreak\noindent{\bf Case 1:} $a_2\ne0$. By making the linear change of coordinates
$\tilde x^2=a_1x^1+a_2x^2$, we obtain a new Type~$\mathcal{B}$ surface with 
$\rho=(x^1)^{-2}\varepsilon dx^2\otimes dx^2$. Since $\nabla\rho$ is recurrent, we have
$\nabla\rho=(x^1)^{-2}\omega\otimes dx^2\otimes dx^2$. Since $\nabla\rho$
is symmetric, we have $\omega=c\otimes dx^2$ for some constant $c$.
Thus the only non-zero component of $\nabla\rho$ is $\rho_{22;2}$.
We compute
\medbreak
$\nabla_{\partial_1}\rho=-(x^1)^{-3}\varepsilon\{
2(1+C_{12}{}^2)dx^2\otimes dx^2+C_{11}{}^2dx^1\otimes dx^2
+C_{11}{}^2dx^2\otimes dx^1\}$,
\medbreak
$\nabla_{\partial_2}\rho=-(x^1)^{-3}\varepsilon\{
2C_{22}{}^2dx^2\otimes dx^2+C_{12}{}^2dx^1\otimes dx^2
+C_{12}{}^2dx^2\otimes dx^1\}$.
\medbreak\noindent
This implies $C_{12}{}^2=-1$ and $C_{12}{}^2=0$ which is not possible.
\medbreak\noindent{\bf Case 2: $a_2=0$.} We have 
$\rho=(x^1)^{-2}\varrho dx^1\otimes dx^1$ for $\varrho\ne0$. Then
\begin{eqnarray*}
&&\nabla_{\partial_1}\rho=\star dx^1\otimes dx^1
-\varrho(x^1)^{-3}C_{12}{}^1(dx^2\otimes dx^1+dx^1\otimes dx^2)\},\\
&&\nabla_{\partial_2}\rho=-\varrho(x^1)^{-3}\{2C_{12}{}^1dx^1\otimes dx^1+
C_{22}{}^1(dx^1\otimes dx^2+dx^2\otimes dx^1)\}\,.
\end{eqnarray*}
Since $\rho$ is recurrent, $C_{12}{}^1=0$ and $C_{22}{}^1=0$.
As $\rho$ is symmetric, $C_{22}{}^2=-C_{12}{}^1=0$.
Thus we obtain the relations of Assertion~1. 
\end{proof}

If $\rho=\rho_{11} dx^1\otimes dx^1$, then 
$\ker(\rho)=\operatorname{Span}\{\partial_2\}$. 
If, moreover, $C_{12}{}^1=0$, $C_{22}{}^1=0$, and $C_{22}{}^2=0$, then $\ker(\rho)$ is a parallel distribution which is totally geodesic.

\section{Affine Killing vector fields}\label{S3}

If $X$ is a smooth vector field on $M$, let $\Phi_t^X$ be the local flow
defined by $X$. We refer to Kobayashi-Nomizu \cite[Chapter VI]{KN63} for the proof of the following
result.

\begin{lemma}\label{L3.1} Let $\mathcal{M}=(M,\nabla)$ be an affine surface. 
\begin{enumerate}
\item The following 3 conditions are equivalent and if any is satisfied, $X$ is said to be
an {\rm affine Killing vector field}:
 \begin{enumerate}
\item $(\Phi_t^X)_*\circ\nabla=\nabla\circ(\Phi_t^X)_*$ on the appropriate domain.
\item The Lie derivative $\mathcal{L}_X(\nabla)$ of $\nabla$ vanishes.
\item $[X,\nabla_YZ]-\nabla_Y[X,Z]-\nabla_{[X,Y]}Z=0$ for all $Y,Z\in C^\infty(TM)$.
\end{enumerate}
\item  Let $\mathfrak{K}(\mathcal{M})$ be the set of affine Killing vector fields. The
Lie bracket gives $\mathfrak{K}(\mathcal{M})$ the structure of a real Lie algebra. Furthermore,
if $X\in\mathfrak{K}(\mathcal{M})$, if $X(P)=0$, and if $\nabla X(P)=0$, then $X\equiv0$.
\item If $\mathcal{M}$ is an affine surface, then $\dim\{\mathfrak{K}(\mathcal{M})\}\le 6$; equality holds
if and only if $\mathcal{M}$ is flat.
\end{enumerate}\end{lemma}

The relations of Assertion 1c will be called {\it Killing equations}.
The following result characterizes Types~$\mathcal{A}$ and $\mathcal{B}$ 
homogeneous affine surfaces by means of the Lie Algebra structure of their affine Killing vector fields. It is a restatement of Arias-Marco and Kowalski \cite[Lemma~1 and Lemma~2]{AMK08}.
\begin{lemma}\label{L3.2}
Let $\mathcal{M}=(M,\nabla)$ be a homogeneous affine surface.
\begin{enumerate}
\item $\mathcal{M}$ is of Type~$\mathcal{A}$ if and only if there exists an Abelian sub-algebra $\mathfrak{g}$
of $\mathfrak{K}(\mathcal{M})$ of rank 2, i.e. there exist $X,Y\in\mathfrak{K}(\mathcal{M})$ which
are linearly independent at some point $P$ of $M$ so that $[X,Y]=0$.
\item $\mathcal{M}$ is of Type~$\mathcal{B}$ if and only if there exists a non-Abelian sub-algebra $\mathfrak{g}$
of $\mathfrak{K}(\mathcal{M})$ of rank 2, i.e there exist $X,Y\in\mathfrak{K}(\mathcal{M})$ which
are linearly independent at some point $P$ of $M$ so that $[X,Y]=Y$.
\end{enumerate}\end{lemma}

 If $P$ is a point of a locally homogeneous surface $\mathcal{M}$,
let $\mathfrak{K}_P(\mathcal{M})$ be the Lie algebra of germs of affine Killing
vector fields at $P$. If $\mathcal{M}$ is both Type~$\mathcal{A}$ and Type~$\mathcal{B}$,
then there is a 2-dimensional Abelian Lie sub-algebra of $\mathfrak{K}_P(\mathcal{M})$ and
there is also a 2-dimensional non-Abelian Lie sub-algebra of $\mathfrak{K}_P(\mathcal{M})$.
Consequently $\dim\{\mathfrak{K}_P(\mathcal{M})\}>2$ in this instance.

\subsection{Affine Killing vector fields on Type~$\mathcal{A}$ surfaces}

We begin by establishing some technical results. Let $\Re(\cdot)$ and $\Im(\cdot)$ denote the real
and imaginary parts of a complex valued function. Let
$$\mathfrak{K}_0^{\mathcal{A}}:=\operatorname{Span}\{\partial_1,\partial_2\}\,.$$
If $\mathcal{M}\in\mathcal{F}^{\mathcal{A}}$, then Remark~\ref{R1.3}
shows $\mathfrak{K}_0^{\mathcal{A}}\subset\mathfrak{K}(\mathcal{M})$.
The adjoint action of $\mathfrak{K}_0^{\mathcal{A}}$ makes $\mathfrak{K}(\mathcal{M})$ into a
$\mathfrak{K}_0^{\mathcal{A}}$ module. This module action will play an important role in the proof
of the following result.

\begin{lemma}\label{L3.3}
Let $\mathcal{M}\in\mathcal{F}^{\mathcal{A}}$. Suppose that $\dim\{\mathfrak{K}(\mathcal{M})\}>2$. 
There exists a linear change of coordinates so that $\mathcal{M}$ has the following properties:
\begin{enumerate}
\item There exists $X\in\mathfrak{K}(\mathcal{M})$ so that one of the following possibilities holds:
\begin{enumerate}
\item $X=\Re\{e^{a_1x^1+a_2x^2}\}\partial_1$ for $0\ne(a_1,a_2)\in\mathbb{C}^2$.
\item $X=(a_1x^1+a_2x^2)\partial_1$ for $0\ne(a_1,a_2)\in\mathbb{R}^2$.
\end{enumerate}
\smallbreak\item $\rho=\rho_{22}dx^2\otimes dx^2$.
\smallbreak\item If $X\in\mathfrak{K}(\mathcal{M})$, then $X=\zeta(x^1,x^2)\partial_1+ c_2\partial_2$
for $c_2 \in\mathbb{R}$.
\end{enumerate}
\end{lemma}
\begin{proof} We proceed seriatim.
\smallbreak\noindent{\bf Step 1: the proof of Assertion~1.}
We complexify to set 
$$
\mathfrak{L}:=\mathfrak{K}(\mathcal{M})\otimes_{\mathbb{R}}\mathbb{C}\text{ and }
\mathfrak{L}_0^{\mathcal{A}}:=\mathfrak{K}_0^{\mathcal{A}}\otimes_{\mathbb{R}}\mathbb{C}\,.
$$
Since $\dim\{\mathfrak{K}(\mathcal{M})\}>2$, we may choose
$X\in\mathfrak{L}-\mathfrak{L}_0$. 
By Lemma~\ref{L3.1}, $\dim\{\mathfrak{L}\}\le5$. 
Since $\partial_iX=[\partial_i,X]\in\mathfrak{L}$ for $i=1,2$,
there must be minimal non-trivial dependence relations:
\begin{eqnarray*}
&&\partial_1^{r}X+c_{r-1}\partial_1^{r-1}X+\dots+c_0X=0\text{ with }r>0,\\
&&\partial_2^{s}X+\tilde c_{s-1}\partial_2^{s-1}X+\dots+\tilde c_0X=0\text{ with }s>0
\end{eqnarray*}
for some suitably chosen constants $c_i$ and $\tilde c_i$.
We factor the associated characteristic polynomials to express these dependence relation in the form:
\begin{eqnarray}
&&\prod_{t=1}^u(\partial_1-\lambda_t)^{\mu_t}X=0\text{ for }\lambda_t\in\mathbb{C}\text{ distinct}
\text{ and }\mu_t\ge1\,,\label{Gx}\\
&&\prod_{v=1}^w(\partial_2-\eta_v)^{\nu_v}X=0\text{ for }\eta_v\in\mathbb{C}\text{ distinct}
\text{ and }\nu_v\ge1\,.\label{Gy}
\end{eqnarray}
\smallbreak\noindent{\bf Case 1.} Suppose that some $\lambda_t$ is non-zero (if all $\lambda_t$ are zero and some $\eta_v$ is non-zero the analysis is analogous). By
reordering the roots, we may assume $\lambda_1\ne0$.
Since we have chosen a minimal dependence relation, we have
$$
0\ne Y:=(\partial_1-\lambda_1)^{\mu_1-1}\dots(\partial_1-\lambda_u)^{\mu_u}X\in\mathfrak{L}\,.
$$
By replacing $X$ by $Y$, we may assume the dependence relation of Equation~(\ref{Gx})
is $(\partial_1-\lambda_1)X=0$. This implies
$$X=e^{\lambda_1x^1}(\xi^1(x^2)\partial_1+\xi^2(x^2)\partial_2)
\in\mathfrak{L}\text{ where }\lambda_1\ne0\,.
$$
A similar argument shows we may assume that Equation~(\ref{Gy}) 
takes the form $(\partial_2-\eta_1)X=0$ for
some $\eta_1$ (possibly 0). We then conclude
$$
X=e^{\lambda_1x^1+\eta_1x^2}X_0\in\mathfrak{L}\text{ for }
0\ne X_0\in\mathfrak{L}_0^{\mathcal{A}}\text{ and }0\ne\lambda_1\,.
$$
If $\Re(X_0)$ and $\Im(X_0)$ are linearly dependent over $\mathbb{R}$, then
we can multiply $X_0$ by an appropriate non-zero complex number to assume
$0\ne X_0\in\mathfrak{K}_0^{\mathcal{A}}$ is real. This implies 
$\Re(X)=\Re\{e^{\lambda_1x^1+\eta_1x^2}\}X_0$ has the form
given in Assertion~1a. 
We therefore suppose that $\Re(X_0)$ and $\Im(X_0)$ are linearly independent.
We can make a linear change of coordinates to assume $X_0=\frac12(\partial_1-\sqrt{-1}\partial_2)=\partial_z$ where
$z=x^1+\sqrt{-1}x^2$. Thus $Z=e^{\phi}\partial_z\in\mathfrak{L}$ 
for some suitably chosen non-trivial linear function $\phi$. 
\smallbreak\noindent{\bf Case 1a.} Suppose $\phi$ is purely imaginary. This implies
$\phi=\sqrt{-1} (a^1x^1+a^2x^2)$
for $0\ne(a^1,a^2)\in\mathbb{R}^2$. 
We can rotate $\mathbb{R}^2$ and then rescale to suppose that
$\phi=\sqrt{-1}x^1$.
We then have $X=\{\cos(x^1)+\sqrt{-1}\sin(x^1)\}(\partial_{1}-\sqrt{-1}\partial_{2})/2$ and thus
$$
\Re\{X\}= {\textstyle\frac12}(\cos(x^1)\partial_{1}+\sin(x^1)\partial_{2})\in\mathfrak{K}(\mathcal{M})\,.
$$
We have Killing equations:
\begin{eqnarray*}
&&(1-2\Gamma_{12}{}^1)\cos(x^1)+\Gamma_{11}{}^1\sin(x^1)=0,\\
&&(\Gamma_{11}{}^1-2\Gamma_{12}{}^2)\cos(x^1)+(1+2\Gamma_{11}{}^2)\sin(x^1)=0,\\
&&\Gamma_{22}{}^1\cos(x^1)=0,\\
&&(\Gamma_{12}{}^1-\Gamma_{22}{}^2)\cos(x^1)+\Gamma_{12}{}^2\sin(x^1)=0.
\end{eqnarray*}
We solve these relations to see
$$\begin{array}{llllll}
\Gamma_{11}{}^1=0,&\Gamma_{11}{}^2=-\frac12,&\Gamma_{12}{}^1=\frac12,&
\Gamma_{12}{}^2=0,&\Gamma_{22}{}^1=0,&\Gamma_{22}{}^2=\frac12.
\end{array}$$
The Ricci tensor of this structure is zero; this is false as $\mathcal{M}$ is assumed non-flat.
\smallbreak\noindent{\bf Case 1b.} Suppose $\phi$ is holomorphic. We can then rotate and rescale
 to ensure that $\phi(z)=z$ so
$$
\Re(X)= \Re\{e^{x^1+\sqrt{-1}x^2}(\partial_{1}-\sqrt{-1}\partial_2)/2\}=e^{x^1}(\cos(x^2)\partial_1+\sin(x^2)\partial_2)/2\,.
$$
We have Killing equations:
\begin{eqnarray*}
&&(1+\Gamma_{11}{}^1)\cos(x^2)+(\Gamma_{11}{}^2+2\Gamma_{12}{}^1)\sin(x^2)=0,\\
&&\Gamma_{11}{}^2\cos(x^2)+(1-\Gamma_{11}{}^1+2\Gamma_{12}{}^2)\sin(x^2)=0,\\
&&\Gamma_{22}{}^2\cos(x^2)-(1+2\Gamma_{12}{}^2+\Gamma_{22}{}^1)\sin(x^2)=0.
\end{eqnarray*}
We solve these equations to see
$$\begin{array}{llllll}
\Gamma_{11}{}^1=-1,&\Gamma_{11}{}^2=0,&\Gamma_{12}{}^1=0,&
\Gamma_{12}{}^2=-1,&\Gamma_{22}{}^1=1,&\Gamma_{22}{}^2=0.
\end{array}$$
The Ricci tensor of this structure is zero; this is false as $\mathcal{M}$ is assumed non-flat.

\smallbreak\noindent{\bf Case 1c.} Assume that $\phi$ is not purely imaginary and that $\phi$ is not
holomorphic. Since $\bar X\in\mathfrak{L}$, 
$$
[X,\bar X]=e^{\phi+\bar\phi}\{\partial_z\bar\phi\cdot\partial_{\bar z}- \partial_{\bar z}\phi\cdot\partial_z\}
= 2\sqrt{-1} e^{\phi+\bar\phi}\Im\{\partial_z\bar\phi\cdot\partial_{\bar z}\}\in\mathfrak{L}\,.
$$
Since $\phi$ is not purely imaginary, the exponent $\phi+\bar\phi=a_1x^1+a_2x^2$ is non-trivial and
real. Since $\phi$ is not holomorphic, $0\ne\xi:=\Im\{\partial_z\bar\phi\cdot\partial_{\bar z}\}\in\mathfrak{K}_0^{\mathcal{A}}$. We can change coordinates to assume $\xi=\partial_1$.
We then have $-\sqrt{-1}[Z,\bar Z]=e^{\tilde a_1x^1+\tilde a_2x^2}\partial_{1}$ satisfies the hypotheses
of Assertion~1a. This completes the analysis of Case~1.
\medbreak\noindent{\bf Case 2.} Neither dependence relation involves a complex root of the
characteristic polynomials, i.e. we have
$\partial_1^rX=\partial_2^sX=0$. Since $X\notin\mathfrak{K}_0^{\mathcal{A}}$, $(r,s)\ne(1,1)$. If $r>2$,
we replace $X$ by $\partial_1^{r-2}X$ to ensure
$r\le2$. We then argue similarly to choose $X$ so $s\le2$ as well. This implies
$$
X=\sum_{i=0}^1\sum_{j=0}^1(x^1)^i(x^2)^jX_{ij}\text{ for }X_{ij}\in\mathfrak{K}_0^{\mathcal{A}}\,.
$$
If $X_{11}$ is non-zero, we can apply $\partial_1$ to reduce the order and after subtracting the constant
term obtain an element with the form given in Assertion~1b. Otherwise, we may simply subtract $X_{00}$ to
see that there exists $X\in\mathfrak{K}(\mathcal{M})$ so that
$$X=a_i^jx^i\partial_j\in\mathfrak{K}(\mathcal{M})\text{ for }(a_i^j)\ne0\,.$$
If $\operatorname{Rank}\{(a_i^j)\}=1$, then we can change coordinates to assume $X$ has the form given
in Assertion~1b. We therefore assume $\operatorname{Rank}\{(a_i^j)\}=2$
and argue, at length, for a contradiction.
Only the Jordan normal form of the coefficient matrix $(a_i^j)$ is relevant since we are working
modulo linear changes of coordinates. Furthermore, we can always rescale $X$
as needed.
\smallbreak\noindent{\bf Case 2a.} $A$ is diagonalizable. We may suppose $X=x^1\partial_1+ax^2\partial_2$
for $a\ne0$.
We obtain the equations:
$$\begin{array}{lll}
\Gamma_{11}{}^1=0,&(a-2)\Gamma_{11}{}^2=0,&a\Gamma_{12}{}^1=0,\\
\Gamma_{12}{}^2=0,&
(-1+2a)\Gamma_{22}{}^1=0,&a\Gamma_{22}{}^2=0.
\end{array}$$
Thus the only possibly non-zero Christoffel symbols are $\Gamma_{11}{}^2$ and $\Gamma_{22}{}^1$. Since it
is not possible that $(a-2)=0$ and $(-1+2a)=0$ simultaneously, 
we also have $\Gamma_{11}{}^2\Gamma_{22}{}^2=0$. This implies $\rho=0$ so this case is ruled out.
\smallbreak\noindent{\bf Case 2b.} $A$ has two equal non-zero eigenvalues and non trivial Jordan normal form.
We may suppose that $X=(x^1+x^2)\partial_1+x^2\partial_2$ and obtain Killing equations:
$$\begin{array}{lll}
\Gamma_{11}{}^1-\Gamma_{11}{}^2=0,&
\Gamma_{11}{}^2=0,&
\Gamma_{11}{}^1+\Gamma_{12}{}^1-\Gamma_{12}{}^2=0,\\
\Gamma_{11}{}^2+\Gamma_{12}{}^2=0,&
2\Gamma_{12}{}^1+\Gamma_{22}{}^1-\Gamma_{22}{}^2=0,&
2\Gamma_{12}{}^2+\Gamma_{22}{}^2=0.
\end{array}$$
We solve these equations to see $\Gamma=0$ and hence $\rho=0$ so this case is ruled out.
\smallbreak\noindent{\bf Case 2c.} The matrix $A$ has two complex eigenvalues with non-zero imaginary part.
We may assume $X=(ax^1+x^2)\partial_1+(-x^1+ax^2)\partial_2$ is an affine Killing vector
field for some $a\in\mathbb{R}$. We use the Killing equations to eliminate variables recursively.
We set $\Gamma_{12}{}^2=2s$ and $\Gamma_{22}{}^2=2t$. At each stage we simplify the resulting
Killing equations based on the previous computations:
\begin{enumerate}
\item The Killing equation $4s+2at+\Gamma_{22}{}^1=0$ yields $\Gamma_{22}{}^1=-4s-2at$.
\item The Killing equation $2as+t+a^2t-\Gamma_{12}{}^1=0$ yields $\Gamma_{12}{}^1=2as+t+a^2t$.
\item The Killing equation $4as-t+a^2t+\Gamma_{11}{}^2=0$ yields $\Gamma_{11}{}^2=-4as+t-a^2t$.
\item The Killing equation $2(1+a^2)s+3at+a^3t+\Gamma_{11}{}^1=0$ yields 
\newline$\Gamma_{11}{}^1=-2(1+a^2)s-3at-a^3t$.
\end{enumerate}
We now obtain Killing equations in the parameters $(s,t)$ which imply
$3s+at=0$ and $2as+t(3+a^2)=0$. We set $s=-at/3$ to obtain the equation $3t+a^2t/3=0$. This implies $t=0$
so $s=0$ and $\Gamma=0$. Thus this case is ruled out. This completes the proof of Assertion~1.

\medbreak\noindent{\bf Step 2: the proof of Assertion~2.} By Assertion~1, 
$X=f(x^1,x^2)\partial_1\in\mathfrak{K}(M)$ for some non-constant function $f$.
Choose $P\in\mathbb{R}^2$ so $df(P)\ne0$. Let
$Y=c_1\partial_1+c_2\partial_2$ for $(c_1,c_2)\ne(0,0)$. Then
$$0=\{\mathcal{L}_X(\rho)\}(Y,Y)=X(\rho(Y,Y))-2\rho([X,Y],Y)\,.$$
Because $\rho(Y,Y)$ is
constant, $X\{\rho(Y,Y)\}=0$. For generic $(c_1,c_2)$, 
$$
[X,Y](P)=\{Y(f)(P)\}\partial_1\ne0\text{ so }\rho(\partial_1,Y)=0\,.
$$
This implies 
$\rho_{11}=\rho_{12}=0$ so $\rho=\rho_{22}dx^2\otimes dx^2$.

\medbreak\noindent{\bf Step 3: the proof of Assertion~3.}
Let $X=\xi^1(x^1,x^2)\partial_1+\xi^2(x^1,x^2)\partial_2\in\mathfrak{K}(\mathcal{M})$. 
Let $Y=c_1\partial_1+c_2\partial_2$. We argue as above to see that $\rho([X,Y],Y)=0$.
Since $\rho=\rho_{22}dx^2\otimes dx^2$, this implies
$$c_2\rho_{22}(c_1\partial_1\xi^2+c_2\partial_2\xi^2)=0\text{ for all }(c_1,c_2)\in\mathbb{R}^2\,.$$
This implies $\xi^2$ is constant which establishes Assertion~3 and completes the proof of the Lemma.
\end{proof}

Lemma~\ref{L3.3} focuses attention on the case that $\operatorname{Rank}\{\rho\}=1$. The following result relates the rank of the Ricci tensor with the dimension of the space of affine Killing vector fields.

\begin{theorem}\label{T3.4}
Let $\mathcal{M}\in\mathcal{F}^{\mathcal{A}}$. 
\begin{enumerate}
\item Suppose $\rho=\rho_{22}dx^2\otimes dx^2$.
\begin{enumerate}
\item If $\Gamma_{11}{}^1\ne0$, then $X\in\mathfrak{K}(\mathcal{M})$ if and only
if $X=e^{-\Gamma_{11}{}^1x^1}\xi(x^2)\partial_1+X_0$ for $X_0\in\mathfrak{K}_0^{\mathcal{A}}$ where $\xi$ satisfies
$\xi^{\prime\prime}+(2\Gamma_{12}{}^1-\Gamma_{22}{}^2)\xi^\prime+\Gamma_{11}{}^1\Gamma_{22}{}^1\xi=0$.
\item If $\Gamma_{11}{}^1=0$, then $X\in\mathfrak{K}(\mathcal{M})$ if and
only if $X=(\xi(x^2)+c_1x^1)\partial_1+X_0$ for $X_0\in\mathfrak{K}_0^{\mathcal{A}}$ where $\xi$ satisfies
$\xi^{\prime\prime}+(2\Gamma_{12}{}^1-\Gamma_{22}{}^2)\xi^\prime-c_1\Gamma_{22}{}^1=0$.
\end{enumerate}
\item The following assertions are equivalent.
 \begin{enumerate}
\item $\dim\{\mathfrak{K}(\mathcal{M})\}=4$.
\item $\dim\{\mathfrak{K}(\mathcal{M})\}>2$.
\item $\operatorname{Rank}\{\rho\}=1$.
\end{enumerate}
\item The following assertions are equivalent:
\begin{enumerate}
\item $\dim\{\mathfrak{K}(\mathcal{M})\}=2$.
\item $\operatorname{Rank}\{\rho\}=2$.
\end{enumerate}\end{enumerate}
\end{theorem}
\begin{proof}
We use Lemma~\ref{L2.3} to impose the conditions $\Gamma_{11}{}^2=0$ and $\Gamma_{12}{}^2=0$ and Lemma~\ref{L3.3} to write $X=\zeta(x^1,x^2)\partial_1+c_2\partial_2$.
The Killing equations now become
$$
\begin{array}{l}
\zeta^{(2,0)}+\Gamma_{11}{}^1 \zeta^{(1,0)} =0,\\
\zeta^{(1,1)}+\Gamma_{11}{}^1 \zeta^{(0,1)} =0,\\
\zeta^{(0,2)}-\Gamma_{22}{}^1 \zeta^{(1,0)}+(2\Gamma_{12}{}^1-\Gamma_{22}{}^2)\zeta^{(0,1)}=0.
\end{array}
$$
We establish Assertion~1
by examiningg cases.
\begin{enumerate}
\item Suppose that $\Gamma_{11}{}^1\ne0$. We have 
$\zeta(x^1,x^2)=u_0(x^2)+u_1(x^2)e^{-\Gamma_{11}{}^1x^1}$. A Killing equation is
$\Gamma_{11}{}^1u_0^\prime=0$. Thus we may take $u_0$ constant and delete it from
further consideration. The remaining Killing equation is the condition of Assertion~1a.
\item Suppose $\Gamma_{11}{}^1=0$. 
We have $\zeta(x^1,x^2)=u_0(x^2)+u_1(x^2)x^1$. A Killing
equation is $u_1^\prime=0$ and hence $u_1(x^2)=c_1$ is constant. 
The remaining Killing equation is the condition of Assertion~1b.
\end{enumerate}

Clearly Assertion~2a implies Assertion~2b.
We use Lemma~\ref{L3.3} to see that Assertion~2b implies Assertion~2c.
We will apply Assertion~1 to see Assertion~2c implies Assertion~2a. We argue as follows.
Suppose first $\Gamma_{11}{}^1\ne0$. Let $\{\xi_1,\xi_2\}$ be a basis for
the space of solutions to the Equation of Assertion 1a. Then
$$
\mathfrak{K}(\mathcal{M})=\operatorname{Span}_{\mathbb{R}}
\{\xi_1(x^2)e^{-\Gamma_{11}{}^1x^1}\partial_1,
\xi_2(x^2)e^{-\Gamma_{11}{}^1x^1}\partial_1,
\partial_1,\partial_2\}
$$
and hence $\dim\{\mathfrak{K}(\mathcal{M})\}=4$.
Suppose on the other hand that $\Gamma_{11}{}^1=0$. 
Choose a solution $\xi_0(x^2)$ to the Equation of Assertion 1b with $c_1=1$, i.e. we have
$$
\zeta(x^1,x^2)=x^1+\xi_0(x^2)\text{ where }
\xi_0^{\prime\prime}+(2\Gamma_{12}{}^1-\Gamma_{22}{}^2)\xi_0^\prime-\Gamma_{22}{}^1=0\,.
$$
Let $\{\xi_1,\xi_2\}$ be a basis for the space of solutions to the homogeneous equation
$\xi^{\prime\prime}+(2\Gamma_{12}{}^1-\Gamma_{22}{}^1)\xi^\prime=0$. Then
$$
\mathfrak{K}(\mathcal{M})=\operatorname{Span}_{\mathbb{R}}
\{\xi_0\partial_1,\xi_1\partial_1,\xi_2\partial_1,\partial_1,\partial_2\}\,.
$$
Since we may take $\xi_1=1$, $\xi_1\partial_1=\partial_1$ and we see that
$\dim\{\mathfrak{K}(\mathcal{M})\}=4$. This completes the proof of Assertion~2; 
the final Assertion is now immediate. 
\end{proof}

There are several Lie algebras which will play an important role in our analysis.
Let $A_2:=\operatorname{Span}_{\mathbb{R}}\{e_1,e_2\}$ with Lie bracket $[e_1,e_2]=e_2$;
up to isomorphism, $A_2$ is 
 the only non-trivial real Lie algebra of dimension two; it is the Lie algebra
of the ``$ax+b$" group. We adopt the notation of Patera et. al \cite{PSWZ76} to define
several other Lie algebras. Let $\{e_1,e_2,e_3,e_4\}$ be a basis of $\mathbb{R}^4$.
We define the following solvable Lie algebras by specifying their bracket relations.
\begin{itemize}
\item $A_{2}\oplus A_{2}$: the relations of the bracket are given by
\[
[e_1,e_2]=e_2, \, [e_3,e_4]=~e_4\,.
\]
\item $A_{4,9}^b$: the relations of the bracket for $-1\leq b\leq 1$ are given by
\[
[e_2,e_3]=e_1,\, [e_1,e_4]=(1+b)e_1,\, [e_2,e_4]=e_2,\, [e_3,e_4]=b e_3\,.
\]
\item $A_{4,12}$: the relations of the bracket are given by
\[
[e_1,e_3]=e_1,\, [e_2,e_3]=e_2,\, [e_1,e_4]=-e_2,\, [e_2,e_4]=e_1\,.
\]
\end{itemize} 
\begin{definition}\label{D3.5}
\rm Let $\mathcal{M}_\star^\star$ be the affine surface defined by the structures:
\medbreak\quad $\mathcal{M}_1$:
 $\Gamma_{11}{}^1=-1$, $\Gamma_{11}{}^2=0$, $\Gamma_{12}{}^1=1$, $\Gamma_{12}{}^2=0$, 
 $\Gamma_{22}{}^1=\phantom{-}0$, $\Gamma_{22}{}^2=2$.
\smallbreak\quad $\mathcal{M}_2^c$:
 $\Gamma_{11}{}^1=-1$, $\Gamma_{11}{}^2=0$, 
 $\Gamma_{12}{}^1=c$, $\Gamma_{12}{}^2=0$, $\Gamma_{22}{}^1=\phantom{-}0$,
$\Gamma_{22}{}^2=1+2c$,\smallbreak\qquad\qquad\qquad\qquad\quad where $c^2+c\ne0$.
\smallbreak\quad $\mathcal{M}_3^c$: $\Gamma_{11}{}^1=\phantom{-}0$, $\Gamma_{11}{}^2=0$,
$\Gamma_{12}{}^1=c$, $\Gamma_{12}{}^2=0$, 
 $\Gamma_{22}{}^1=\phantom{-}0$,
$\Gamma_{22}{}^2=1+2c$,\smallbreak\qquad\qquad\qquad\qquad\quad where $c^2+c\ne0$.
\smallbreak\quad $\mathcal{M}_4^c$: $\Gamma_{11}{}^1=\phantom{-}0$, 
$\Gamma_{11}{}^2=0$, $\Gamma_{12}{}^1=1$,
$\Gamma_{12}{}^2=0$, $\Gamma_{22}{}^1=\phantom{-}c$, $\Gamma_{22}{}^2=2$.
\smallbreak\quad $\mathcal{M}_5^c:$ $\Gamma_{11}{}^1=-1$, 
$\Gamma_{11}{}^2=0$, $\Gamma_{12}{}^1=c$,
$\Gamma_{12}{}^2=0$,
$\Gamma_{22}{}^1=-1$, $\Gamma_{22}{}^2=2c$.
\end{definition}

We use Lemma ~\ref{L2.3} to see $\rho(\mathcal{M}_\star^\star)=\rho_{22}dx^2\otimes dx^2\ne0$
for $\rho_{22}\ne0$ so
none of these examples is flat. We compute $\rho_{22}$ and $\alpha$:
$$\begin{array}{llll}
\rho_{22}(\mathcal{M}_1)=1,&\alpha(\mathcal{M}_1)=16,\\[0.05in]
\rho_{22}(\mathcal{M}_2^c)=c^2+c,&
\alpha(\mathcal{M}_2^c)=\frac{4(1+2c)^2}{c^2+c}\in(-\infty,0]\cup(16,\infty),\\[0.05in]
\rho_{22}(\mathcal{M}_3^c)=c^2+c,&
{ \alpha(\mathcal{M}_3^c)}=\frac{4(1+2c)^2}{c^2+c}\in(-\infty,0]\cup(16,\infty),\\[0.05in]
\rho_{22}(\mathcal{M}_4^c)=1,&\alpha(\mathcal{M}_4^c)=16,\\[0.05in]
\rho_{22}(\mathcal{M}_5^c)=1+c^2,&\alpha(\mathcal{M}_5^c)=\frac{16c^2}{1+c^2}\in [0,16).
\end{array}$$
The general linear group $\operatorname{GL}(2,\mathbb{R})$ acts on the space
of Christoffel symbols by pull-back; we say that $\Gamma_1$ and $\Gamma_2$
are {\it linearly equivalent} if there exists $T\in\operatorname{GL}(2,\mathbb{R})$
so that $T^*(\nabla^{\Gamma_1})=\nabla^{\Gamma_2}$.
We have the following classification result.

\begin{lemma}\label{L3.6}
\ \begin{enumerate}
 \item If $\mathcal{M}\in\mathcal{F}^{\mathcal{A}}$ and $\operatorname{Rank}(\rho)=1$, then
$\mathcal{M}$ is linearly equivalent to $\mathcal{M}_1$, $\mathcal{M}_2^c$, $\mathcal{M}_3^c$, 
$\mathcal{M}_4^c$, or $\mathcal{M}_5^c$.
\smallbreak
\item $\mathfrak{K}(\mathcal{M}_1)=\operatorname{Span}_{\mathbb{R}}\{e^{x^1}\partial_1,
x^2e^{x^1}\partial_1\}\oplus\mathfrak{K}_0^{\mathcal{A}}\approx A_{4,9}^0$.
\smallbreak\item 
$\mathfrak{K}(\mathcal{M}_2^c)=\operatorname{Span}_{\mathbb{R}}
\{e^{x^1}\partial_1,e^{x^1+x^2}\partial_1\}
\oplus\mathfrak{K}_0^{\mathcal{A}}\approx A_2\oplus A_2$.
\smallbreak\item 
$\mathfrak{K}(\mathcal{M}_3^c)=\operatorname{Span}\{
e^{x^2}\partial_1,x^1\partial_1\}\oplus\mathfrak{K}_0^{\mathcal{A}}
\approx A_2\oplus A_2$.
\smallbreak\item  
$\mathfrak{K}(\mathcal{M}_4^c)=\operatorname{Span}\{
x^2\partial_1,(c\cdot(x^2)^2+2x^1)\partial_1\}
\oplus\mathfrak{K}_0^{\mathcal{A}}\approx A_{4,9}^0$.
\smallbreak\item 
$\mathfrak{K}(\mathcal{M}_5^c)=\operatorname{Span}_{\mathbb{R}}
\{e^{x^1}\cos(x^2)\partial_1,e^{x^1}\sin(x^2)\partial_1\}\oplus \mathfrak{K}_0^\mathcal{A} 
\approx A_{4,12}$.
\smallbreak\item $\mathcal{M}_1$, $\mathcal{M}_2^c$, $\mathcal{M}_3^c$, and
$\mathcal{M}_4^c$ are also Type~$\mathcal{B}$; $\mathcal{M}_5^c$ is not
Type~$\mathcal{B}$.
\end{enumerate}\end{lemma}
\begin{proof}Assume $\rho$ has rank $1$ and make a linear change of coordinates to assume
$\rho=\rho_{22}dx^2\otimes dx^2$. By Theorem~\ref{T3.4},
there exists $X\in\mathfrak{K}(\mathcal{M})-\mathfrak{K}_0^{\mathcal{A}}$
of the form $\zeta(x^1,x^2)\partial_1$ where $\zeta$ is non-constant. Lemma~\ref{L3.3}
then shows either $\zeta=e^{a_1x^1+a_2x^2}$ for $(0,0)\ne(a_1,a_2)\in\mathbb{R}^2$
(Case 1 and Case 2 below), or $\zeta=a_1x^1+a_2x^2$ for $(0,0)\ne(a_1,a_2)\in\mathbb{R}^2$
(Case 3 below), or $\zeta=\Re\{e^{a_1x^1+a_2x^2}\}$ for $(a_1,a_2)\in\mathbb{C}^2-\mathbb{R}^2$
(Case 4 below).
We examine these possibilities seriatim.
\smallbreak\noindent{\bf Case 1.} 
Assume $e^{a_1x^1+a_2x^2}\partial_1\in\mathfrak{K}(\mathcal{M})$ for
$a_1\ne0$. Let 
$$
(u^1,u^2):=(a_1x^1+a_2x^2,x^2)\quad\text{ so }
e^{u^1}\partial_{1}^u=(a_1)^{-1}e^{a_1x^1+a_2x^2}\partial_1\in\mathfrak{K}(\mathcal{M})\,.
$$
Thus we may assume that $e^{x^1}\partial_1\in\mathfrak{K}(\mathcal{M})$.
Let $X_1:=\partial_1+\partial_2$
and $X_2:=e^{x^1}\partial_1$, then $\{X_1(P),X_2(P)\}$ are linearly independent
for any
point $P\in\mathbb{R}^2$. By Lemma~\ref{L3.2},
$\mathcal{M}$ is also Type~$\mathcal{B}$ since $[X_1,X_2]=X_2$.
A direct computation shows $X=e^{x^1}\partial_1$ is an affine Killing vector field if and only if
$$
\Gamma_{11}{}^1=-1,\quad\Gamma_{11}{}^2=0,\quad\Gamma_{12}{}^2=0,\quad\Gamma_{22}{}^1=0\,.$$
We impose these relations and obtain 
$\rho_{22}=\Gamma_{12}{}^1(\Gamma_{22}{}^2-\Gamma_{12}{}^1)\ne0$. 
Two sub-cases present themselves when we search for another affine Killing vector field:
\smallbreak\noindent{\bf Case 1a.} Assume $\Gamma_{22}{}^2=2\Gamma_{12}{}^1$.
We set $Y=x^2e^{x^1}\partial_1$ and verify $Y$ is an affine Killing vector field. By 
Theorem~\ref{T3.4}, $\dim\{\mathfrak{K}(\mathcal{M})\}=4$. Thus
$$
\mathfrak{K}(\mathcal{M})=\operatorname{Span}_{\mathbb{R}}
\{e^{x^1}\partial_1,x^2e^{x^1}\partial_1\}\oplus\mathfrak{K}_0^{\mathcal{A}}\,.
$$
Since $\rho_{22}\ne0$, $\Gamma_{12}{}^1\ne0$. By rescaling $x^2$, we
may assume that $\Gamma_{12}{}^1=1$;
 this yields the surface $\mathcal{M}_1$. Set $e_1:=e^{x^1}\partial_1$,
$e_2:=x^2e^{x^1}\partial_1$, $e_3:=-\partial_2$, $e_4:=-\partial_1$.
We then have
 the bracket relations of the Lie algebra $A_{4,9}^0$:
$$
[e_2,e_3]=e_1,\quad[e_1,e_4]=e_1,\quad[e_2,e_4]=e_2\,.
$$

\smallbreak\noindent{\bf Case 1b.} Assume $\Gamma_{22}{}^2-2\Gamma_{12}{}^1\ne0$.
Then $Y=e^{x^1+(\Gamma_{22}{}^2-2\Gamma_{12}{}^1)x^2} \partial_1$ is an affine
 Killing vector field
distinct from $e^{x^1}\partial_1$.
By replacing $x^2$ by $(\Gamma_{22}{}^2-2\Gamma_{12}{}^1)^{-1}x^2$, we may assume
$Y=e^{x^1+x^2}\partial_1$; the Killing equations then yield 
$\Gamma_{22}{}^2=2\Gamma_{12}{}^1+1$
and thus
$$
\mathfrak{K}(\mathcal{M})=\operatorname{Span}_{\mathbb{R}}
\{e^{x^1}\partial_1,e^{x^1+x^2}\partial_1\}\oplus\mathfrak{K}_0^{\mathcal{A}}\,.
$$
We set $e_1:=\partial_2$, $e_2:=e^{x^1+x^2}\partial_1$, $e_3:=\partial_1-\partial_2$, 
$e_4:=e^{x^1}\partial_1$. This yields the surface $\mathcal{M}_2^c$.
We then
have the bracket relations of the Lie algebra $A_2\oplus A_2$:
$$ [e_1,e_2]=e_2,\quad[e_3,e_4]=e_4\,.$$

\medbreak\noindent{\bf Case 2.}
Assume $e^{a_1x^1+a_2x^2}\partial_1\in\mathfrak{K}(\mathcal{M})$ for
$a_1=0$. Hence 
 $e^{a_2x^2}\partial_1\in\mathfrak{K}(\mathcal{M})$ for $a_2\ne0$.
We may rescale $x^2$ to assume $a_2=1$.
A direct computation shows $e^{x^2}\partial_1$ is an affine Killing vector field if and only if
$$
\Gamma_{11}{}^1=0,\quad\Gamma_{11}{}^2=0,\quad
\Gamma_{12}{}^2=0,\quad\Gamma_{22}{}^2=1+2\Gamma_{12}{}^1\,.
$$
If we set $Y=(x^1-\Gamma_{22}{}^1x^2)\partial_1$, then this is an affine Killing vector field.
We may make a linear change of variables to replace $x^1$ by $x^1-\Gamma_{22}{}^1x^2$
to obtain $x^1\partial_1$ is an affine Killing vector field; this implies $\Gamma_{22}{}^1=0$.
This yields the surface $\mathcal{M}_3^c$.
We then have
$$
\mathfrak{K}(\mathcal{M})
=\operatorname{Span}\{e^{x^2}\partial_1,x^1\partial_1\}\oplus\mathfrak{K}_0^{\mathcal{A}}\,.
$$
We set $X_1=\partial_2$ and $X_2=e^{x^2}\partial_1$. Since $[X_1,X_2]=X_2$
and $\{X_1(P),X_2(P)\}$ are linearly independent for any $P\in\mathbb{R}^2$, Lemma~\ref{L3.2}
implies $\mathcal{M}$ is Type~$\mathcal{B}$ as well. We set $e_1:=-x^1\partial_1-\partial_2$,\quad 
$e_2:=-\partial_1$, 
$e_3:=\partial_2$, and
$e_4:=e^{x^2}\partial_1$. 
We then have the bracket relations of the Lie algebra
$A_2\oplus A_2$:
$$
[e_1,e_2]=e_2,\quad [e_3,e_4]=e_4\,.
$$
\smallbreak\noindent{\bf Case 3.} Assume $(a_1x^1+a_2x^2)\partial_1\in\mathfrak{K}(\mathcal{M})$
for $(a_1,a_2)\ne(0,0)$. This gives rise to several cases; we can rescale $x^1$ and $x^2$
to replace $a_i$ by $\lambda_ia_i$. Thus we need only consider $(a_1,a_2)\in\{(1,0),(1,1),(0,1)\}$.
\smallbreak\noindent{\bf Case 3a.}
 Suppose $(a_1,a_2)=(0,1)$ so $x^2\partial_1\in\mathfrak{K}(\mathcal{M})$.
A direct computation shows $x^2\partial_1$ is an  affine Killing vector field if and only if
$$
\Gamma_{11}{}^1=0,\quad\Gamma_{11}{}^2=0,
\quad\Gamma_{12}{}^2=0,\quad\Gamma_{22}{}^2=2\Gamma_{12}{}^1\,.
$$
We then have $\rho_{22}=(\Gamma_{12}{}^1)^2$. By rescaling $x^2$, we may assume $\Gamma_{12}{}^1=1$
and hence $\Gamma_{22}{}^2=2$. 
 We obtain the surface $\mathcal{M}_4^c$.
We set $Y=(2x^1+c\cdot(x^2)^2)\partial_1$ and
verify that $Y$ is an affine Killing vector field. Thus
$$
\mathfrak{K}(\mathcal{M})=\operatorname{Span}_{\mathbb{R}}
\{x^2\partial_1,(2x^1+c\cdot(x^2)^2)\partial_1\}\oplus\mathfrak{K}_0^{\mathcal{A}}\,.
$$
We set $X_1=\partial_1$ and $X_2=\partial_2+ (x^1+c\cdot(x^2)^2/2)\partial_1$.
Then $\{X_1(P),X_2(P)\}$ are linearly independent for any point $P\in\mathbb{R}^2$. Set
$e_1:=\partial_1$, $e_2:=x^2\partial_1$,  $e_3:=-\partial_2+cx^2\partial_1$,
$e_4:=( \frac12c\cdot(x^2)^2+x^1)\partial_1$.  We 
obtain the bracket relations of the Lie algebra $A_{4,9}^0$:
$$
 [e_2,e_3]=e_1,\quad[e_1,e_4]=e_1,\quad[e_2,e_4]=e_2\,.
$$

\smallbreak\noindent{\bf Case 3b.}
 Suppose $(a_1,a_2)=(1,0)$ so $x^1\partial_1\in\mathfrak{K}(\mathcal{M})$.
The Killing equations yield the relations:
$$\Gamma_{11}{}^1=0,\quad\Gamma_{11}{}^2=0,\quad\Gamma_{12}{}^2=0,\quad\Gamma_{22}{}^1=0\,.$$
Suppose first $2\Gamma_{12}{}^1\ne\Gamma_{22}{}^2$. Set $Y=e^{(\Gamma_{22}{}^2-2\Gamma_{12}{}^1)x^2}\partial_1$. We then verify that $Y$ is an affine Killing vector field. Thus this is subsumed in Case~2. We
may therefore suppose $2\Gamma_{12}{}^1=\Gamma_{22}{}^2$ and we obtain that $x^2\partial_1$
also is an affine Killing vector field. This is subsumed in Case 3a.
\smallbreak\noindent{\bf Case 3c.} Suppose $(a_1,a_2)=(1,1)$ so 
$(x^1+x^2)\partial_1\in\mathfrak{K}(\mathcal{M})$.
By replacing $x^1$ by $x^1+x^2$, we obtain $x^1\partial_1\in\mathfrak{K}(\mathcal{M})$. 
This is subsumed in Case 3b.

\medbreak\noindent{\bf Case 4.}
Assume $(e^{\alpha_1x^1+\alpha_2x^2})\partial_1$ is a complex
 affine Killing vector field where we have $(\alpha_1,\alpha_2)\in\mathbb{C}^2-\mathbb{R}^2$.
Set $\alpha_1=a_1+\sqrt{-1}a_2$ and $\alpha_2=b_1+\sqrt{-1}b_2$. Then the following
two vector fields are affine Killing vector fields:
\begin{eqnarray*}
&&X:=e^{a_1x^1+b_1x^2}\cos(a_2x^1+b_2x^2)\partial_1,\\
&&Y:=e^{a_1x^1+b_1x^2} \sin(a_2x^1+b_2x^2)\partial_1\,.
\end{eqnarray*}
Consequently
$\mathfrak{K}(\mathcal{M})=\operatorname{Span}_{\mathbb{R}}\{X,Y\}\oplus\mathfrak{K}_0^{\mathcal{A}}$.

\medbreak\noindent{\bf Case 4a.} Suppose $a_2\ne0$.
We can then make a linear change of coordinates to
assume $X=e^{a_1x^1+b_1x^2}\cos(x^1)\partial_1$. The Killing equations yield:
\begin{eqnarray*}
0&=&(-1+a_1^2+a_1\Gamma_{11}{}^1-b_1\Gamma_{11}{}^2)\cos(x^1)
-(2a_1+\Gamma_{11}{}^1)\sin(x^1),\\
0&=&2\Gamma_{11}{}^2(a_1\cos(x^1)-\sin(x^1))\,.
\end{eqnarray*}
This implies:
$$
\Gamma_{11}{}^2=0,\quad
(a_1)^2+a_1 \Gamma_{11}{}^1-1=0,\quad
2 a_1+\Gamma_{11}{}^1=0\,.
$$
Thus $\Gamma_{11}{}^1=-2a_1$. We show this case does not occur by deriving the contradiction:
$$0=(a_1)^2+a_1 \Gamma_{11}{}^1-1=-1-a_1^2\,.$$

\smallbreak\noindent{\bf Case 4b.}
Suppose $a_2=0$ and normalize $x^2$ so that $b_2=1$ and
$$X=e^{a_1x^1+b_1x^2}\cos(x^2)\partial_1\,.$$
Suppose $a_1=0$ so $X=e^{b_1x^2}\cos(x^2)\partial_1$.
We obtain two relations:
\smallbreak\qquad\hfill
$b_1^2+2 b_1 \Gamma_{12}{}^1-b_1 \Gamma_{22}{}^2-1=0$ and 
$-2 b_1-2 \Gamma_{12}{}^1+\Gamma_{22}{}^2=0$.\hfill\quad
\smallbreak\noindent
This implies $b_1=(\Gamma_{22}{}^2-2\Gamma_{12}{}^1)/2$. We derive a contradiction
and show this case can not occur by computing:
\begin{eqnarray*}
&&b_1^2+2 b_1 \Gamma_{12}{}^1-b_1 \Gamma_{22}{}^2-1\\
&=&-(4+4 (\Gamma_{12}{}^1)^2-4 \Gamma_{12}{}^1 \Gamma_{22}{}^2+(\Gamma_{22}{}^2)^2)/4\\
&=&-(4+(2\Gamma_{12}{}^1-\Gamma_{22}{}^2)^2)/4=0\,.
\end{eqnarray*}
Thus $a_1\ne0$ so we can renormalize the coordinates to ensure
$X=e^{x^1}\cos(x^2)\partial_1$.
 The bracket with $\partial_2$ then yields $\tilde X=e^{x^1}\sin(x^2)\partial_1$
also is an  affine Killing vector field. This generates the 4-dimensional Lie algebra 
$\mathfrak{K}(\mathcal{M})$.  Let
$$
e_1:=e^{x^1}\cos(x^2)\partial_1,\quad
e_2:=e^{x^1}\sin(x^2)\partial_1,\quad
e_3:=-\partial_1,\quad 
e_4:=-\partial_2.
$$
We then have the bracket relations of $A_{4,12}$:
$$
[e_1,e_3]=e_1,\quad [e_2,e_3]=e_2,\quad [e_1,e_4]=-e_2,\quad [e_2,e_4]=e_1\,.
$$

This establishes Assertions~1-6;
Assertion~7 follows from Lemma~\ref{L3.2}. 
\end{proof}

\begin{remark}\rm No surface in one family of
Definition~\ref{D3.5} is linearly isomorphic
to a surface in another family. We argue as follows to see this.
The Lie algebra $\mathfrak{K}(\mathcal{M}_5^c)$ is
$A_{4,12}$; this is different from the Lie algebras of the other 4 families so this
family is distinct. Similarly, the Lie algebra of $\mathcal{M}_1$
or $\mathcal{M}_4^c$ is $A_{4,9}^0$ while the Lie algebra of $\mathcal{M}_2^c$ or
$\mathcal{M}_3^c$ is $A_2\oplus A_2$. So we must construct a linear invariant
distinguishing $\mathcal{M}_1$ from $\mathcal{M}_4^c$ or distinguishing
$\mathcal{M}_2^c$ from $\mathcal{M}_3^c$. The Ricci tensor of any surface
in Definition~\ref{D3.5} has rank $1$ so $\ker(\rho)$ is a $1$-dimensional distribution;
we have normalized the coordinate system so 
$\ker(\rho)=\partial_1\cdot\mathbb{R}$. Let $\rho_0:=\Gamma_{ij}{}^jdx^i$. Since
contraction of an upper against a lower index is invariant under the action of
$\operatorname{GL}(2,\mathbb{R})$, $\rho_0$ and hence
$\dim\{\ker(\rho)\cap\ker(\rho_0)\}$ is a
linear invariant.
We compute 
$$
\rho_0^{\mathcal{M}_1}(\partial_1)=-1,\quad
\rho_0^{\mathcal{M}_2^c}(\partial_1)=-1,\quad
\rho_0^{\mathcal{M}_3^c}(\partial_1)=0,\quad
\rho_0^{\mathcal{M}_4^c}(\partial_1)=0\,.
$$
Thus $\ker(\rho)\cap\ker(\rho_0)=\{0\}$ if $\mathcal{M}=\mathcal{M}_1$ or $\mathcal{M}=\mathcal{M}_2^c$ while $\ker(\rho)\cap\ker(\rho_0)\ne\{0\}$ 
if $\mathcal{M}=\mathcal{M}_3^c$ or $\mathcal{M}=\mathcal{M}_4^c$. Thus in fact the 5 families of Definition~\ref{D3.5} are distinct under linear equivalence 
and Lemma~\ref{L3.6} is minimal in this respect.
\end{remark}

 Although the 5 basic families of Definition~\ref{D3.5} are distinct under linear equivalence,
there are non-linear changes of coordinates that can be used to relate members of
different families. We use such changes to establish the following result that shows
that the invariants $\alpha$ and $\epsilon$
form a complete system of invariants for Type~$\mathcal{A}$ surfaces where the Ricci tensor has rank 1.

\begin{theorem}\label{T3.8}
Let $\mathcal{M}$ and $\tilde{\mathcal{M}}$ be Type~$\mathcal{A}$ affine surfaces with
$\operatorname{Rank}\{\rho\}=1$.
Assume that $\alpha(\mathcal{M})=\alpha(\tilde{\mathcal{M}})=\alpha$ and 
that $\epsilon(\mathcal{M})=\epsilon(\tilde{\mathcal{M}})=\epsilon$.
\begin{enumerate}
\item If $\alpha=16$, then $\mathcal{M}\approx\tilde{\mathcal{M}}$,
$\mathfrak{K}(\mathcal{M})\approx A_{4,9}^0$, and $\mathcal{M}$ is
also of Type~$\mathcal{B}$. 
\item If $\alpha\in(0,16)$, then $\mathcal{M}\approx\tilde{\mathcal{M}}$,
$\mathfrak{K}(\mathcal{M})\approx A_{4,12}$ and $\mathcal{M}$
is not of Type~$\mathcal{B}$.
\item If $\alpha\notin[0,16]$, then $\mathcal{M}\approx\tilde{\mathcal{M}}$,
$\mathfrak{K}(\mathcal{M})\approx A_2\oplus A_2$, and
$\mathcal{M}$ is also of Type~$\mathcal{B}$.
\item Assume $\alpha=0$.
\begin{enumerate}
\item If $\epsilon<0$, then $\mathcal{M}\approx\tilde{\mathcal{M}}$,
 $\mathfrak{K}(\mathcal{M})\approx A_2\oplus A_2$, and
$\mathcal{M}$ is also of Type~$\mathcal{B}$.
\item If $\epsilon>0$, then $\mathcal{M}\approx\tilde{\mathcal{M}}$,
$\mathfrak{K}(\mathcal{M})\approx A_{4,12}$ and $\mathcal{M}$
is not of Type~$\mathcal{B}$.
\end{enumerate}
\end{enumerate}
\end{theorem}
\begin{proof}
We first deal with the surfaces $\mathcal{M}_1$ and $\mathcal{M}_2^c$. 

Assume
$
{}^x\Gamma_{11}{}^1=-1$, ${}^x\Gamma_{11}{}^2=0$, ${}^x\Gamma_{12}{}^2=0$,
 and ${}^x\Gamma_{22}{}^1=0$.
Set $u^1=e^{-x^1}$ and $u^2=x^2$. Then
$$\begin{array}{ll}
du^1=-e^{-x^1}dx^1,& du^2=dx^2,\\
\partial_1^u=-e^{x^1}\partial_1,&\partial_2^u=\partial_2\,.
\end{array}$$
We then have
$\operatorname{Span}_{\mathbb{R}}\{\partial_1,e^{x^1}\partial_1,\partial_2\}
=\operatorname{Span}_{\mathbb{R}}\{-u^1\partial_1^u,\partial_1^u,\partial_2^u\}$.
We compute:
\begin{eqnarray*}
&&\nabla_{\partial_1^u}\partial_1^u= e^{x^1}\nabla_{\partial_1}\{e^{x^1}\partial_1\}
=e^{2x^1}\{(1+{}^x\Gamma_{11}{}^1)\partial_1+{}^x\Gamma_{11}{}^2\partial_2\},\\
&&\nabla_{\partial_1^u}\partial_2^u=-e^{x^1}\nabla_{\partial_1}\partial_2
=-e^{x^1}\{{}^x\Gamma_{12}{}^1\partial_1+{}^x\Gamma_{12}{}^2\partial_2\},\\
&&\nabla_{\partial_2^u}\partial_2^u=\nabla_{\partial_2}\partial_2
={}^x\Gamma_{22}{}^1\partial_1+{}^x\Gamma_{22}{}^2\partial_2\,.
\end{eqnarray*}
This implies that:
$$\begin{array}{ll}
{}^u\Gamma_{11}{}^1=-(1+{}^x\Gamma_{11}{}^1)\cdot e^{x^1}=0,&
{}^u\Gamma_{12}{}^1={}^x\Gamma_{12}{}^1\in\mathbb{R},\\
{}^u\Gamma_{22}{}^1=-{}^x\Gamma_{22}{}^1\cdot e^{-x^1}=0,&
{}^u\Gamma_{11}{}^2={}^x\Gamma_{11}{}^2\cdot e^{2x^1}=0,\\
{}^u\Gamma_{12}{}^2=-{}^x\Gamma_{12}{}^2\cdot e^{x^1}=0,&
{}^u\Gamma_{22}{}^2={}^x\Gamma_{22}{}^2\in\mathbb{R}.
\end{array}
$$
Thus $\alpha$ is unchanged and $\tilde{\mathcal{M}}:=(\mathbb{R}^+\times\mathbb{R},{}^u\Gamma)$
is isomorphic to $\mathcal{M}:=(\mathbb{R}^2,{}^x\Gamma)$. (This
shows, incidentally, that $(\mathbb{R}^2,{}^x\Gamma)$ is incomplete in this instance).
\subsection*{ Case~1. The surface $\mathcal{M}_1$}

 We may identify $\mathcal{M}_1$ with $\mathcal{M}_4^0$. We will discuss the surfaces
 $\mathcal{M}_4^c$ for more general $c$ subsequently.
 
\subsection*{ Case 2. The surfaces $\mathcal{M}_2^c$}  We may identify $\mathcal{M}_2^c$ with
 $\mathcal{M}_3^c$. Let $x=\Gamma_{12}{}^1$. We have
$\alpha=\frac{4(1+2x)^2}{x^2+x}$. This is symmetric about the line $x=-\frac12$. 
We note that $\alpha=0$
precisely when $\Gamma_{12}{}^2=-\frac12$; in this setting $\rho_{22}<0$. 

This is the
setting of Theorem~\ref{T3.8}~(4a) and there is only one  surface in this class. If
we assume $\alpha\ne0$, then $\alpha$ takes values in $(-\infty,0)\cup(16,\infty)$. There
are two possible values of $x$ (and two corresponding  surfaces). We make a linear
change of coordinates $x^1\rightarrow x^1-x^2$ and $x^2\rightarrow 2 x^2$ to have
$\mathfrak{K}(\mathcal{M})=\operatorname{Span}_{\mathbb{R}}
\{e^{x^1-x^2}\partial_1,e^{x^1+x^2}\partial_1\}\oplus\mathfrak{K}_0^{\mathcal{A}}$.
We have $e^{x^1\pm x^2}\partial_1$ are affine Killing vector fields if and only if the following equations
are satisfied:
$$\begin{array}{ll}
0=\Gamma_{11}{}^1-\Gamma_{11}{}^2+1,&
0=\Gamma_{11}{}^1+\Gamma_{11}{}^2+1,\\
0=\Gamma_{11}{}^2,&
0=\Gamma_{11}{}^2,\\
0=\Gamma_{11}{}^1-\Gamma_{12}{}^2+1,&
0=\Gamma_{11}{}^1-\Gamma_{12}{}^2+1,\\
0=\Gamma_{11}{}^2+\Gamma_{12}{}^2,&
0=\Gamma_{12}{}^2-\Gamma_{11}{}^2,\\
0=2 \Gamma_{12}{}^1-\Gamma_{22}{}^1-\Gamma_{22}{}^2+1,&
0=2 \Gamma_{12}{}^1+\Gamma_{22}{}^1-\Gamma_{22}{}^2-1,\\
0=\Gamma_{12}{}^2,&
0=\Gamma_{12}{}^2,
\end{array}$$
or equivalently
$$\Gamma_{11}{}^1=-1,\quad\Gamma_{11}{}^2=0,\quad
\Gamma_{12}{}^2=0,\quad\Gamma_{22}{}^1=1,\quad\Gamma_{22}{}^2=2\Gamma_{12}{}^1\,.
$$
We now have $\alpha=16x^2/(x^2-1)$ where $x=\Gamma_{12}{}^1$. 
The symmetry can now be realized by 
$x^2\rightarrow-x^2$, i.e. $\Gamma_{12}{}^1\rightarrow-\Gamma_{12}{}^1$.
 Thus $\alpha$ and $\epsilon$ completely detect the surfaces $\mathcal{M}_2^c$.

\subsection*{Case 3. The surface $\mathcal{M}_3^c$}  These
surfaces have been identified
with the surfaces $\mathcal{M}_2^c$ and dealt with in Case 2.

\subsection*{ Case 4. The surfaces $\mathcal{M}_4^c$} 
We have the relations
$$
{}^x\Gamma_{11}{}^1=0,\quad{}^x\Gamma_{11}{}^2=0,\quad{}^x\Gamma_{12}{}^1=1,
\quad{}^x\Gamma_{12}{}^2=0,\quad{}^x\Gamma_{22}{}^2=2\,.
$$
We have $\alpha=16$ and
$\mathfrak{K}(\mathcal{M})=\operatorname{Span}_{\mathbb{R}}\{x^2\partial_1,
({}^x\Gamma_{22}{}^1(x^2)^2+2x^1)\partial_1\}\oplus\mathfrak{K}_0^{\mathcal{A}}$.
The parameter $ c:={}^x\Gamma_{22}{}^1$ is undetermined. Let
$u^1=x^1+\frac12{}^x\Gamma_{22}{}^1(x^2)^2$ and $u^2=x^2$ be a  change of coordinates. We have
$$\begin{array}{ll}
du^1=dx^1+{}^x\Gamma_{22}{}^1x^2dx^2,&du^2=dx^2,\\
\partial_1^u=\partial_1,&\partial_2^u=-{}^x\Gamma_{22}{}^1x^2\partial_1+\partial_2.
\end{array}$$
We compute:
\smallbreak\qquad
$\nabla_{\partial_1^u}\partial_1^u=\nabla_{\partial_1}\partial_1
={}^x\Gamma_{11}{}^1\partial_1+{}^x\Gamma_{11}{}^2\partial_2=0$,
\smallbreak\qquad
$\nabla_{\partial_2^u}\partial_1^u=-{}^x\Gamma_{22}{}^1x^2({}^x\Gamma_{11}{}^1\partial_1
+{}^x\Gamma_{11}{}^2\partial_2)+{}^x\Gamma_{12}{}^1\partial_1+{}^x\Gamma_{12}{}^2\partial_2$
\smallbreak\qquad\qquad\qquad$=\partial_1=\partial_1^u$,
\smallbreak\qquad
$\nabla_{\partial_2^u}\partial_2^u=({}^x\Gamma_{22}{}^1x^2)^2\nabla_{\partial_1}\partial_1
-2{}^x\Gamma_{22}{}^1x^2\nabla_{\partial_1}\partial_2
-{}^x\Gamma_{22}{}^1\partial_1+\nabla_{\partial_2}\partial_2$
\smallbreak\qquad\qquad\qquad
$=0-2{}^x\Gamma_{22}{}^1x^2\partial_1-{}^x\Gamma_{22}{}^1\partial_1+{}^x\Gamma_{22}{}^1\partial_1
+{}^x\Gamma_{22}{}^2\partial_2=2\partial_2^u$.
\smallbreak\noindent Consequently,
$$\begin{array}{lll}
{}^u\Gamma_{11}{}^1=0,&{}^u\Gamma_{11}{}^2=0,&{}^u\Gamma_{12}{}^1=1,\\
{}^u\Gamma_{12}{}^2=0,&{}^u\Gamma_{22}{}^1=0,&{}^u\Gamma_{22}{}^2=2,
\end{array}$$
 so the surfaces $\mathcal{M}_1$ and $\mathcal{M}_4^c$ are equivalent for any $c$.

\subsection*{Case 5. The surfaces $\mathcal{M}_5^c$} We have the relations
$$\Gamma_{11}{}^1=-1,\quad\Gamma_{11}{}^2=0,\quad\Gamma_{12}{}^2=0,
\quad\Gamma_{22}{}^1=-1,\quad\Gamma_{22}{}^2=2\Gamma_{12}{}^1\,.
$$
We have $\alpha=16x^2/(1+x^2)$ takes values in $[0,16)$ where $x=\Gamma_{12}{}^1$.
If $\alpha=0$, there is only one surface
given by $\Gamma_{12}{}^1=0$
and we have $\rho_{22}=1+(\Gamma_{12}{}^1)^2=1$ corresponding to Theorem~\ref{T3.8}~(4b). If $\alpha\in(0,16)$, we have two surfaces
given by $\pm\Gamma_{12}{}^1$ and the symmetry is realized by $x^2\rightarrow-x^2$. We
have $\mathcal{M}$ is not of Type~$\mathcal{B}$ and $\mathfrak{K}(\mathcal{M})\approx A_{4,12}$, thus Assertion~2 follows.

This shows that $(\alpha,\epsilon)$ completely
determines the isomorphism type of $\mathcal{M}$ and
completes the proof of Theorem~\ref{T3.8}.
\end{proof}
We summarize our conclusions as follows: 
\medbreak
{\bf Table 1.} {\small Classification of homogeneous affine surfaces of Type~$\mathcal{A}$ with 
$\operatorname{Rank}\{\rho\}=1$.} Let $\kappa(M):=\dim\{\mathfrak{K}(M)\}$.
\smallbreak
\begin{center}
\begin{tabular}{|c|c|c|c|c|c|c|}
\hline
$\alpha$&$\epsilon$ &{  $\mathcal{M}$}&$\mathfrak{K}(\mathcal{M})$&
$\kappa(\mathcal{M})$&  Type~$\mathcal{A}$ & Type~$\mathcal{B}$\\
\hline
 $\alpha<0$&$-1$&{ $\mathcal{M}_2^c$, $\mathcal{M}_3^c$, $|c+\frac12|<\frac12$}&$A_2\oplus A_2$ &$4$& \checkmark&\checkmark\\
\hline
$\alpha=0$&$-1$&{ $\mathcal{M}_2^c$, $\mathcal{M}_3^c$, $c=-\frac12$}& $A_2\oplus A_2$&$4$& \checkmark&\checkmark\\
\hline
 $\alpha=0$&$+1$ &{ $\mathcal{M}_5^0$}&$A_{4,12}$ &$4$& \checkmark&No\\
\hline
$0<\alpha<16$&$+1$&{  $\mathcal{M}_5^c$, $c\neq 0$}& $ A_{4,12}$&$4$& \checkmark&No\\
\hline
 $\alpha=16$&$+1$&{ $\mathcal{M}_1$, $\mathcal{M}_4^c$, $c\in \mathbb{R}$}& $A_{4,9}^0$ &$4$&  \checkmark& \checkmark\\
\hline
 $16<\alpha$&$+ 1$&{ $\mathcal{M}_2^c$, $\mathcal{M}_3^c$, $\frac12<|c+\frac12|$}&$A_2\oplus A_2$ &$4$& \checkmark&\checkmark\\
\hline
\end{tabular}
\end{center}

\subsection{ Affine Killing vector fields on 
Type~$\mathcal{B}$ homogeneous  surfaces}\label{S3.2}
Linear equivalence for Type~$\mathcal{A}$ surfaces is the
action of $\operatorname{GL}(2,\mathbb{R})$. Linear equivalence for Type~$\mathcal{B}$
surfaces is a bit more subtle in view of Remark~\ref{R1.3}.
\begin{lemma}\label{L3.9}
Let $T_{b,c}:(x^1,x^2)\rightarrow(x^1,bx^1+cx^2)$ for $c\ne0$.
Let $C$ and $\tilde C$ define affine manifolds $\mathcal{M}$ and $\tilde{\mathcal{M}}$ of
Type~$\mathcal{B}$. Then $\mathcal{M}$ and $\tilde{\mathcal{M}}$ are linearly equivalent
if and only if there exists $T_{b,c}$ so $T_{b,c}^\ast C=\tilde C$.
\end{lemma}

\begin{proof} Let $\mathfrak{G}:=\{T:(x^1,x^2)\rightarrow(tx^1,ux^1+vx^2+w)\}$
for $t>0$ and $v\ne0$ be the 4-dimensional subgroup of $\operatorname{GL}(2,\mathbb{R})$
which preserves $\mathbb{R}^+\times\mathbb{R}$. Then by definition,
$\mathcal{M}$ is linearly equivalent
to $\tilde{\mathcal{M}}$ if and only if there exists $T\in\mathfrak{G}$ so that
$T^\ast C=\tilde C$. There are two non-Abelian subgroups of $\mathfrak{G}$ which play an
important role. Set
$$
\mathfrak{H}:=\{S:(x^1,x^2)\rightarrow(ax^1,ax^2+b)\text{ for }a>0\}
\text{ and }
\mathfrak{I}:=\{T_{b,c}\}\,.
$$
The subgroups $\mathfrak{H}$ and $\mathfrak{I}$ generate $\mathfrak{G}$ as a Lie group.
By Remark~\ref{R1.3}, $\mathfrak{H}$ preserves Type~$\mathcal{B}$ structures. Thus
only the action of $\mathfrak{I}$ is relevant in studying linear equivalence for Type $\mathcal{B}$
structures and the Lemma follows.
\end{proof} 

The Lie group $\mathfrak{I}$ plays the crucial role in studying linear equivalence for Type $\mathcal{B}$ structures; the shear $(x^1,x^2)\rightarrow(x^1,\varepsilon x^1+x^2)$ and
the rescaling $(x^1,x^2)\rightarrow(x^1,cx^2)$ for $c\ne0$ generate $\mathfrak{I}$ and will play a central role
in what follows. The group $\mathfrak{H}$ also plays an important role.
 Let $\mathfrak{K}_0^{\mathcal{B}}$ be the Lie algebra of $\mathfrak{H}$. Then
 $$
\mathfrak{K}_0^{\mathcal{B}}:=\operatorname{Span}\{x^1\partial_1+x^2\partial_2,\partial_2\}
\subset\mathfrak{K}(\mathcal{M})\text{ for any }\mathcal{M}\in\mathcal{F}^{\mathcal{B}}\,.
$$
 This non-Abelian Lie sub-algebra
plays the same role in the analysis of Type~$\mathcal{B}$ surfaces that $\mathfrak{K}_0^{\mathcal{A}}$
played in the analysis of Type~$\mathcal{A}$ surfaces. 

Let $\mathfrak{su(1,1)}$ be the Lie algebra
of $SU(1,1)$ or, equivalently, of $SL(2,\mathbb{R})$. It is the Lie algebra on 3 generators (also denoted by $A_{3,8}$ in \cite{PSWZ76}) satisfying
the relations:
\begin{equation}\label{E3.d}
[e_1,e_2]=e_1,\quad[e_2,e_3]=e_3,\quad[e_1,e_3]=-2e_2\,.
\end{equation}

\begin{definition}\label{D3.10}\rm For $c\ge0$, set
\smallbreak$\mathcal{N}_1^\pm:=\mathcal{M}(
C_{11}{}^1=-\frac32$, $C_{11}{}^2=0$, $C_{12}{}^1=0$, 
$C_{12}{}^2=-\frac12$, $C_{22}{}^1=\mp\frac12$,
$C_{22}{}^2=0$).
\smallbreak $\mathcal{N}_2^c:=\mathcal{M}(C_{11}{}^1=-\frac32$, 
$C_{11}{}^2=0$, $C_{12}{}^1=1$, $C_{12}{}^2=-\frac12$, 
$C_{22}{}^1=\phantom{A.}c$, $C_{22}{}^2=2$).
\smallbreak $\mathcal{N}_3:=\mathcal{M}(C_{11}{}^1=-1$, $C_{11}{}^2=0$, $C_{12}{}^1=0$, $C_{12}{}^2=-1$, $C_{22}{}^1=\phantom{}-1$, $C_{22}{}^2=0$).
\smallbreak $\mathcal{N}_4:=\mathcal{M}(C_{11}{}^1=-1$, $C_{11}{}^2=0$, $C_{12}{}^1=0$, $C_{12}{}^2=-1$, $C_{22}{}^1=\phantom{Aa}1$,
$C_{22}{}^2=0$).
\medbreak\noindent
We show that these surfaces are not flat and thus $\mathcal{N}_\star^\star$ is Type~$\mathcal{B}$
 by computing:
\begin{equation}\label{E3.e}
\begin{array}{l}
\rho(\mathcal{N}_1^\pm)=\pm(x^1)^{-2}dx^2\otimes dx^2,\\[0.05in]
\rho(\mathcal{N}_2^c)=(x^1)^{-2}\{\textstyle\frac32(dx^1\otimes dx^2-dx^2\otimes dx^1)+(1-2c)dx^2\otimes dx^2\},\\[0.05in]
\rho(\mathcal{N}_3)=(x^1)^{-2}(-dx^1\otimes dx^1+dx^2\otimes dx^2),\\[0.05in]
\rho(\mathcal{N}_4)=(x^1)^{-2}(-dx^1\otimes dx^1-dx^2\otimes dx^2)\,.
\end{array}\end{equation}
\end{definition}

The main result of this section is the following:
\begin{theorem}\label{T3.11}
 If $\mathcal{M}\in\mathcal{F}^{\mathcal{B}}$, then 
$2\le\dim\{\mathfrak{K}(\mathcal{M})\}\le4$.
\begin{enumerate}
\item If $\dim\{\mathfrak{K}(\mathcal{M})\}=4$, then
$\rho=(x^1)^{-2}\tilde\rho_{11}dx^1\otimes dx^1$, $C_{12}{}^1=C_{22}{}^1=C_{22}{}^2=0$,
$\mathcal{M}$ is also of Type~$\mathcal{A}$,
and up to linear equivalence one of the following 3 possibilities holds:
\begin{enumerate}
\item $C_{11}{}^1-2C_{12}{}^2=0$, $C_{11}{}^2=1$, 
$\tilde\rho_{11}=(1+C_{12}{}^2)C_{12}{}^2\ne0$, and\newline
$\mathfrak{K}(\mathcal{M})=\operatorname{Span}\{x^1\partial_1-x^1\log(x^1)\partial_2,x^1\partial_2\}\oplus\mathfrak{K}_0^{\mathcal{B}}$.
\item $C_{11}{}^2=0$, 
$\tilde\rho_{11}=(1+C_{11}{}^1-C_{12}{}^2)C_{12}{}^2\ne0$, and\newline
$\mathfrak{K}(\mathcal{M})=\operatorname{Span}\{x^1\partial_1,(x^1)^a\partial_2\}
\oplus\mathfrak{K}_0^{\mathcal{B}}$ for some $a\ne0$.
\item $C_{11}{}^2=0$, $\tilde\rho_{11}=(C_{12}{}^2)^2\ne0$, and\newline
$\mathfrak{K}(\mathcal{M})=\operatorname{Span}\{x^1\partial_1,\log(x^1)\partial_2\}\oplus\mathfrak{K}_0^{\mathcal{B}}$.
\end{enumerate}
\item If $\dim\{\mathfrak{K}(\mathcal{M})\}=3$, then 
$\mathfrak{K}(\mathcal{M})=\operatorname{Span}\{X(\sigma)\}\oplus\mathfrak{K}_0^{\mathcal{B}}\approx \mathfrak{su(1,1)}$ where
$$X(\sigma):=2x^1x^2\partial_1+\{(x^2)^2+\sigma\cdot(x^1)^2\}\partial_2\,\,\,\text{ for }\,\,\,\sigma\in\{-1,0,1\},$$
$\mathcal{M}$ is not of Type~$\mathcal{A}$, and
up to linear equivalence, one of the following possibilities holds:
\begin{enumerate}
\smallbreak\item $\sigma=0$, $\mathcal{M}=\mathcal{N}_1^\pm$, and
$\mathcal{M}$ is not Type~$\mathcal{C}$.
\smallbreak\item $\sigma=0$, $\mathcal{M}=\mathcal{N}_2^c$, and
$\mathcal{M}$ is not Type~$\mathcal{C}$.
\smallbreak\item $\sigma=1$, $\mathcal{M}=\mathcal{N}_3$, and
$\mathcal{M}$ is Type~$\mathcal{C}$.
\smallbreak\item $\sigma=-1$, $\mathcal{M}=\mathcal{N}_4$, and
$\mathcal{M}$ is Type~$\mathcal{C}$.
\end{enumerate}
\smallbreak\item For each of the 3 structures listed in Assertion~1, 
$\dim\{\mathfrak{K}(\mathcal{M})\}=4$.
For each of the 4 structures listed in Assertion~2, $ \dim\{\mathfrak{K}(\mathcal{M})\}=3$.
\end{enumerate}\end{theorem}

\begin{remark} If $\mathcal{M}$ is Type~$\mathcal{C}$,
then $ \dim\{\mathfrak{K}(\mathcal{M})\}=3$, $\rho=\rho^s$, and $\operatorname{Rank}\{\rho\}=2$.
Thus if $\mathcal{M}\in\mathcal{F}^{\mathcal{B}}$ then $\mathcal{M}$ is also of Type~$\mathcal{C}$
if and only if either Assertion~2c or Assertion~2d of Theorem~\ref{T3.11} holds.
Similarly, $\mathcal{M}$ is both Type~$\mathcal{B}$ and Type~$\mathcal{A}$
if and only if Assertion~1 of Theorem~\ref{T3.11} holds.
\end{remark}

The proof of this result will occupy most of this section 
 and will be a direct consequence of the following lemmas. It gives a complete
description of those homogeneous
affine surfaces of Type~$\mathcal{B}$ with $\dim\{\mathfrak{K}(\mathcal{M})\}>2$. 
If $X\in C^\infty(TM)$ is a smooth vector field on $M$, let
\begin{equation}\label{E3.f}
\Theta:=\ad(x^1\partial_1+x^2\partial_2)\text{ i.e. } \Theta(X):=[x^1\partial_1+x^2\partial_2,X]\,.
\end{equation}
\begin{lemma}\label{L3.13}
Let $X\in C^\infty(TM)$ be polynomial in $(x^1,x^2)$. If
$X$ is homogeneous of degree $\ell$, then $\Theta(X)=(\ell-1)X$.
\end{lemma}

\begin{proof} Let $X=\displaystyle\sum_{i+j=\ell}(x^1)^i(x^2)^j(c_{i,j}^1\partial_1+c_{i,j}^2\partial_2)$ for $c_{i,j}^\nu\in\mathbb{R}$. Then:
\begin{eqnarray*}
&&[x^1\partial_1,X]=\sum_{i+j=\ell} (x^1)^i(x^2)^j\{(i-1)c_{i,j}^1\partial_1+ic_{i,j}^2\partial_2\},\\
&&[x^2\partial_2,X]=\sum_{i+j=\ell} (x^1)^i(x^2)^j\{jc_{i,j}^1\partial_1+(j-1)c_{i,j}^2\partial_2\}\,.
\end{eqnarray*}
We add these two expressions to see $\Theta(X)=(\ell-1)X$.
\end{proof}

The following result is an 
analogue of Lemma~\ref{L3.3}; Assertion~2a (resp. Assertion~2b)
will give rise to Assertion~1 (resp. Assertion~2) of Theorem~\ref{T3.11}.

\begin{lemma}\label{L3.14}
Let $\mathcal{M}\in\mathcal{F}^{\mathcal{B}}$. Suppose that $\dim\{\mathfrak{K}(\mathcal{M})\}>2$.
\begin{enumerate}
\item If $X\in\mathfrak{K}(\mathcal{M})$, then $X$ is polynomial in $x^2$, i.e.
$X=\displaystyle\sum_{k=0}^n(x^2)^kX_k(x^1)$. 
\item Choose $n=n(X)$ minimal so $X\in\mathfrak{K}(\mathcal{M})-\mathfrak{K}_0^{\mathcal{B}}$
has the form of Assertion~1. Then one of the following two possibilities holds:
\begin{enumerate}
\item $n=0$ and 
$X=a_1(x^1)\partial_1+a_2(x^1)\partial_2$.
\item $n=2$. By making a change of coordinates $(x^1,x^2)\rightarrow(x^1,\alpha x^1+\beta x^2)$, we can
ensure $X=2x^1x^2\partial_1+\{(x^2)^2+\sigma\cdot(x^1)^2\}\partial_2$ for $\sigma\in\{-1,0,1\}$.
\end{enumerate}\end{enumerate}\end{lemma}

\begin{proof} We use the structure of $\mathfrak{K}(\mathcal{M})$ as a $\mathfrak{K}_0^{\mathcal{B}}$
module. 
Let $X\in\mathfrak{K}(\mathcal{M})$. Since $\operatorname{ad}(\partial_2)=\partial_2$ is an
endomorphism of $\mathfrak{K}(\mathcal{M})$ and since $\mathfrak{K}(\mathcal{M})$ is finite dimensional,
there is a minimal dependence relation of the form:
$$
\partial_2^{s}X+c_{s-1}\partial_2^{s-1}X+\dots+c_0X=0\text{ with }s>0\text{ and }c_i\in\mathbb{R}\,.
$$
We factor this relation over $\mathbb{C}$ to construct a relation
$$
\prod_{v=1}^w(\partial_2-\lambda_v)^{\nu_v}X=0\text{ for }\lambda_v\in\mathbb{C}\text{ distinct and }\nu_v\ge1\,.
$$
We clear the previous notation and let $\mathfrak{L}$ (resp. $\mathfrak{L}_0^{\mathcal{B}}$)
be the complexification of
$\mathfrak{K}(\mathcal{M})$ (resp. $\mathfrak{K}_0^{\mathcal{B}}$).
Suppose some $\lambda_v\ne0$. By
reordering the roots, we may assume $\lambda_1\ne0$.
Since we have chosen a minimal dependence relation, we have
$$0\ne Y:=(\partial_2-\lambda_1)^{\nu_1-1}\prod_{v=2}^w(\partial_2-\lambda_v)^{\nu_v}X\in\mathfrak{L}\,.$$
Since $(\partial_2-\lambda_1)Y=0$,
$Y=e^{\lambda_1 x^2}Y_0(x^1)\in\mathfrak{L}$.
Since $\mathfrak{L}$ is finite dimensional, we may
choose $Z\in\mathfrak{L}$ for $n$ maximal of the form
$$
0\ne Z=e^{\lambda_1x^2}\sum_{k=0}^n(x^2)^kZ_k(x^1)\text{ for }Z_n(x^1)\text{ not identically zero}\,.
$$
We then have $0\ne\Theta(Z)=(x^2)^{n+1}\lambda_1e^{\lambda_1 x^2} Z_n(x^1)+O((x^2)^n)\in\mathfrak{L}$ which
contradicts the assumption that $n$ was maximal. 
Thus terms which are true exponentials in $x^2$ do not occur and the minimal relation for $X$
takes the form $(\partial_2)^nX=0$. This implies that $X$ is polynomial in $x^2$ and establishes
Assertion~1.

We now establish Assertion~2. Choose $X\in\mathfrak{K}(\mathcal{M})-\mathfrak{K}_0^{\mathcal{B}}$
so that $n=n(X)$ is minimal.
If $n=0$, then $X=X(x^1)$ and Assertion~2a holds.
We suppose therefore that $n>0$. One has
$0\ne(\partial_2)^nX= n!X_n(x^1)\in\mathfrak{K}(\mathcal{M})$. Because $n$ was minimal,
$X_n(x^1)\in\mathfrak{K}_0^{\mathcal{B}}$. Since $X_n(x^1)$ does not depend on $x^2$,
$X_n(x^1)$ is a constant multiple of $\partial_2$.
Therefore after rescaling $X$ if necessary, we may assume
$$X=(x^2)^n\partial_2+\sum_{k=0}^{n-1}(x^2)^kX_k(x^1)
\in\mathfrak{K}(\mathcal{M})-\mathfrak{K}_0^{\mathcal{B}}\,.$$
If $n>2$, then $\partial_2X\in\mathfrak{K}(\mathcal{M})$ has degree at least $2$
in $x^2$ so $\partial_2X\notin\mathfrak{K}_0^{\mathcal{B}}$. This contradicts the minimality of $n$
and shows $n=1$ or $n=2$.
If $n=1$, then 
$\tilde X(x^1):=X-x^1\partial_1-x^2\partial_2\in\mathfrak{K}(\mathcal{M})-\mathfrak{K}_0^{\mathcal{B}}$ 
is independent of $x^2$ and therefore has $n(\tilde X)=0$; this
contradicts the minimality of $n$. Thus $n=2$ so
$$
X=(x^2)^2\partial_2+x^2X_1(x^1)+X_0(x^1)\in\mathfrak{K}(\mathcal{M})-\mathfrak{K}_0^{\mathcal{B}}\,.
$$
We note that
$$
\tilde X^1(x^1):=\partial_2X-2(x^1\partial_1+x^2\partial_2)=X_1(x^1)-2x^1\partial_1\in\mathfrak{K}(\mathcal{M})\,.
$$
Thus by the minimality of $n$, $Y:=X_1(x^1)-2x^1\partial_1\in\mathfrak{K}_0^{\mathcal{B}}$. Since
$Y=Y(x^1)$,
$Y=c_2\partial_2$. By replacing $X$ by $X-c_2(x^1\partial_1+x^2\partial_2)$,
we may assume $X_1=2x^1\partial_1$ so
$$
X=(x^2)^2\partial_2+2x^1x^2\partial_1+X_0(x^1)
\in\mathfrak{K}(\mathcal{M})-\mathfrak{K}_0^{\mathcal{B}}\,.
$$
Since $(x^2)^2\partial_2+2x^1x^2\partial_1$ is homogeneous of degree 2,
 Lemma~\ref{L3.13} implies
$$
Y:=(\Theta-1)X=(x^1\partial_1-2)X_0(x^1)\in\mathfrak{K}(\mathcal{M})\,.
$$
The minimality of $n$ then shows $Y\in\mathfrak{K}_0^{\mathcal{B}}$ so
$(x^1\partial_1-2)X_0(x^1)=\epsilon\partial_2$. By subtracting an appropriate multiple
of $\partial_2$ from $X$ we can assume $\epsilon=0$ so
$$
(x^1\partial_1-2)X_0(x^1)=0\,.
$$
We can solve this ODE to see $X_0$ is homogeneous of degree $2$ in $x^1$ and, consequently,
$$X=\{2x^1x^2+(x^1)^2 c_1\}\partial_1 +\{(x^2)^2+(x^1)^2c_2\}\partial_2\,\,\,\,
\text{for some}\ \ (c_1,c_2)\in\mathbb{R}^2.$$

We consider the linear change of coordinates $(u^1,u^2)=(x^1,\epsilon x^1+x^2)$. 
Then
\begin{eqnarray*}
&&\partial_1^u=\partial_1^x-\epsilon\partial_2^x,\quad\partial_2^u=\partial_2^x,\\
&&X=\{2x^1x^2+(x^1)^2c_1\}\partial_1^x+\{(x^2)^2+(x^1)^2c_2\}\partial_2^x\\\
&&\phantom{X}=\{2u^1(u^2-\epsilon u^1)+(u^1)^2c_1\}(\partial_1^u+\epsilon\partial_2^u)+
\{(u^2-\epsilon u^1)^2+(u^1)^2c_2\}\partial_2^u\\
&&\phantom{X}=\{2u^1u^2+(c_1-2\epsilon)(u^1)^2\}\partial_1^u+\{(u^2)^2+(u^1)^2c_3(\epsilon)\}\partial_2^u\text{ for }c_3(\epsilon)\in\mathbb{R}\,.
\end{eqnarray*}
Thus by choosing $\epsilon=\frac12c_1$, we may assume $c_1=0$ to assume
$$X=2x^1x^2\partial_1+\{(x^2)^2+(x^1)^2 c_3\}\partial_2\,.$$
We may now rescale $x^1$ to ensure $ c_3\in\{-1,0,1\}$.
\end{proof}

We continue our study firstly assuming the existence of affine Killing vector fields as in Assertion 2a of Lemma \ref{L3.14}.  The condition $C_{12}{}^1=C_{22}{}^1=C_{22}{}^2=0$ will play a crucial
role in our analysis; by Lemma~\ref{L2.10}, it is an affine invariant in this setting. The surfaces
of Assertion~1 of Theorem~\ref{T3.11} will arise as follows:
\begin{lemma}\label{L3.15}
Let $\mathcal{M}\in\mathcal{F}^{\mathcal{B}}$.
\begin{enumerate}
\item $x^2\partial_1\notin\mathfrak{L}(\mathcal{M})$ and $(x^1)^a\partial_1\in\mathfrak{L}(\mathcal{M})$ iff $a=1$, $C_{22}{}^2=0$,
$C_{12}{}^1=0$, $C_{22}{}^1=0$, and
$C_{11}{}^2=0$. 
\item 
 $(x^1)^a\partial_2\in\mathfrak{L}(\mathcal{M})-\mathfrak{L}_0^{\mathcal{B}}$ 
 iff $a:=1+C_{11}{}^1-2C_{12}{}^2\ne0$, $C_{12}{}^1=0$, $C_{22}{}^1=0$, $C_{22}{}^2=0$,
\item $\log(x^1)\partial_2\in\mathfrak{L}(\mathcal{M})$ iff 
$1+C_{11}{}^1-2C_{12}{}^2=0$, $C_{12}{}^1=0$, $C_{22}{}^1=0$, $C_{22}{}^2=0$.
\item If there exists 
$X=X(x^1)\in\mathfrak{K}(\mathcal{M})-\mathfrak{K}_0^{\mathcal{B}}$,
 then $C_{12}{}^1=C_{22}{}^1=C_{22}{}^2=0$.
\item If $C_{12}{}^1=C_{22}{}^1=C_{22}{}^2=0$, then $\mathcal{M}$ is also of Type~$\mathcal{A}$, $\dim\{\mathfrak{K}(\mathcal{M})\}=4$, $\rho=(x^1)^{-2}\tilde\rho_{11}dx^1\otimes dx^1$,
and one of the following holds:
\begin{enumerate}
\item $C_{11}{}^1-2C_{12}{}^2=0$, $C_{11}{}^2=1$, 
$\tilde\rho_{11}=(1+C_{12}{}^2)C_{12}{}^2\ne0$, and\newline
$\mathfrak{K}(\mathcal{M})=\operatorname{Span}\{x^1\partial_1-x^1\log(x^1)\partial_2,x^1\partial_2\}\oplus\mathfrak{K}_0^{\mathcal{B}}$.
\item $C_{11}{}^2=0$, $a\ne0$, 
$\tilde\rho_{11}=(1+C_{11}{}^1-C_{12}{}^2)C_{12}{}^2\ne0$, and\newline
$\mathfrak{K}(\mathcal{M})=\operatorname{Span}\{x^1\partial_1,(x^1)^a\partial_2\}
\oplus\mathfrak{K}_0^{\mathcal{B}}$.
\item $C_{11}{}^2=0$, $a=0$, $\tilde\rho_{11}=(C_{12}{}^2)^2\ne0$, and\newline
$\mathfrak{K}(\mathcal{M})=\operatorname{Span}\{x^1\partial_1,\log(x^1)\partial_2\}\oplus\mathfrak{K}_0^{\mathcal{B}}$.
\end{enumerate}
\end{enumerate}
\end{lemma}

\begin{proof} The first three assertions follow by direct computation.
We prove Assertion~4 as follows.
Suppose that $X=a_1(x^1)\partial_1+a_2(x^1)\partial_2\in
\mathfrak{K}(\mathcal{M})-\mathfrak{K}_0^{\mathcal{B}}$.
Let $\Theta$ be as defined in Equation~(\ref{E3.f}).
Because $\Theta X=(x^1\partial_1-1)X$, $x^1\partial_1X\in\mathfrak{K}(\mathcal{M})$. We factor the
minimal dependence relation
$$(x^1\partial_1)^nX+c_{n-1}(x^1\partial_1)^{n-1}X+\dots+c_0X=0\text{ for }n\ge1$$
over $\mathbb{C}$ to express this relation in the form
$$
\prod_{v=1}^s(x^1\partial_1-\lambda_\nu)^{\nu_v}X=0\,.
$$
Suppose some $\lambda_\nu\ne0$. By renumbering the roots, we may suppose $a:=\lambda_1\ne0$ so
$$
0\ne Y=(x^1\partial_1-a)^{\nu_1-1}\prod_{v=2}^2(x^1\partial_1-\lambda_\nu)^{\nu_v}X\in\mathfrak{L}
$$
satisfies the ODE $(x^1\partial_1-a)Y=0$ and hence
$Y=(x^1)^a(c_1\partial_1+c_2\partial_2)\in\mathfrak{L}-\mathfrak{L}_0$.
If $c_1\ne0$, by making a (possibly) complex change of coordinates which takes the form
$(x^1,x^2)\rightarrow(x^1,x^2+\epsilon x^1)$,
we may assume $Y=(x^1)^a\partial_1$.  Assertion~1 then implies $a=1$. Therefore,
we may take $Y$ to be real and the change of coordinates involved is real. The
relations of Assertion~4 then follow from Assertion~1. 
If, on the other hand, $c_1=0$, then $Y=(x^1)^a\partial_2$ and we use Assertion~2 to show Assertion~4 holds.

We may therefore assume the minimal relation takes the form $(x^1\partial_1)^nX=0$ and
we do not need to complexity. If
$n=1$, then $X$ is constant. Since $X\in\mathfrak{K}-\mathfrak{K}_0$, we may assume
$X=\partial_1$; this is ruled out by Assertion~1.  We therefore conclude that $n>1$. 
By replacing $X$ by $(x^1\partial_1)^{n-2} X$ if necessary, we may assume without
loss of generality that $n=2$. Since 
$$ 0\ne(x^1\partial_1)X\in\ker\{x^1\partial_1\}=\operatorname{Span}\{ \partial_1,\partial_2\}\cap\mathfrak{K}(\mathcal{M})\,,
$$
again by making a change of coordinates of the form $(x^1,x^2)\rightarrow(x^1,x^2+\epsilon x^1)$ if necessary, we may assume $x^1\partial_1X=\partial_2$.
 We solve this ODE to see
$$
 X=\log(x^1)\partial_2+X_0\text{ for }X_0=c_1\partial_1+c_2\partial_2\text{ a constant
vector field}\,.
$$
If $c_1\ne1$, we renormalize the coordinates so $X_0=\partial_1$ and obtain
Killing equations:
$$\begin{array}{ll}
 2C_{12}{}^1-C_{11}{}^1=0,& 2C_{12}{}^2-C_{11}{}^1-C_{11}{}^2 -1=0,\\
C_{22}{}^1-C_{12}{}^1=0,&C_{22}{}^2-C_{12}{}^1-C_{12}{}^2=0,\\
C_{22}{}^1=0,&C_{22}{}^2+C_{22}{}^1=0.
\end{array}$$
This implies $C_{11}{}^1=0$, $C_{11}{}^2=-1$, $C_{12}{}^1=0$, $C_{12}{}^2=0$, $C_{22}{}^1=0$,
$C_{22}{}^2=0$, and $\rho=0$. This is impossible. 
We therefore have $X=\log(x^1)\partial_2$ and the relations of Assertion~4 follow from Assertion~3.

We impose the relations $C_{12}{}^1=C_{22}{}^1=C_{22}{}^2=0$ for the remainder of the proof.
If $\mathcal{M}$ is also of Type~$\mathcal{A}$, then Theorem~\ref{T3.4}
implies $\dim\{\mathfrak{K}(\mathcal{M})\}=4$. Thus our task is to construct two additional
 affine Killing vector fields $\xi_1$ and $\xi_2$ so $\{\xi_1,\xi_2,x^1\partial_1+x^2\partial_2,\partial_2\}$
are linearly independent and so we can apply Lemma~\ref{L3.2}.

If $C_{11}{}^2=0$, then
$x^1\partial_1\in\mathfrak{K}(\mathcal{M})$ by Assertion~1. We apply Lemma~\ref{L3.2} to the
pair $\{x^1\partial_1,\partial_2\}$ to see $\mathcal{M}$ is also of Type~$\mathcal{A}$ and
hence by Theorem~\ref{T3.4}, $\dim\{\mathfrak{K}(\mathcal{M})\}=4$. 
We set $a=1+C_{11}{}^1-2C_{12}{}^2$. We apply  Assertion~2 of Lemma~\ref{L3.15} to obtain  Assertion~5b if $a\ne0$
and to obtain Assertion~5c if $a=0$.

If $C_{11}{}^1-2C_{12}{}^2\ne0$, Assertion~2 of Lemma \ref{L2.8} shows 
there is a linear change of coordinates, which does
not affect the normalization $C_{12}{}^1=C_{22}{}^1=C_{22}{}^2=0$, to  ensure $C_{11}{}^2=0$.
The analysis of the previous paragraph then pertains. We may therefore assume $ C_{11}{}^1-2C_{12}{}^2=0$
and $C_{11}{}^2\ne0$. By rescaling $x^2$, we may assume $C_{11}{}^2=1$. 
We apply Assertion~2 with $a=1$ to see $x^1\partial_2\in\mathfrak{K}(\mathcal{M})$. 
A direct computation shows $x^1\partial_1-x^1\log(x^1)\partial_2\in\mathcal{M}$. We apply Lemma~\ref{L3.2}
to the pair $\{x^1\partial_1-x^1\log(x^1)\partial_2,\partial_2\}$ to see $\mathcal{M}$ is of Type~$\mathcal{A}$
and hence by Theorem~\ref{T3.4}
$\dim\{\mathfrak{K}(\mathcal{M})\}=4$. We then obtain Assertion~5a.
\end{proof}

\begin{remark}
\rm\ \begin{enumerate}
\item Observe that $C_{12}{}^1=C_{22}{}^1=C_{22}{}^2=0$ in Assertion~4 of
Lemma \ref{L3.15}
 is an equivalent condition for a Type $\mathcal{B}$  surface to be also of 
 Type $\mathcal{A}$ (compare with the results in \cite{CGV10}).
\item We apply Lemma~\ref{L3.9} to the three classes in Assertion~5 of Lemma~\ref{L3.15}.
Let $C$ define such a connection. Then $C$
transforms to a new connection $\tilde C=T_{b,c}^\ast C$  
for $\tilde C_{11}{}^2=\frac{1}{c}(C_{11}{}^2+b(2 C_{12}{}^2-C_{11}{}^1))$ and 
$\tilde C_{ij}{}^k=C_{ij}{}^k$ otherwise. It now follows 
that the three classes in Assertion~5 of Lemma~\ref{L3.15} are linearly inequivalent surfaces since
it is not possible to transform one class into another using a transformation $T_{b,c}$.
\item The $\alpha$ invariant satisfies $\alpha(\mathcal{M})\in (-\infty,0]\cup (16,\infty)$ if 
$\mathcal{M}$ corresponds to the families 5a and 5b. This shows that these families
 are affine isomorphic to $\mathcal{M}_2^c$, whereas $\alpha(\mathcal{M})=16$
  for any surface in 5c, and thus they are  affine isomorphic to $M_1$. 
\end{enumerate}\end{remark}

Next we assume Assertion~2b of Lemma~\ref{L3.14} holds.
 This will give rise to Assertion~2 of Theorem~\ref{T3.11}.

\begin{lemma}
Let $\mathcal{M}\in\mathcal{F}^{\mathcal{B}}$. Assume there exists $X\in\mathfrak{K}(\mathcal{M})$ of
the form
$$X(\sigma):=2x^1x^2\partial_1+\{(x^2)^2+\sigma\cdot(x^1)^2\}\partial_2\,\,\,\text{ for }\,\,\,\sigma\in\{-1,0,1\}\,.$$
\begin{enumerate}
\item Up to linear equivalence, one of the following possibilities holds:
\begin{enumerate}
\smallbreak\item $\sigma=0$, $\mathcal{M}=\mathcal{N}_1^\pm$,
$\mathcal{M}$ is not Type~$\mathcal{C}$, and
$\rho=\pm(x^1)^{-2}dx^2\otimes dx^2$
\smallbreak\item $\sigma=0$, $\mathcal{M}=\mathcal{N}_2^c$, and
$\mathcal{M}$ is not Type~$\mathcal{C}$.
\smallbreak\item $\sigma=1$, $\mathcal{M}=\mathcal{M}_3$, and
$\mathcal{M}$ is Type~$\mathcal{C}$.
\smallbreak\item $\sigma=-1$, $\mathcal{M}=\mathcal{N}_4$, and
$\mathcal{M}$ is Type~$\mathcal{C}$.
\end{enumerate}
\smallbreak\item $\mathfrak{K}(\mathcal{M})=\operatorname{Span}\{X(\sigma)\}\oplus\mathfrak{K}_0^{\mathcal{B}}\approx \mathfrak{su(1,1)}$. 
\item $\dim\{\mathfrak{K}(\mathcal{M})\}=3$, and
$\mathcal{M}$ is not of Type~$\mathcal{A}$. 
\item Two different affine surfaces in Definition~\ref{D3.10} are not locally affine isomorphic.
In particular, linearly equivalent and affine isomorphic are equivalent notions in this setting.
\end{enumerate}
\end{lemma}

\begin{proof} 
Suppose first $\sigma=0$. The Killing equations are 
$$\begin{array}{llll}
C_{11}{}^2=0,&
C_{11}{}^1-C_{12}{}^2+1=0,&
2 C_{12}{}^1-C_{22}{}^2=0,&2 C_{12}{}^2+1=0.
\end{array}$$
We solve these equations to see that
$$
\textstyle C_{11}{}^1=-\frac32,\quad C_{11}{}^2=0,\quad
C_{12}{}^2=-\frac12,\quad C_{22}{}^2=2C_{12}{}^1\,.
$$
If $C_{12}{}^1=0$, we may rescale $x^2$ to ensure that
$C_{22}{}^1=\mp\frac12$ and obtain the  surfaces $\mathcal{M}_1^\pm$
and compute that $\rho= \pm (x^1)^{-2} dx^2\otimes dx^2$ so $\rho^a=0$. 
The nature of the Ricci tensor (see Equation~(\ref{E3.e})) distinguishes these
two surfaces.

On the other hand, if $C_{12}{}^1\ne0$, we  may rescale $x^2$ 
to assume $C_{12}{}^1=1$ and obtain the surfaces $\mathcal{N}_2^c$. 
We then have $\rho^a=\frac32 (x^1)^{-2}dx^1\wedge dx^2$ is invariantly defined. In particular, none of these  surfaces is locally isomorphic to 
$\mathcal{N}_1^\pm$.
If we express $\nabla\rho^a=\omega\otimes\rho^a$, then $\omega$ is invariantly defined. 
We compute
\begin{eqnarray*}
&&\nabla_{\partial_1}\rho^a=(x^1)^{-3}\{-2-C_{11}{}^1-C_{12}{}^2\}dx^1\wedge dx^2=0,\\
&&\nabla_{\partial_2}\rho^a
= {\textstyle\frac32}(x^1)^{-3}\{-C_{21}{}^1-C_{22}{}^2\}dx^1\wedge dx^2=-{\textstyle\frac92}
(x^1)^{-3}dx^1\wedge dx^2.
\end{eqnarray*}
This shows $\tilde\omega:=(x^1)^{-1}dx^2$ is invariantly defined. Thus by expressing
 $$
\rho^s=\{1-2c\}(x^1)^{-2}dx^2\otimes dx^2=\{1-2c\}\tilde\omega\otimes\tilde\omega\,,
 $$
we conclude $1-2c$ is an affine invariant and hence all these examples are distinct as well.

Suppose next $\sigma=1$. The Killing equations are
$$
\begin{array}{ll}
2 C_{12}{}^1- C_{11}{}^2=0,&
2 C_{12}{}^2- C_{11}{}^1+1 =0,\\
C_{11}{}^1-C_{12}{}^2+C_{22}{}^1+1=0,&
C_{11}{}^2-C_{12}{}^1+C_{22}{}^2=0,\\
2 C_{12}{}^1- C_{22}{}^2=0,&2 C_{12}{}^2- C_{22}{}^1+1=0.
\end{array}$$
We solve these relations to see $\mathcal{M}=\mathcal{N}_3$. The symmetric
Ricci tensor distinguishes this surface from the  surfaces $\mathcal{N}_1^\pm$
or $\mathcal{N}_2^c$.

Suppose finally $\sigma=-1$. The Killing equations are
$$
\begin{array}{ll}
C_{11}{}^2+2 C_{12}{}^1=0,&
C_{11}{}^1-2 C_{12}{}^2-1=0,\\
C_{11}{}^1-C_{12}{}^2-C_{22}{}^1+1=0,&
C_{11}{}^2+C_{12}{}^1-C_{22}{}^2=0,\\
2 C_{12}{}^1-C_{22}{}^2=0,&2 C_{12}{}^2+C_{22}{}^1+1=0.
\end{array}$$
We solve these equations to see $\mathcal{M}=\mathcal{N}_4$. The
Ricci tensor distinguishes these surfaces from the previous examples.

We now examine the Lie algebra structure. 
Let $e_1:=X(\sigma)$, $e_2:=-x^1\partial_1-x^2\partial_2$, $e_3:=-\partial_2$. We then have
$[e_1,e_2]=e_1$, $[e_2,e_3]=e_3$, and $[e_1,e_3]=-2e_2$. These are the structure
equations for $\mathfrak{su}(1,1)$ given in Equation~\eqref{E3.d}.
The range of the adjoint map is 3-dimensional; the range of the adjoint
map in either $A_2\oplus A_2$ or $A_{4,9}^0$ is $2$-dimensional. Thus $A_{3,8} = \mathfrak{su(1,1)}$ is not
a Lie sub-algebra of either $A_2\oplus A_2$ or of $A_{4,9}^0$ and hence by Theorem~\ref{T3.8},
$\mathcal{M}$ is not of Type~$\mathcal{A}$. 

Let 
$\mathfrak{S}:=X(\sigma)\cdot\mathbb{R}\oplus\mathfrak{K}_0^{\mathcal{B}}$.
Suppose to the contrary, there is some additional  affine Killing vector field $Y\in\mathfrak{K}(\mathcal{M})-\mathfrak{S}$.
Since $\mathcal{M}$ is not of Type~$\mathcal{A}$, by Lemma~\ref{L3.15}, $Y\ne Y(x^1)$. The argument given
to prove Lemma~\ref{L3.14} shows, therefore, $\partial_2^nY=0$ for $n\ge2$. If $n=2$,
then we have that $Y=2x^1x^2\partial_1+(x^2)^2\partial_2+Z(x^1)$ and hence $X(\sigma)-Y= \sigma\cdot(x^1)^2\partial_2-Z(x^1)$ only depends of $x^1$ which
contradicts the observation made above. If $n>3$, we may replace $Y$ by $(\partial_2)^{n-3}Y$ to ensure
$n=3$. Since $(\partial_2)^2(\partial_2Y)=0$, we conclude $\partial_2Y$ must be a multiple of $X$ and hence
$$
Y=x^1(x^2)^2\partial_1+(\textstyle\frac13(x^2)^3+\sigma\cdot(x^1)^2 x^2)\partial_2+Y_0(x^1)\,.
$$
We apply $\Theta-2$ to see $(\Theta-2)Y_0\in\mathfrak{K}_0^{\mathcal{B}}$ and hence
$$
Y=x^1(x^2)^2\partial_1+(\textstyle\frac13(x^2)^3+\sigma\cdot(x^1)^2 x^2)\partial_2
+(x^1)^3(a^1\partial_1+a^2\partial_2)\,.
$$
We have Killing equations:
\medbreak\qquad $\sigma=0$: 
$a^1+2a^2C_{12}{}^1=0$, $a^2=0$, $2-4a^1C_{22}{}^1=0$.
\smallbreak\qquad $\sigma=1$:
$4a^1=0$, $3a^2=0$, $2(a^1-1)=0$. 
\smallbreak\qquad $\sigma=-1$:
$4a^1=0$, $3a^2=0$, $2(1+a^1)=0$.
\medbreak\noindent
These equations are inconsistent and thus there is no additional affine Killing vector field.

Up to linear equivalence and homothety, the only pseudo-Riemannian metrics which are
of Type~$\mathcal{C}$ have the form $ds^2=(x^1)^{-2}((dx^1)^2+\epsilon(dx^2)^2)$. We use the
Koszul formula
$\Gamma_{ijk}=\frac12\{g_{ik/j}+g_{jk/i}-g_{ij/k}\}$ to see:
$$\begin{array}{ll}
\Gamma_{111}=\frac12 g_{11/1}=-(x^1)^{-3},&\Gamma_{11}{}^1=-(x^1)^{-1},\\[.03in]
\Gamma_{122}=\frac12 g_{22/1}=-(x^1)^{-3}\epsilon,&\Gamma_{12}{}^2=-(x^1)^{-1},\\[.03in]
\Gamma_{221}=-\frac12 g_{22/1}= (x^1)^{-3}\epsilon,&\Gamma_{22}{}^1= (x^1)^{-1}\epsilon.
\end{array}$$
Taking $\epsilon=1$ (resp. $\epsilon=-1$) yields the  surfaces $\mathcal{N}_3$ or $\mathcal{N}_4$. Thus these are of Type~$\mathcal{C}$. On the other hand, the symmetric
Ricci tensor has rank at most 1 if $\mathcal{M}=\mathcal{N}_1^\pm$ or $\mathcal{M}=\mathcal{N}_2^c$ so these  surfaces are not of Type~$\mathcal{C}$.
 \end{proof}

\begin{remark}\rm
In Definition~\ref{D3.10}, we let $\mathcal{M}:=\mathcal{N}_2^c$ for $c=\frac12$ be determined by
 $$\textstyle
 C_{11}{}^1=-\frac{3}{2}\,,\, C_{11}{}^2=0\,,\, C_{12}{}^1=1\,,\, C_{12}{}^2=-\frac{1}{2}\,,\, C_{22}{}^1=\frac{1}{2}\,,\, C_{22}{}^2=2\,.
 $$ 
The Ricci tensor of this Type~$\mathcal{B}$ surface is alternating and this
affine surface corresponds to the distinguished situation in \cite[Theorem 2-(A.1)]{KVOp2}.
We shall see presently that, up to affine equivalence, this is the only affine surface of Type~$\mathcal{B}$
with $\dim\{\mathfrak{K}(\mathcal{M})\}=3$
 which admits an affine gradient Ricci  soliton.
\end{remark}

Lemma~\ref{L2.2} shows that every
 Type~$\mathcal{A}$  surface has $\rho$ and $\nabla \rho$ symmetric. If, moreover, 
 $\operatorname{Rank}\{\rho\}=1$, then $\rho$ is recurrent (see \cite{CGV10}). 
 The next result shows that these geometric conditions identify Type~$\mathcal{A}$ 
 among Type~$\mathcal{B}$ surfaces. It is an immediate consequence of Lemma~\ref{L2.10} 
 and the discussion of this section.

\begin{corollary}\label{C3.19}
Let $\mathcal{M}$ be a Type~$\mathcal{B}$ surface which is not flat. The following conditions are
equivalent:
\begin{enumerate}
\item $\mathcal{M}$ is also of Type~$\mathcal{A}$.
\item $\rho$ is symmetric, recurrent, and of rank $1$ and $\nabla\rho$ is symmetric.
\item $\dim\{\mathfrak{K}(\mathcal{M})\}=4$.
\item $C_{12}{}^1=C_{22}{}^1=C_{22}{}^2=0$.
\end{enumerate}
\end{corollary}

\begin{remark}
\rm
Note that the geometric conditions given in Corollary~\ref{C3.19}, i.e. $\rho$ is symmetric, recurrent, 
and of rank $1$ and $\nabla\rho$ is symmetric, characterize 
Type~$\mathcal{A}$ surfaces amongst 
Type~$\mathcal{B}$ ones, but not the converse. In Theorem~\ref{T3.8},
we have identified which surfaces
of Type~$\mathcal{A}$ are also of Type~$\mathcal{B}$ in terms of the $\alpha$ invariant given in Definition~\ref{D2.4} (see Table~1).
\end{remark}

\subsection{Change of coordinates}The
following result is closely related to the work of \cite{G-SG,Opozda} and deals with the
homogeneous affine surfaces where $\dim\{\mathfrak{K}(\mathcal{M})\}=2$.
As noted in the proof of Lemma~\ref{L3.9}, the Lie group
$$\mathfrak{G}:=\{T:(x^1,x^2)\rightarrow(tx^1,ux^1+vx^2+w)\text{ for }t>0\text{ and }v\ne0\}$$
is the 4-dimensional subgroup of $\operatorname{GL}(2,\mathbb{R})$
which preserves $\mathbb{R}^+\times\mathbb{R}$. 

\begin{theorem}\label{T3.21}
Let $\mathcal{M}$ be a simply connected locally homogeneous
affine surface with $\dim\{\mathfrak{K}(\mathcal{M})\}=2$.
\begin{enumerate}
\item If $\mathcal{M}$ is Type~$\mathcal{A}$, then the coordinate transformations of any Type $\mathcal{A}$ atlas for $\mathcal{M}$
 take the form $\vec x\rightarrow A\vec x+\vec b$ for $A\in\operatorname{GL}(2,\mathbb{R})$ and $\vec b\in\mathbb{R}^2$.
\item If $\mathcal{M}$ is Type~$\mathcal{B}$, then the coordinate transformations of any Type $\mathcal{B}$ atlas for $\mathcal{M}$ belong to $\mathfrak{G}$.
\end{enumerate}\end{theorem}

\begin{proof} Suppose first that $\mathcal{M}$ is of Type~$\mathcal{A}$.
Cover $M$ by Type~A coordinate charts $(\mathcal{O}_\alpha,\phi_\alpha)$
so ${}^\alpha\Gamma\in\mathbb{R}$ is constant. The transition functions $\phi_{\alpha\beta}$ 
then are local diffeomorphisms
of $\mathbb{R}^2$ so that $\phi_{\alpha\beta}^*\{{}^\beta\rho\}={}^\alpha\rho$.
Since $\dim\{\mathfrak{K}(\mathcal{M})\}=2$,
Theorem~\ref{T3.4} shows that $\operatorname{Rank}\{\rho^\alpha\}=2$
so ${}^\alpha\rho$
and ${}^\beta\rho$ define flat pseudo-Riemannian metrics with $\phi_{\alpha\beta}^*\{{}^\beta\rho\}={}^\alpha\rho$.
This implies $d\phi_{\alpha\beta}$ is constant and, consequently $\phi_{\alpha\beta}$ is
an affine transformation as given in Assertion~1.

Next suppose $\mathcal{M}$ is of Type~$\mathcal{B}$.
Cover $\mathcal{M}$ by Type~$\mathcal{B}$ coordinate charts $(\mathcal{O}_\alpha,\phi_\alpha)$
with transition functions $\phi_{\alpha\beta}$. Fix $\alpha$ and $\beta$ and let
$$
\vec x=(x^1_\beta,x^2_\beta),\quad\vec u=(u^1_\alpha,u^2_\alpha),\quad
\phi_{\alpha\beta}=(x^1(u^1,u^2),x^2(u^1,u^2))\,.
$$
We have
\begin{eqnarray*}
&&\partial_1^u=\partial_1^ux^1\cdot\partial_{1}^x+\partial_1^ux^2\cdot\partial_{2}^x\quad\text{ and }\quad
\partial_2^u=\partial_2^ux^1\cdot\partial_{1}^x+\partial_2^ux^2\cdot\partial_{2}^x\,.
\end{eqnarray*}
Since $[x^1\partial_1^x+x^2\partial_2^x,\partial_2^x]=-\partial_2^x$ and
$\dim\{\mathfrak{K}(\mathcal{M})\}=2$, $\partial_2^x$ (and similarly $\partial_2^u$) span the
range of the adjoint action. Consequently $\partial_2^u$ is a constant multiple of $\partial_2^x$.
This implies that $\partial_2^ux^1=0$ and $\partial_2^ux^2\in\mathbb{R}$ so
$$
\phi_{\alpha\beta}=(x^1(u^1),\tilde x^2(u^1)+cu^2)\text{ for }c\in\mathbb{R}\,.
$$
We now have that
$\partial_1^u=\partial_1^ux^1\cdot\partial_{1}^x+\partial_1^u\tilde x^2\cdot\partial_{2}^x$,
$\partial_2^u=c\partial_{2}^x$, and 
$$
u^1\partial_1^u+u^2\partial_2^u=
u^1\partial_1^ux^1\partial_{1}^x+\star\partial_{2}^x
=\epsilon_1(x^1\partial_{1}^x+x^2\partial_{2}^x)+\star\partial_{2}^x\,.
$$
This tells us that $x^1=au^1$ for some $a\in\mathbb{R}$ and that $\epsilon_1=1$. Consequently
$$
\phi_{\alpha\beta}=(au^1,\tilde x^2(x^1)+cu^2)\,.
$$
Therefore $\partial_1^u=a\partial_{1}^x+\partial_1^u\tilde x^2(x^1)\partial_{2}^x$
and $\partial_2^u=c\partial_{2}^x$ so
$$
u^1\partial_1^u+u^2\partial_2^u=x^1\partial_{1}^x+x^2\partial_{2}^x+\{u^1\partial_1^u\tilde x^2(u^1)-\tilde x^2(u^1)\}\partial_{2}^x\,.
$$
Since $u^1\partial_1^u+u^2\partial_2^u\in\operatorname{Span}\{x^1\partial_1^x+x^2\partial_2^x,\partial_2^x\}$,
we conclude 
$$
 u^1\partial_1^u\tilde x^2(u^1)-\tilde x^2(u^1)=d\in\mathbb{R}\,.
$$
Thus $\tilde x^2(u^1)=b u^1+d$ and the
coordinate transformation has the desired form.
\end{proof}

\begin{remark}\rm
If $\dim\{\mathfrak{K}(\mathcal{M})\}>2$, then Theorem~\ref{T3.11} shows that there are Killing
vector fields which do not belong to the Lie algebra of $\operatorname{GL}(2,\mathbb{R})$. Consequently,
there are admissible coordinate transformations which are non-linear. This shows the condition 
$\dim\{\mathfrak{K}(\mathcal{M})\}=2$ is essential in Theorem~\ref{T3.21}.
\end{remark}

Lemma~\ref{L3.6} gives representatives of all the elements in $\mathcal{F}^{\mathcal{A}}$
which are also of Type~$\mathcal{B}$. We now give an explicit identification of those
surfaces with
elements of $\mathcal{F}^{\mathcal{B}}$.
In the proof of Theorem \ref{T3.8}, we showed that every Type~$\mathcal{A}$ surface
which is also Type~$\mathcal{B}$ admits coordinates $(x^1,x^2)$ such that the 
corresponding Christoffel symbols $\Gamma_{ij}^k\in\mathbb{R}$ satisfy
$\Gamma_{11}{}^1=\Gamma_{11}{}^2=\Gamma_{12}{}^2=\Gamma_{22}{}^1=0$. 
We now give an explicit construction to show that such elements of 
$\mathcal{F}^{\mathcal{A}}$, which are affine isomorphic to 
$\mathcal{M}_3^c$ or $\mathcal{M}_4^0$,  are also of Type~$\mathcal{B}$.

\begin{lemma} Let $\mathcal{M}\in\mathcal{F}^{\mathcal{A}}$ satisfy
$\Gamma_{11}{}^1=\Gamma_{11}{}^2=\Gamma_{12}{}^2=\Gamma_{22}{}^1=0$.
We consider the change of coordinates $(u^1,u^2)=(e^{x^2},x^1)$. We then have:
\medbreak\centerline{$
{}^u\Gamma_{11}{}^1=\frac{1}{u^1}(-1+{}^x\Gamma_{22}{}^2)\,,\,
{}^u\Gamma_{12}{}^2=\frac{1}{u^1}{}^x\Gamma_{12}{}^1\,,\,
{}^u\Gamma_{ij}{}^k=0\,\text{otherwise}.
$}
\end{lemma}
\begin{proof} We compute:
\medbreak\qquad
$du^1=e^{x^2}dx^2,\quad du^2=dx^1,\quad\partial_1^u=e^{-x^2}\partial_2^x,
\quad\partial_2^u=\partial_1^x$,
\smallbreak\qquad
$\nabla_{\partial_1^u}\partial_1^u=e^{-2x^2}\{-\partial_2^x
+{}^x\Gamma_{22}{}^1\partial_1^x+{}^x\Gamma_{22}{}^2\partial_2^x\} =\frac1{u^1}(-1+{}^x\Gamma_{22}{}^2)\partial_1^u$,
\smallbreak\qquad
$\nabla_{\partial_1^u}\partial_2^u=e^{-x^2}\{{}^x\Gamma_{12}{}^1\partial_1^x
+{}^x\Gamma_{12}{}^2\partial_2^x\} =\frac{1}{u^1}{}^x\Gamma_{12}{}^1 \partial_2^u$,
\smallbreak\qquad
$\nabla_{\partial_2^u}\partial_2^u={}^x\Gamma_{11}{}^1\partial_1^x+{}^x\Gamma_{11}{}^2\partial_2^x =0$,
\smallbreak\qquad
$\begin{array}{lll}
{}^u\Gamma_{11}{}^1=\frac1{u^1}(-1+{}^x\Gamma_{22}{}^2),&
{}^u\Gamma_{11}{}^2=0,&
{}^u\Gamma_{12}{}^1=0,\\[.03in]
{}^u\Gamma_{12}{}^2=\frac1{u^1}\Gamma_{12}{}^1,&
{}^u\Gamma_{22}{}^1=0,&{}^u\Gamma_{22}{}^2=0.
\end{array}$
\end{proof}

\section{Affine Gradient Ricci Solitons}\label{S4}

In this section we study affine gradient Ricci solitons and affine gradient Yamabe solitons. Recall from Definition~\ref{D1.5} that $(M,\nabla,f)$ is an affine gradient Ricci (resp. Yamabe) soliton if
 $H_f^\nabla+\rho_s=0$ (resp. $H_f^\nabla =0$). 
 $\mathfrak{A}(\mathcal{M})$ (resp. $\mathcal{Y}(\mathcal{M})$) is the space of functions on $M$ so that $(M,\nabla, f)$ is an affine gradient Ricci (resp. Yamabe) soliton. 
 The following result relates these two notions.

\begin{lemma}\label{L4.1} Let $\mathcal{M}=(M,\nabla)$ be an affine surface. 
\begin{enumerate}
\item If $f\in\mathfrak{A}(\mathcal{M})$ and if $X\in\mathfrak{K}(\mathcal{M})$,
then $X(f)\in\ker(H^\nabla)$, i.e. $X(f)\in\mathcal{Y}(M)$.
\item If $h\in\ker(H^\nabla)$, then $R_{ij}(dh)=0$ for $1\le i<j\le m$.
\end{enumerate}\end{lemma}
\begin{proof} 
Let $f$ be an affine gradient Ricci soliton and 
let $X$ be an affine Killing vector field. We have by naturality
that $(\Phi_t^X)^*f$ is again an affine gradient Ricci soliton. 
Since the difference of two affine gradient Ricci solitons
belongs to $\ker(H^\nabla)$, $(\Phi_t^X)^*f-f\in\ker(H^\nabla)$. 
Differentiating this relation with respect to $t$
and setting $t=0$ yields Assertion~1. Assertion~2 follows
from the identity
\medbreak\hfill$h_{;ijk}-h_{;ikj}=\{R_{kj}(dh)\}_i$.\hfill\vphantom{.}
\end{proof}

\subsection{Type~$\mathcal{A}$ affine gradient Ricci solitons}\label{S4.1}
 Let $\mathcal{M}$ be a Type~$\mathcal{A}$ affine surface. The
associated Ricci tensor is symmetric and 
$\mathfrak{K}_0^{\mathcal{A}}:=\operatorname{Span}_{\mathbb{R}}\{\partial_1,\partial_2\}\subset\mathfrak{K}(\mathcal{M})$. The components
of the Hessian are given by:
\medbreak\qquad
$H_{11}^\nabla(f)= -\Gamma_{11}{}^2 f^{(0,1)}-\Gamma_{11}{}^1 f^{(1,0)}+f^{(2,0)}$,
\smallbreak\qquad
$H_{12}^\nabla(f)=H_{21}^\nabla(f)= -\Gamma_{12}{}^2 f^{(0,1)}-\Gamma_{12}{}^1 f^{(1,0)}+f^{(1,1)}$,
\smallbreak\qquad
$H_{22}^\nabla(f)=-\Gamma_{22}{}^2 f^{(0,1)}+f^{(0,2)}-\Gamma_{22}{}^1 f^{(1,0)}$.
\medbreak\noindent
If $\operatorname{Rank}\{\rho\}=1$, we normalize the coordinate
system so $\rho=\rho_{22}dx^2\otimes dx^2\ne0$. We examine this situation
in the following result.
\begin{lemma}\label{L4.2}
Let $\mathcal{M}=(\mathbb{R}^2,\nabla)$ where $\Gamma_{ij}{}^k\in\mathbb{R}$
and $\rho=\rho_{22}dx^2\otimes dx^2\ne0$. Then
$f\in\mathfrak{A}(\mathcal{M})$
 if and only if $f(x^1,x^2)=\xi(x^2)$ where 
$\xi^{\prime\prime}-\Gamma_{22}{}^2\xi^\prime+\rho_{22}=0$.
\end{lemma}
\begin{proof}
Let $\rho=\rho_{22}dx^2\otimes dx^2\ne0$. We impose the relations of Lemma~\ref{L2.3} and
set $\Gamma_{11}{}^2=0$ and $\Gamma_{12}{}^2=0$. 
Let $f\in\mathfrak{A}(\mathcal{M})$.
 We have soliton equations
\[
\begin{array}{l}
f^{(2,0)}-\Gamma_{11}{}^1f^{(1,0)}=0,\\
f^{(1,1)}-\Gamma_{12}{}^1 f^{(1,0)}=0,\\
\rho_{22}-\Gamma_{22}{}^1 f^{(1,0)}-\Gamma_{22}{}^2
 f^{(0,1)}+f^{(0,2)}=0.
\end{array}
\]
We use the first equation to break the analysis into two cases.
\subsection*{ Case 1.
$\Gamma_{11}{}^1\ne0$} Then $f=u_0(x^2)+u_1(x^2)e^{\Gamma_{11}{}^1x^1}$.
The second soliton equation yields $u_1^\prime(x^2)-\Gamma_{12}{}^1u_1(x^2)=0$. 
Thus
 $f=u_0(x^2)+c_1e^{\Gamma_{11}{}^1x^1+\Gamma_{12}{}^1x^2}$. The final soliton equation is
$u_0^{\prime\prime}(x^2)-\Gamma_{22}{}^2 u_0^\prime(x^2)-\rho_{22}(c_1 e^{\Gamma_{11}{}^1 x^1+\Gamma_{12}{}^1 x^2}-1)
=0$.
This implies that $c_1=0$ and that $u_0$ satisfies the ODE given above.
\subsection*{Case 2. $\Gamma_{11}{}^1=0$.} Then $f=u_0(x^2)+u_1(x^2)x^1$. We consider the second soliton equation $u_1'(x^2)-\Gamma_{12}{}^1 u_1(x^2)=0$ to see that $f=u_0(x^2)+c_1e^{\Gamma_{12}{}^1x^2}x^1$. Hence the final soliton equation becomes $0=-c_1e^{x^2\Gamma_{12}{}^1}x^1\rho_{22}+...$,
where we have omitted terms not involving $x^1$. Since $\rho_{22}\ne0$,
we see that $c_1=0$ so $f(x^1,x^2)=\xi(x^2)$
and $f$ is a gradient Ricci soliton if and only if $\xi$ satisfies the ODE given.
This completes the proof.
\end{proof}

The next result shows that gradient Ricci solitons given in Lemma~\ref{L4.2} are the only ones in Type~$\mathcal{A}$ surfaces.

\begin{theorem}\label{T4.3}
Let $\mathcal{M}=(\mathbb{R}^2,\nabla)$ where $\Gamma_{ij}{}^k\in\mathbb{R}$ and $\rho\ne0$.
The following assertions are equivalent:
\begin{enumerate}
\item $\operatorname{Rank}\{\rho\}=1$.
\item $\mathfrak{A}(\mathcal{M})$ is non-empty.
\end{enumerate}
\end{theorem}
\begin{proof}
It is clear from Lemma~\ref{L4.2} that 
if $\operatorname{Rank}\{\rho\}=1$, then there exists $f\in \mathfrak{A}(\mathcal{M})$.
We now show that $\operatorname{Rank}\{\rho\}=1$
if $\mathfrak{A}(\mathcal{M})\ne\{0\}$.
Suppose to the contrary that $\operatorname{Rank}\{\rho\}=2$; we argue
for a contradiction. We necessarily have
 $\operatorname{Rank}\{R_{12}\}=2$. We apply Lemma~\ref{L4.1}.
 If $h\in\ker(H^\nabla)$, then $R_{12}(dh)=0$ and
 $dh=0$. Thus $\ker(H^\nabla)$ consists
of the constants. Suppose $f$ is a non-trivial gradient Ricci soliton. Since $\partial_1$ and
$\partial_2$ are affine Killing vector fields, $\partial_1f$ and $\partial_2f$ are constant.
This implies $f(x^1,x^2)=ax^1+bx^2+c$ is affine. We may make an affine change of coordinates
to assume $f(x^1,x^2)=x^2$.
We shall establish the desired contradiction by showing $\Gamma_{11}{}^2=0$ and $\Gamma_{12}{}^2=0$
and then applying Lemma~\ref{L2.3} to see $\operatorname{Rank}\{\rho\}=1$.
The soliton equations for $f(x^1,x^2)=x^2$ are given by:
\begin{eqnarray*}
0&=&\Gamma_{12}{}^2 (\Gamma_{11}{}^1-\Gamma_{12}{}^2)+\Gamma_{11}{}^2 (-\Gamma_{12}{}^1+\Gamma_{22}{}^2-1),\\
0&=&(\Gamma_{12}{}^1-1) \Gamma_{12}{}^2-\Gamma_{11}{}^2 \Gamma_{22}{}^1,\\
0&=&\Gamma_{11}{}^1 \Gamma_{22}{}^1+\Gamma_{12}{}^1 (\Gamma_{22}{}^2-\Gamma_{12}{}^1)
-\Gamma_{12}{}^2 \Gamma_{22}{}^1-\Gamma_{22}{}^2\,.
\end{eqnarray*}
Again, we examine possibilities:
\subsection*{Case 1. Suppose $\Gamma_{12}{}^2\ne0$} We normalize and set
$\Gamma_{12}{}^2=1$. A soliton equation implies
$\Gamma_{11}{}^2 \Gamma_{22}{}^1-\Gamma_{12}{}^1+1=0$. Thus we set $\Gamma_{12}{}^1=1+\Gamma_{11}{}^2\Gamma_{22}{}^1$.
This yields a soliton equation 
$$
\Gamma_{11}{}^1-(\Gamma_{11}{}^2)^2 \Gamma_{22}{}^1+\Gamma_{11}{}^2 (\Gamma_{22}{}^2-2)-1=0\,.
$$
We set 
$\Gamma_{11}{}^1=(\Gamma_{11}{}^2)^2 \Gamma_{22}{}^1-\Gamma_{11}{}^2 (\Gamma_{22}{}^2-2)+1$.
This yields an inconsistent equation. Thus this is impossible.
\subsection*{Case 2. Suppose $\Gamma_{12}{}^2=0$}
 If $\Gamma_{11}{}^2=0$, we have the desired contradiction.
Thus we suppose $\Gamma_{11}{}^2\ne0$. 
By renormalizing $x^2$, we may assume $\Gamma_{11}{}^2=1$.
We have soliton equations $1+\Gamma_{12}{}^1-\Gamma_{22}{}^2=0$ and $\Gamma_{22}{}^1=0$.
We set $\Gamma_{22}{}^2=1+\Gamma_{12}{}^1$ and $\Gamma_{22}{}^1=0$. The final soliton
equation then becomes $0=-1$. This provides the desired contradiction and
completes the proof of the theorem.
\end{proof}

\subsection{Type~$\mathcal{B}$ affine Gradient Ricci Solitons}
Our analysis is similar to that of Section~\ref{S4.1} which
 dealt with Type~$\mathcal{A}$ surfaces. We compute the components of the Hessian
 in this setting:
 \medbreak\qquad
$H_{11}^\nabla(f)=-(x^1)^{-1}\{C_{11}{}^1 f^{(1,0)}+C_{11}{}^2 f^{(0,1)}-x^1
 f^{(2,0)}\} $,
\smallbreak\qquad
$H_{12}^\nabla(f)=H_{21}^\nabla(f)= -(x^1)^{-1}\{C_{12}{}^1
 f^{(1,0)}+C_{12}{}^2
 f^{(0,1)}-x^1 f^{(1,1)}\}$,
\smallbreak\qquad
$H_{22}^\nabla(f)=-(x^1)^{-1}\{C_{22}{}^1
 f^{(1,0)}+C_{22}{}^2 f^{(0,1)}-x^1
 f^{(0,2)}\}$.
\begin{definition}\label{D4.4}\rm
Let $(a,c)\ne(0,0)$ and $c\ge0$. Let $\tilde c\in\mathbb{R}$. Let
$\mathcal{P}_{a,c}^\pm$ and $\mathcal{Q}_{\tilde c}$ be the affine surfaces defined by:$$
\begin{array}{llll}\mathcal{P}_{a,c}^\pm:&C_{11}{}^1=\frac{1}{2} \left(a^2+4 a\mp2c^2+2\right),
&C_{11}{}^2=c,&C_{12}{}^1=0,\\[0.05in]
&C_{12}{}^2=\frac{1}{2} \left(a^2+2 a\mp2c^2\right),&
C_{22}{}^1=\pm1,&C_{22}{}^2=\pm2c,\\[0.05in]
\mathcal{Q}_{\tilde c}:&C_{11}{}^1=0,&C_{11}{}^2=\tilde c,&C_{12}{}^1=1,\\
& C_{12}{}^2=0,& C_{22}{}^1=0,& C_{22}{}^2=1\,.
\end{array}$$
\end{definition}

\begin{remark}\rm
We show that $\mathcal{Q}_{\tilde c}$ is not flat by computing:
$$
\rho(\mathcal{Q}_{\tilde c})=(x^1)^{-2}\left(\begin{array}{rr}0&1\\-1&0\end{array}\right)\,.
$$
Similarly, since $(a,c)\ne(0,0)$, we show $\mathcal{P}_{a,c}^\pm$ is not flat by computing:
$$
\rho(\mathcal{P}_{a,c}^\pm)=(x^1)^{-2}\left(\begin{array}{rr}
a (\frac12(a+2)^2 \mp c^2)&
\pm c\\
\mp c&\pm a\end{array}\right)\,.
$$
A direct computation shows $a\log(x^1)\in\mathfrak{A}(\mathcal{P}_{a,c}^\pm)$
and $0\in\mathfrak{A}(\mathcal{Q}_c)$. 
We will show
in Theorem~\ref{T4.12} that none of these surfaces is isomorphic to a different surface and that
$\mathcal{P}_{0,\frac3{\sqrt2}}^+\approx\mathcal{N}_2^{\frac12}$
where
$$
\begin{array}{llll}
\mathcal{N}_2^c:&C_{11}{}^1=-\frac32,& 
C_{11}{}^2=0,&C_{12}{}^1=1,\\[0.05in]
&C_{12}{}^2=-\frac12,&C_{22}{}^1=c,&C_{22}{}^2=2
\end{array}$$
is as defined in Definition~\ref{D3.10}.
This is the only surface with $\dim\{\mathfrak{K}{ (\mathcal{M})}\}=3$ and $\mathfrak{A}$ non-empty.
\end{remark}

As opposed to Type~$\mathcal{A}$ surfaces, Type~$\mathcal{B}$ 
surfaces do not have symmetric Ricci tensor in the generic situation. 
We recall that as well as there is a one to one relation between affine gradient 
Ricci soliton on an affine surface $(\Sigma, D)$ and gradient Ricci soliton on 
the associated Riemannian extension $(T^*\Sigma, g_D)$ \cite{BG14}, 
there is also a one to one relation between Einstein (indeed Ricci flat) Riemannian extensions 
$(T^*\Sigma, g_D)$ and affine surfaces $(\Sigma, D)$ with alternating 
Ricci tensor \cite{GKVV99}.
 The following result
gives a complete characterization
of the elements of $\mathcal{F}^{\mathcal{B}}$ where
$\rho$ is alternating and is a slightly different treatment than that in
\cite{KVOp2}. 

\begin{lemma}\label{L4.6}
Let $\mathcal{M}\in\mathcal{F}^{\mathcal B}$.
The following assertions are equivalent.
\begin{enumerate}
\item The Ricci tensor is alternating, i.e., $\rho_{ij}=-\rho_{ji}$ for all $1\le i,j\le 2$.
\item $0\in\mathfrak{A}(\mathcal{M})$.
\item $\mathcal{M}$ is isomorphic to $\mathcal{P}_{0,c}^\pm$ for $c>0$ or to $\mathcal{Q}_c$
for arbitrary $c$.
\item $\mathfrak{A}(\mathcal{M})=\mathbb{R}$ consists of the constant functions.
\end{enumerate}\end{lemma}
\begin{proof}
The equivalence of Assertion~1 and Assertion~2
is immediate.  A direct computation shows $\rho$ is alternating and non-trivial
 if $\mathcal{M}$ is isomorphic
to $\mathcal{P}_{0,c}^\pm$ for $c\ne0$ or if $\mathcal{M}$ is isomorphic to $\mathcal{Q}_c$.
Conversely, suppose $\rho$ is alternating. We distinguish two cases:
\subsection*{Case 1. Suppose $C_{22}{}^1\ne0$} We apply Lemma~\ref{L2.8} to normalize
the coordinate system so $C_{12}{}^1=0$; we then rescale to assume $C_{22}{}^1=\pm1$.
We set $C_{11}{}^2=c$; by changing the sign of $x^2$, we may assume $c\ge0$.
We have $\rho^s_{12}=C_{22}{}^2\mp2c$
and $\rho_{22}=\pm(-1+C_{11}{}^1-C_{12}{}^2)$. We set $C_{22}{}^2=\pm2c$
and $C_{11}{}^1=1+C_{12}{}^2$. We have $\rho_{11}=C_{12}{}^2\pm c^2$. To ensure $\rho\ne0$,
we require $c\ne0$ and hence $c>0$. Thus we obtain
the relations of $\mathcal{P}_{0,c}^\pm$:
$$
C_{11}{}^1=1\mp c^2,\ C_{11}{}^2=c,\ C_{12}{}^1=0,\ C_{12}{}^2=\mp c^2,\ 
C_{22}{}^1=\pm1,\ C_{22}{}^2=\pm2c\,.
$$
\subsection*{Case 2. Suppose $C_{22}{}^1=0$} We set $\rho_0:=\Gamma_{ij}{}^jdx^i$;
this is invariant under the action of $\operatorname{GL}(2,\mathbb{R})$.  We compute
$$
0\ne\rho_{12}-\rho_{21}=(x^1)^{-2}(C_{12}{}^1+C_{22}{}^2)=(x^1)^{-1}\rho_0(\partial_2)\,.
$$
We can rescale $x^2$ to ensure $C_{12}{}^1+C_{22}{}^2=2$. By replacing
$\partial_1$ by $\partial_1-\varepsilon\partial_2$ for suitably chosen $\varepsilon$,
we may assume $\rho_0(\partial_1)=(x^1)^{-1}\{C_{11}{}^1+C_{12}{}^2\}=0$. We 
set $C_{22}{}^1=0$, $C_{11}{}^1=-C_{12}{}^2$, and $C_{12}{}^1=2-C_{22}{}^2$.
We obtain 
$$
\rho_{22}=-2{ (x^1)^{-2}}\{2-3C_{22}{}^2+(C_{22}{}^2)^2\}\,.
$$
 If $C_{22}{}^2=2$,
we obtain $\rho_{12}^s=(x^1)^{-2}$ which is false. Thus $C_{22}{}^2=1$. We then obtain
$\rho_{12}^s={ (x^1)^{-2}}C_{12}{}^2$ so we set $C_{12}{}^2=0$. Let $C_{11}{}^2=c$; this
is a free parameter. We obtain the structure $\mathcal{Q}_c$.

We have shown the equivalence of Assertion~1, Assertion~2, and of Assertion~3.
If $\mathfrak{A}(\mathcal{M})=\mathbb{R}$, then $0\in\mathfrak{A}(\mathcal{M})$ and
consequently Assertion~2 holds. Conversely, suppose Assertion~3 holds. Again, we distinguish
cases:
\subsection*{Case 1. Suppose $\mathcal{M}=\mathcal{P}_{0,c}^\pm$}
Let $f$ be a gradient Ricci soliton. A soliton equation $0=\pm c^2f^{(0,1)}+x^1f^{(1,1)}$
implies $f^{(0,1)}=(x^1)^{\mp c^2}f_0(x^2)$. We integrate to see that
$f=(x^1)^{\mp c^2}f_1(x^2)+f_2(x^1)$ and obtain a soliton equation 
$$0=c^2f_1(x^2)+x^1\{\mp2cf_1^\prime(x^2)\mp(x^1)^{\pm c^2}f_2^\prime(x^1)
+x_1f_1^{\prime\prime}(x^2)\}\,.$$
We differentiate with respect to $x^2$ to see
$0=c^2f_1^\prime(x^2)+x^1\{\mp2cf_1^{\prime\prime}(x^2)+x_1f^{\prime\prime\prime}(x^2)\}\,.$
Set $x^1=0$ to see $f_1^\prime(x^2)=0$ since $c\ne0$. Consequently $f_1$ is constant and
$f=f(x^1)$. We now obtain a soliton equation $0=\mp x^1f^\prime(x^1)$ so $f$ is constant
as desired.
\end{proof}

The intersection between Type~$\mathcal{B}$ and Type~$\mathcal{A}$ surfaces was previously studied in Section~\ref{S3.2} (cf. Corollary \ref{C3.19}). Now we consider the existence of affine gradient Ricci solitons in that particular setting.

\begin{lemma}\label{L4.7}
Let $\mathcal{M}\in\mathcal{F}^{\mathcal B}$. Assume $\mathcal{F}$ also is of Type~$\mathcal{A}$.
Then $f(x^1,x^2)\in\mathfrak{A}(\mathcal{M})$ if and only if $f(x^1,x^2)=\xi(x^1)$ where $\xi$
satisfies the ODE
\smallbreak\centerline{$(1+C_{11}{}^1-C_{12}{}^2)C_{12}{}^2-x^1C_{11}{}^1\xi^\prime
+(x^1)^2\xi^{\prime\prime}=0$.}

\end{lemma}
\begin{proof}
By Lemma \ref{L2.10} and Corollary \ref{C3.19} we have that $C_{12}{}^1=0$, $C_{22}{}^1=0$, and $C_{22}{}^2=0$. Then setting $\varepsilon:=(1+C_{11}{}^1-C_{12}{}^2)C_{12}{}^2$ the Ricci tensor takes the form:
$$
\rho=(x^1)^{-2}\left(
\begin{array}{cc}
\varepsilon& 0 \\
 0 & 0 \\
\end{array}
\right)\,.
$$
A soliton equation yields
$f^{(0,2)}(x^1,x^2)=0$ so $f=a(x^1)+b(x^1)x^2$. We obtain a soliton equation
$C_{12}{}^2b(x^1)=x^1b'(x^1)$. This implies $f(x^1,x^2)=a(x^1)+cx^2(x^1)^{C_{12}{}^2}$.
We obtain a single soliton equation:
$$
0=(x^1)^2 a^{\prime\prime}(x^1)-C_{11}{}^1 x^1 a^\prime(x^1)
+\varepsilon-c\, C_{11}{}^2 (x^1)^{C_{12}{}^2+1}
-c x^2 (x^1)^{C_{12}{}^2}\varepsilon\,.
$$
Since $\varepsilon\ne0$, we conclude $c=0$ and
$f=f(x^1)$ satisfies the given ODE.
\end{proof}

\medbreak

The following is a useful technical result. We adopt the notation of Definition~\ref{D3.10}.

\begin{lemma}\label{L4.8}
Assume that $\mathcal{M}\in\mathcal{F}^{\mathcal B}$
admits a gradient Ricci soliton.
At least one
of the following possibilities holds:
\begin{enumerate}
\item $C_{12}{}^1=C_{22}{}^1=C_{22}{}^2=0$, i.e. $\dim\{\mathfrak{K}(\mathcal{M})\}=4$
and $\mathcal{M}$ also is of Type~$\mathcal{A}$.
\item $f=a\log(x^1)\in\mathfrak{A}(\mathcal{M})$.
\end{enumerate}
\end{lemma}

\begin{proof} We will use Lemma~\ref{L4.1}.
Suppose that
$\operatorname{Rank}\{R_{12}\}=2$. 
If $h\in\ker(H^\nabla)$, then
$R_{12}(dh)=0$ so $dh=0$ and $h\in\mathbb{R}$ is a constant.
Let $f$ be an affine gradient Ricci soliton. Let $\xi_1=\partial_{2}$ and $\xi_2=x^1\partial_1+x^2\partial_{2}$
be affine Killing vector fields. Since $\ker(H^\nabla)=\mathbb{R}$,
$$\xi_1f=[\xi_1,\xi_2]f=\xi_1(\xi_2f)-\xi_2(\xi_1f)=0$$ so $f=f(x^1)$. Furthermore
$\xi_2(f)=x^1\partial_1f=a\in\mathbb{R}$ so, if $f\in\mathfrak{A}(\mathcal{M})$, then $f(x^1)=a\log(x^1)$
as desired.

We may therefore assume that $\operatorname{Rank}\{R_{12}\}=1$. 
We have $(x^1)^2R_{12}$ is constant. Choose $(\alpha,\beta)\in\mathbb{R}^2-\{0\}$ so that
$$
\operatorname{\ker(R_{12})}=\operatorname{Span}\{\alpha dx^1+\beta dx^2\}\,.
$$
If $\beta=0$, then $dx^1$ spans $\ker(R_{12})$. If $\beta\ne0$, let $(\tilde x^1,\tilde x^2):=(x^1,\alpha x^1+\beta x^2)$
to ensure that $d\tilde x^2$ spans $\ker(R_{12})$. Thus we may suppose
that $\ker(H^\nabla)=\operatorname{Span}\{dx^i\}$ for $i=1,2$.

\subsection*{Case 1. Suppose that
$\ker(R_{12})=\operatorname{Span}\{dx^2\}$.} Then
$$\ker(H^\nabla)=\{h:h( x^1, x^2)=h( x^2)\}\,.$$
Let $f\in\mathfrak{A}(\mathcal{M})$. Since $\partial_{2}f\in\ker(H^\nabla)$,
$\partial_2f$ is a function of $x^2$. This shows that
$f( x^1, x^2)=u( x^1)+v( x^2)$ so the problem decouples. Furthermore, since
$$
( x^1\partial_{1}+ x^2\partial_{2})f\in\ker(H^\nabla)\,,
$$
we have that
$ x^1\partial_{1}u(x^1)$ is a function only of $x^2$ and hence
$x^1\partial_{1}u(x^1)\in\mathbb{R}$. Thus we conclude
$$f( x^1, x^2)=a\log(x^1)+v( x^2)\,.$$
A Ricci soliton equation is:
$$
0=\star+x^1C_{22}{}^2v^\prime(x^2)-(x^1)^2v^{\prime\prime}(x^2)
$$
where $\star$ indicates a coefficient which is independent of $x^1$ and of $x^2$.
From this we see that $v^{\prime\prime}(x^2)=0$ so $v$ is linear in $x^2$
and $f=a\log(x^1)+bx^2$. We have
$$
a+bx^2=(x^1\partial_1+x^2\partial_2)f\in\ker(H^\nabla)\,.
$$
Subtracting this from $f$ yields $a\log(x^1)-a\in\mathfrak{A}(\mathcal{M})$
and hence $a\log(x^1)\in\mathfrak{A}(\mathcal{M})$ as desired.

\subsection*{Case 2. Suppose $\ker(R_{12})=\operatorname{Span}\{dx^1\}$.}
Thus if $f\in\mathfrak{A}(\mathcal{M})$,
then $\partial_{2}f$ must be a function only of $x^1$.
Thus $f(x^1,x^2)=u(x^1)+v(x^1) x^2$. Since $(x^1\partial_1+x^2\partial_2)f$
also is only a function of $x^1$, we obtain the equation
$(x^1v^\prime(x^1)+v(x^1))=0$ so $v(x^1)=b\cdot(x^1)^{-1}$ and
$f=u(x^1)+b\cdot(x^1)^{-1}x^2$ for $b\in\mathbb{R}$. There are two 
subcases to be considered. 
\subsection*{Case 2a. Suppose $f=u(x^1)+b\cdot(x^1)^{-1}x^2$ for $b\ne0$.}
We normalize $x^2$ to assume $b=1$ so 
$f=u(x^1)+(x^1)^{-1}x^2$. We obtain soliton equations:
\begin{eqnarray*}
0&=&(2+C_{11}{}^1)x^2+\star(x^1),\\
0&=&2C_{12}{}^1x^2+\star(x^1),\\
0&=&C_{22}{}^1x^2+\star(x^1)\,.
\end{eqnarray*}
This implies that $C_{11}{}^1=-2$, $C_{12}{}^1=0$ and $C_{22}{}^1=0$.
A soliton equation then also yields $C_{22}{}^2=0$. This is covered by Assertion~1.
\subsection*{Case 2b. Suppose $f(x^1,x^2)=f(x^1)$.}
Assume also
$f(x^1)\ne a\log(x^1)+b$ so $x^1f^\prime(x^1)\notin\mathbb{R}$.
We obtain soliton equations:
$$
0=\star+2x^1C_{12}{}^1u^\prime(x^1),\quad\text{and}\quad
0=\star+x^1C_{22}{}^1u^\prime(x^1)\,.
$$
We may then conclude that $C_{12}{}^1=0$ and $C_{22}{}^1=0$. 
A remaining soliton equation then yields $C_{22}{}^2=0$ which is the
case treated in Assertion~1
\end{proof}

We can now establish the following classification result.

\begin{theorem}\label{T4.9}
Let $\mathcal{M}\in\mathcal{F}^{\mathcal B}$.
The space $\mathfrak{A}(\mathcal{M})$ is non-empty if and only if at least one of the following possibilities holds up to linear equivalence:
\begin{enumerate}
\item $\mathcal{M}$ is also Type~$\mathcal{A}$, i.e.,
$C_{12}{}^1=0$, $C_{22}{}^1=0$, and $C_{22}{}^2=0$.
\item $\mathcal{M}$ is isomorphic to $\mathcal{P}_{a,c}^\pm$ for $(a,c)\ne(0,0)$ or to $\mathcal{Q}_c$
for arbitrary $c$.
\end{enumerate}
\end{theorem}

\begin{proof}
We examine cases. 
We apply Lemma~\ref{L4.8}.
The case $C_{12}{}^1=0$, $C_{22}{}^1=0$, and $C_{22}{}^2=0$ corresponding to Assertion 1 was examined
in Lemma~\ref{L4.7}. We complete the proof of Theorem~\ref{T4.9}, 
by assuming that $a\log(x^1)\in\mathfrak{A}(\mathcal{M})$. If $a=0$,
then $0\in\mathfrak{A}(\mathcal{M})$. This is the setting of Lemma~\ref{L4.6};
$\rho$ is alternating and we obtain the examples $\mathcal{P}_{0,c}$ for $c\ne0$
or $\mathcal{Q}_c$ for arbitrary $c$. We therefore assume $a\ne0$. We decompose the analysis into two cases depending on  whether
$C_{22}{}^1\neq 0$ or $C_{22}{}^1=0$.

Suppose first that $C_{22}{}^1\ne0$.
We apply Lemma~\ref{L2.8} to assume in addition that $C_{12}{}^1=0$.
We then obtain three soliton equations:
\begin{eqnarray*}
0&=&-a (1 + C_{11}{}^1) + C_{12}{}^2 + C_{11}{}^1 C_{12}{}^2 - (C_{12}{}^2)^2 + C_{11}{}^2 C_{22}{}^2,\\
0&=&-2C_{11}{}^2C_{22}{}^1+C_{22}{}^2,\\
0&=&C_{22}{}^1 (a-C_{11}{}^1+C_{12}{}^2+1)\,.
\end{eqnarray*}
The second equation implies $C_{22}{}^2=2C_{11}{}^2C_{22}{}^1$ and, since $C_{22}{}^1\ne0$, the third equation shows that
$C_{11}{}^1=a+C_{12}{}^2+1$. We rescale $x^2$ to assume $C_{22}{}^1=\pm1$ and
obtain the surface $\mathcal{P}_{a,c}^\pm$ where $c:=C_{11}{}^2$.

Next suppose that $C_{22}{}^1=0$. We obtain soliton equations
\begin{eqnarray*}
0&=&C_{12}{}^1 (-2 a+2 C_{12}{}^2-1)+C_{22}{}^2,\\
0&=&C_{12}{}^1(C_{22}{}^2-C_{12}{}^1)\,.
\end{eqnarray*}
If $C_{12}{}^1=0$, then we also obtain $C_{22}{}^2=0$. 
This is the case of Assertion~1 of Theorem~\ref{T4.9}.
We therefore assume $C_{12}{}^1\ne0$ and obtain
$C_{22}{}^2=C_{12}{}^1$. The soliton equation then implies $a=C_{12}{}^2$. A final
soliton equation then yields $C_{12}{}^2=0$. This implies $a=0$
contrary to our assumption. 
\end{proof}

Theorem~\ref{T4.9} classifies the geometries  of Type~$\mathcal{B}$ surfaces which
can admit a function resulting in an affine gradient Ricci soliton. Generically, either 
$\mathfrak{A}(\mathcal{M})$ is empty or it is an affine line. For example, Lemma~\ref{L4.6} 
shows that surfaces that are isomorphic to $\mathcal{P}_{0,c}^\pm$ for $c>0$ or to $\mathcal{Q}_c$
for arbitrary $c$ only admit constant functions as solutions of the Ricci soliton equation (see Statement~2 of
Definition~\ref{D1.5}). The next theorem shows that $\mathfrak{A}(\mathcal{M})$ is also 
an affine line for all $\mathcal{P}_{a,c}^\pm$ with $a\ne 0$ except in two cases: $a=-2$ and $a=-\frac12$.

\begin{theorem}\label{T4.10}
Let $\mathcal{M}\in\mathcal{F}^{\mathcal B}$ such that the
space $\mathfrak{A}(\mathcal{M})$ is neither empty nor an affine line. Then, up to linear
equivalence, one of
the following alternatives holds:
\begin{enumerate}
\item $\mathcal{M}$ is also Type~$\mathcal{A}$, i.e.,
$C_{12}{}^1=0$, $C_{22}{}^1=0$, and $C_{22}{}^2=0$.
\smallbreak\item $\mathcal{M}=\mathcal{P}_{-2,0}^\pm$ and
$\mathfrak{A}(M)=\{-2\log(x^1)+c_1x^2+c_0\}$
for $c_i\in\mathbb{R}$.
\item $\mathcal{M}=\mathcal{P}_{-\frac12,c}^-$, and
$\mathfrak{A}(\mathcal{M})=\{-\frac12\log(x^1)+c_1(x^2-2cx^1)+c_0\}$ 
for $c_1,c_0\in\mathbb{R}$ and $c^2=\frac38$.\end{enumerate}
\end{theorem}
\begin{proof}We examine the Hessian. 
The setting of Assertion~1 in Theorem~\ref{T4.10}  was examined previously in Lemma~\ref{L4.7}. We assume the setting of Assertion~2 of Theorem~\ref{T4.9}, i.e.
 that for $a\ne0$ and $c\ge0$, we have:
\begin{eqnarray*}
 &&\textstyle C_{11}{}^1=\frac{1}{2} \left(a^2+4 a\mp2c^2+2\right),\ C_{11}{}^2=c,\phantom{...\ } C_{12}{}^1=0,\\
&&\textstyle C_{12}{}^2=\frac{1}{2} \left(a^2+2 a\mp2c^2\right),\phantom{+2c}\ C_{22}{}^1=\pm1,\ C_{22}{}^2=\pm2c,,\\
&& \rho=(x^1)^{-2}\left(\begin{array}{rr}
 \frac a2(4+4a+a^2\mp2c^2)&
 \pm c\\\mp c&\pm a\end{array}\right).
\end{eqnarray*}
 We examine the kernel of the Hessian to determine
the most general solution. Let $h\in\ker(H^\nabla)$ with $dh\ne0$. If $h=h(x^1)$, then
$$
H_h^\nabla=(x^1)^{-1}\left(\begin{array}{rr}
\star(x^1)&0\\
0&\mp h^\prime(x^1)\end{array}\right)\,.
$$
This is not possible since $h^\prime\ne0$. Thus $h$ exhibits non-trivial $x^2$
dependence. We return to the general setting to obtain a relation. 
To simplify the
notation, we leave $C_{12}{}^2$ as a parameter and obtain:
$$
0=x^1 h^{(1,1)}-C_{12}{}^2 h^{(0,1)}\,.
$$
This implies $h(x^1,x^2)=(x^1)^{C_{12}{}^2}u(x^2)+v(x^1)$. We obtain:
\begin{eqnarray*}
0&=&\pm x^1 v'(x^1)-(x^1)^{C_{12}{}^2+2} u^{\prime\prime}(x^2)\pm2 C_{11}{}^2 (x^1)^{C_{12}{}^2+1} u^{\prime}(x^2)\\
&&\qquad\pm C_{12}{}^2 u(x^2) (x^1)^{C_{12}{}^2}\,.
\end{eqnarray*}
The powers of $x^1$ decouple. Because $h(x^1,x^2)$ exhibits non-trivial $x^2$
dependence, we may conclude that $C_{12}{}^2=0$ and hence $C_{11}{}^1=a+1$.
We also conclude $u^{\prime\prime}(x^2)$ must be constant. 
Let $h(x^1,x^2)=c_2\cdot(x^2)^2+c_1x^2+v(x^1)$. We obtain:
$$
0=\mp 2c_1c+2 c_2 (x^1\mp 2c x^2)\mp v'(x^1)\,.
$$
This ODE implies $v$ is quadratic in $x^1$ so
$h(x^1,x^2)=b_2\cdot(x^1)^2+b_1x^1+c_2\cdot(x^2)^2+c_1x^2$. 
We obtain an equation $b_1+ab_1+cc_1+2ab_2x_1+2cc_2x_2=0$.
Since $a\ne0$, $b_2=0$ so $h=c_2(x^2)^2+c_1x^2+b_1x^1$.
We obtain $2c_2x^1\mp b_1\mp 2cc_1\mp 4cc_2x^2=0$. This implies $c_2=0$
so $h=c_1x^2+b_1x^1$. The remaining equations become
$$b_1(1+a)+cc_1=0\text{ and }b_1+2cc_1=0\,.$$
Thus $b_1=-2c c_1$. We set $c_1=1$ to take $h=x^2-2cx^1$. 
This yields the final equation $c(2a+1)=0$. We require $C_{12}{}^2=\frac12(a^2+2a\mp2c^2)=0$.
We consider cases:
\subsubsection*{\bf Case 1. Suppose $c=0$.} Since $C_{12}{}^2=0$, $a^2+2a=0$ so since $a\ne0$,
we obtain $a=-2$. This yields the possibility of Assertion~2.
\subsubsection*{\bf Case 2. Suppose $a=-\frac12$} Since $C_{12}{}^2=0$, $-\frac34\mp2c^2=0$. Thus $C_{22}{}^1=-1$ and $c^2=\frac38$. This yields the possibility of Assertion~3.This completes the proof of Theorem~\ref{T4.10}.
\end{proof}

\begin{remark}\rm
The proof of Theorem~\ref{T4.10} is based on the study of the kernel of the Hessian on those surfaces that admit an affine gradient Ricci soliton. 
Thus, the given families result in examples of non-trivial Yamabe solitons of Type~$\mathcal{B}$, i.e. with nonconstant potential function:
\begin{enumerate}
\smallbreak\item If $C_{12}{}^1=0$, $C_{22}{}^1=0$, and $C_{22}{}^2=0$, then
$\mathcal{Y}(N)$ consists of the solutions $f=f(x^1,x^2)=\xi(x^1)$ to the ODE
$-x^1C_{11}{}^1\xi^\prime
+(x^1)^2\xi^{\prime\prime}=0$.
\smallbreak\item 
If $\mathcal{M}=\mathcal{P}_{-2,c}^\pm$ for $c=0$ or $\mathcal{M}=\mathcal{P}_{-\frac12,c}^-$
for $c^2=\frac38$, then $\mathcal{Y}(M)$ consists of the functions
$c_1(x^2-2cx^1)+c_2$ for $c_1,c_2\in\mathbb{R}\}$.
\end{enumerate}
\end{remark}

\subsection{The moduli space of homogeneous affine gradient Ricci solitons}

This section is devoted to the proof of the following result, which describes the moduli space of homogeneous affine gradient Ricci solitons.

\begin{theorem}\label{T4.12}
	Let $(\mathcal{M},\nabla, f)$ be a non-flat homogeneous affine gradient Ricci soliton. Then one of the following possibilities holds:
\begin{enumerate}
\item  $(\mathcal{M},\nabla)$ is isomorphic to $\mathcal{M}_4^0$ 
($\cong\mathcal{M}_1\cong\mathcal{M}_4^c$ for all $c\in \mathbb{R}$), and 
$f\in \mathfrak{A}(\mathcal{M})$ if and only if $f(x^1,x^2)\equiv f(x^2)$ with $f''-f+2=0$.
\smallbreak
\item $(\mathcal{M},\nabla)$ is isomorphic to $\mathcal{M}_3^c$ ($\cong\mathcal{M}_2^c$) 
with $c(1+c)\neq 0$, and $f\in \mathfrak{A}(\mathcal{M})$ if and only if 
$f(x^1,x^2)\equiv f(x^2)$ with $f''-(1+2c)f+c(1+c)=0$.
\smallbreak
\item $(\mathcal{M},\nabla)$ is isomorphic to $\mathcal{M}_5^c$ for all $c\in [0,\infty)$, and 
$f\in \mathfrak{A}(\mathcal{M})$ if and only if $f(x^1,x^2)\equiv f(x^2)$ with $f''-2c f+(1+c^2)=0$.
\smallbreak
\item $(\mathcal{M},\nabla)$ is isomorphic to $\mathcal{N}_2^{\frac12}$ ($\cong\mathcal{P}_{0,c}^-$
for $c=\frac3{\sqrt 2}$), and 
$f\in \mathfrak{A}(\mathcal{M})$ if and only if $f$ is constant.
\smallbreak
\item $(\mathcal{M},\nabla)$ is isomorphic to $\mathcal{Q}_{c}$ for all $c\in \mathbb{R}$, and 
$f\in \mathfrak{A}(\mathcal{M})$ if and only if $f$ is constant.
\smallbreak
\item $(\mathcal{M},\nabla)$ is isomorphic to $\mathcal{P}_{a,c}^\varepsilon$, where 
$\varepsilon=\pm1$, $(a,c)\neq (0,0)$, and 
\begin{enumerate}
\item if $\mathcal{P}_{a,c}^\varepsilon= \mathcal{P}_{0,c}^\varepsilon$, then 
$f\in \mathfrak{A}(\mathcal{M})$ if and only if $f$ is constant;
\smallbreak
\item if $\mathcal{P}_{a,c}^\varepsilon= \mathcal{P}_{-2,0}^\varepsilon$, then 
$f\in \mathfrak{A}(\mathcal{M})$ if and only if
\newline $f(x^1,x^2)=-2\log(x^1)+c_1 x^2+c_0$, for $c_0, c_1\in\mathbb{R}$; 
\smallbreak
\item if $\mathcal{P}_{a,c}^\varepsilon=\mathcal{P}_{-\frac{1}{2},c}^-$ with $c^2=\frac{3}{8}$, then
$f\in \mathfrak{A}(\mathcal{M})$ if and only if
\newline $f(x^1,x^2)=-\frac{1}{2}\log(x^1)+c_1(x^2-2c x^1)+c_0$, for $c_0, c_1\in\mathbb{R}$;
\smallbreak
\item if $\mathcal{P}_{a,c}^\varepsilon\neq\mathcal{P}_{0,c}^\varepsilon$, 
$\mathcal{P}_{a,c}^\varepsilon\neq\mathcal{P}_{-2,0}^\varepsilon$ and 
$\mathcal{P}_{a,c}^\varepsilon(a,c)\neq\mathcal{P}_{-\frac{1}{2},c}^-$ with $c^2=\frac{3}{8}$,  then 
$f\in \mathfrak{A}(\mathcal{M})$ if and only if 
$f(x^1,x^2)\equiv f(x^1)=a\log(x^1)+c_0$, for $c_0\in\mathbb{R}$. 
\end{enumerate}\end{enumerate}
The classes listed above represent distinct affine equivalence classes.
\end{theorem}
\begin{proof}
	Type $\mathcal{A}$ affine gradient Ricci solitons are characterized by Theorem \ref{T4.3} and Lemma \ref{L4.2}. Thus Assertions~1--3 follow from Lemma~\ref{L3.6} and Theorem~\ref{T3.8}.
	
	Type $\mathcal{B}$ affine gradient Ricci solitons are characterized in Theorem \ref{T4.9}
	and Theorem \ref{T4.10}, thus leading to Assertions~4--6.  The only Type~$\mathcal{B}$
	surface with skew-symmetric Ricci tensor and $\dim\{\mathfrak{K}( \mathcal{M})\}=3$  is
	$\mathcal{N}_2^\frac{1}{2}$, which is isomorphic to $\mathcal{P}_{0,c}^-$ for
	$c=\frac{3}{\sqrt{2}}$.
	
We complete the proof by showing that the affine structures given Assertions 1--6 are inequivalent.
By Theorem \ref{T3.8},  classes (1), (2) and (3) have 4-dimensional
Killing algebra $A_{4,9}^0$, $A_2\oplus A_2$ and $A_{4,12}$, respectively.
Hence these three classes are inequivalent. Adopt the notation of Definition~\ref{D2.4} 
to define the $\alpha$ invariant. We have $\alpha(\mathcal{M}_4^0)=16$, 
$\alpha(\mathcal{M}_3^c)=(c^2+c)^{-1}{4(1+2c)^2}$, and 
$\alpha(\mathcal{M}_5^c)=(1+c^2)^{-1}{16c^2}$. This shows that if $i=3,5$
 then $\mathcal{M}_i^c\cong \mathcal{M}_i^{\tilde c}$ if and only if  $c=\tilde c$.
By Theorem~\ref{T3.11}, $\mathfrak{K}(\mathcal{N}_2^{\frac12})=\mathfrak{su}(1,1)$
and hence $\mathcal{N}_2^{\frac12}$ is not affine isomorphic to  (1),  (2) or (3). 
Classes (5) and (6) are of Type $\mathcal{B}$ and have Killing algebra of dimension 2, which shows that they
are not isomorphic to any of the other classes. 

We now show the surfaces in Assertions~5 and 6 are inequivalent as well. 
Set:
\begin{eqnarray*}
		&&\rho_1:=\textstyle\frac1{x^1}
		\{\Gamma_{12}{}^2dx^1\otimes dx^1+\Gamma_{22}{}^2dx^1\otimes dx^2
		-\Gamma_{12}{}^1dx^2\otimes dx^1-\Gamma_{22}{}^1dx^2\otimes dx^2\},\nonumber\\
		&&\rho_2:=\Gamma_{ij}{}^k\Gamma_{kl}{}^ldx^i\otimes dx^j,\quad
		\rho_3:=\Gamma_{ik}{}^l\Gamma_{jl}{}^kdx^i\otimes dx^j,\quad
		\rho_0:=\Gamma_{ij}{}^jdx^i.
\end{eqnarray*}
By Theorem~\ref{T3.21}~1, the coordinate transformations of any Type~$\mathcal{B}$ surface
$\mathcal{M}$ with $\dim\{\mathfrak{K}(\mathcal{M})\}=2$ belong to the Lie group $\mathfrak{G}$
which is a subgroup of $\operatorname{GL}(2,\mathbb{R})$. Since contracting an upper against a lower
index is a $\operatorname{GL}(2,\mathbb{R})$ invariant, the tensors $\{\rho_0,\rho_2,\rho_3\}$
are invariantly defined on any such surface. Since we may express $\rho=\rho_1+\rho_2-\rho_3$, 
we conclude that $\rho_1$ is invariantly defined as well; $\rho_1$ is a $\mathfrak{G}$ invariant
but not a $\operatorname{GL}(2,\mathbb{R})$ invariant.
We note that $\rho_1$ is skew-symmetric for any surface $\mathcal{Q}_{\tilde{c}}$ and that
 $\rho_1(\partial_2,\partial_2)\neq 0$ for any surface $\mathcal{P}_{a,c}^\varepsilon$. Hence no surface in Assertion~5 may be equivalent to any surface in Assertion~6.
	
The invariant $\rho_2$ is a symmetric $(0,2)$-tensor field for $\mathcal{Q}_c$ which is given by
$$
\rho_2(\mathcal{Q}_c)=2(x^1)^{-2}( c\, dx^1\otimes dx^1+ dx^2\otimes dx^2)\,.
$$
It defines a pseudo-Riemannian metric of constant curvature $-c^{-1}$ if $c\neq 0$. 
This shows that $\mathcal{Q}_c\cong \mathcal{Q}_{\tilde{c}}$ if and only if $c=\tilde{c}$.
We apply Lemma~\ref{L3.9}. The pull-back action of $T_{b,c}$ rescales $\partial_2$:
$(T_{b,c})_*\partial_1=\partial_1+b\partial_2$ and
$(T_{b,c})_*\partial_2=c\partial_2$.	The surfaces $\mathcal{P}_{a,c}^{\varepsilon}$ satisfy 
$\rho_1(\partial_2,\partial_2)=-\epsilon(x^1)^{-2}$ and $\rho_0(\partial_2,\partial_2)=2c\epsilon(x^1)^{-1}$. 
Consequently, if $\mathcal{P}_{a,c}^\varepsilon$ is affine isomorphic to 
	$\mathcal{P}_{\tilde a,\tilde c}^{\tilde\varepsilon}$ then $\epsilon=\tilde \epsilon$ and $c=\tilde c$. 
We compute that
\begin{eqnarray*}
&&\rho_2(\partial_2,\partial_2)=(x^1)^{-2}{\epsilon (1 + 3 a + a^2 + 2 c^2\epsilon)},\\
&&\rho_3(\partial_2,\partial_2)=(x^1)^{-2}{\epsilon ( 2 a + a^2 + 2 c^2\epsilon)}\,.
\end{eqnarray*}
This implies that $a=\tilde a$ which completes the proof.
\end{proof}

\subsection{Geodesic completeness} We have the following application of our analysis.
\begin{lemma} Let $\mathcal{M}$ be a locally homogeneous surface of Type~$\mathcal{A}$ which
is not symmetric and with $\operatorname{Rank}\{\rho\}=1$. Then $\mathcal{M}$ is not geodesically
complete.
\end{lemma}

\begin{proof} The analysis of Section~\ref{S3} shows that in any Type~$\mathcal{A}$ chart $(x^1,x^2)$,
the affine Killing vector fields are real analytic. 
 If $(u^1,u^2)$ is another Type~$\mathcal{A}$ chart which
intersects the given one, then $\partial_1^u$ and $\partial_2^u$ are
  affine Killing vector fields and hence real
analytic. This implies that $\mathcal{M}$ is a real analytic surface with respect to an
atlas of Type~$\mathcal{A}$ charts and our analysis shows $\mathfrak{A}(\mathcal{M})$ consists
of real analytic functions on $\mathcal{M}$. We suppose $\operatorname{Rank}\{\rho\}=1$ and apply
Lemma~\ref{L2.3} to see $\Gamma_{11}{}^2=0$ and $\Gamma_{12}{}^2=0$. We have
$\nabla\rho=-2\Gamma_{22}{}^2dx^2\otimes \rho$. Since $\mathcal{M}$
is not symmetric, $\Gamma_{22}{}^2\ne0$, and we can further normalize the coordinates so $\Gamma_{22}{}^2=1$.
Let $\sigma(t):=(x^1(t),x^2(t))$ be a local geodesic. The geodesic equations become
$\ddot x^2(t)+\dot x^2(t)\dot x^2(t)=0$ which may be solved by setting $x^2(t)=\log(t)$ for $t\in(t_0,t_1)$
some appropriate positive interval.
By Lemma~\ref{L4.2}, $\xi(x^2)=\rho_{22}x^2\in\mathfrak{A}(\mathcal{M})$. Since $\mathcal{M}$
is simply connected, we can extend $\xi$ to a global element of $\mathfrak{A}(\mathcal{M})$ which
is real analytic. Furthermore, since $\mathcal{M}$ is geodesically complete, we can extend
$\sigma$ to a global real analytic geodesic. Since $\xi(\sigma(t))=\rho_{22} \log(t)$ for $t\in(t_0,t_1)$,
$\xi(\sigma(t))=\rho_{22} \log(t)$ for all $t\in\mathbb{R}$; this is not possible.
\end{proof}

\bigskip

\noindent
\textbf{In memory of the attacks in  Beirut and Paris in November 2015. \\Solidarit\'e.}


\begin{thebibliography}{ll}

\bibitem{BCGV16}
M. Brozos-V\'{a}zquez, E. Calvi\~{n}o-Louzao, E. Garc\'{i}a-R\'{i}o, and R. V\'{a}zquez-Lorenzo, ``Local structure of self-dual gradient Yamabe solitons". 
Geometry, Algebra and applications: From Mechanics to Crytography. In Honor of Jaime Mu\~noz Masqu\'e, Springer Proc. Math. Stat., to appear.

\bibitem{AMK08} T. Arias-Marco and O. Kowalski,
``Classification of locally homogeneous affine connections with arbitrary torsion on 2-manifolds",
{\it Monatsh. Math. \bf153} (2008), 1--18.

\bibitem{BG14} M. Brozos-V\'{a}zquez and E. Garc\'{i}a-R\'{i}o,
``Four-dimensional neutral signature self-dual gradient Ricci solitons", arXiv:1410.8654.

\bibitem{BGG16}
M. Brozos-V\'{a}zquez, E. Garc\'{i}a-R\'{i}o, and P. Gilkey,
``Homogeneous affine surfaces: Moduli spaces" (in preparation).

\bibitem{CGV10} 
E. Calvi\~{n}o-Louzao, E. Garc\'{i}a-R\'{i}o, and R. V\'{a}zquez-Lorenzo, 
``Riemann Extensions of Torsion-Free Connections with Degenerate Ricci Tensor",
{\it Canad. J. Math. \bf62} (2010), 1037--1057.

\bibitem{D}
A. Derdzinski,
``Noncompactnes and maximum mobility of Type III Ricci-flat self-dual neutral Walker four-manifolds",
\emph{Q. J. Math.} \textbf{62} (2011), 363--395.

\bibitem{G-RGN}
E. Garc\'\i a-R\'\i o, P. Gilkey, and S. Nik\v cevi\'c,
``Homothety Curvature Homogeneity and Homothety Homogeneity'',
\emph{Ann. Global Anal. Geom.} \textbf{48} (2015), 149--170.

\bibitem{GKVV99}
E. Garc\'\i a-R\'\i o, D. Kupeli, M. E. V\'azquez-Abal, R. V\'azquez-Lorenzo,
``Affine Osserman connections and their Riemannian extensions",
\emph{Differential Geom. Appl.} \textbf{11} (1999), 145--153.

\bibitem{G-SG}
A. Guillot and A. S\'anchez-Godinez,
``A classification of locally homogeneous affine connections on compact surfaces",
\emph{Ann. Global Anal. Geom.} \textbf{46} (2014), 335--349.

\bibitem{KN63} S. Kobayashi and K. Nomizu, ``Foundations
of Differential Geometry vol. I and \text{II}", {\it Wiley Classics Library}. 
A Wyley-Interscience Publication, John Wiley $\&$ Sons, Inc., 
New York, 1996.
 
 \bibitem{KVOp2}
O. Kowalski, B. Opozda, and Z. Vlasek, 
``A classification of locally homogeneous affine connections with skew-symmetric Ricci tensor on $2$-dimensional manifolds'', 
\emph{Monatsh. Math.} \textbf{130} (2000), 109--125.

\bibitem{KVOp}
O. Kowalski, B. Opozda, and Z. Vlasek,
``On locally nonhomogeneous pseudo-Riemannian manifolds with locally homogeneous Levi-Civita connections''. 
\emph{Internat. J. Math.} \textbf{14} (2003), 559--572. 

\bibitem{KV03}
O. Kowalski, Z. Vlasek,
``On the local moduli space of locally homogeneous affine connections in plane domains",
 \emph{Comment. Math. Univ. Carolina} \textbf{44} (2003), 229--234.

\bibitem{Op04} B. Opozda, ``A classification of locally homogeneous connections on 2-dimensional manifolds",
\emph{Differential Geom. Appl.} {\bf 21} (2004), 173--198.

\bibitem{Opozda}
B. Opozda,
``Locally homogeneous affine connections on compact surfaces",
\emph{Proc. Amer. Math. Soc.} \textbf{132} (2004), 2713--2721.

\bibitem{PSWZ76} J. Patera, R. T. Sharp, P. Winternitz, and H. Zassenhaus, 
``Invariants of real low dimension Lie algebras", {\it J. Mathematical Phys. \bf17} (1976), 86--994.

\end{thebibliography}
\end{document}